\documentclass[12pt]{amsart}
\usepackage{amscd,amssymb,graphics,color,a4wide,hyperref,verbatim}
\usepackage{graphicx}
\usepackage{psfrag} 
\usepackage{marvosym}

\usepackage{mathrsfs}
\input xy
\xyoption{all}

\DeclareFontFamily{U}{rsf}{}
\DeclareFontShape{U}{rsf}{m}{n}{
  <5> <6> rsfs5 <7> <8> <9> rsfs7 <10->  rsfs10}{}
\DeclareMathAlphabet{\mathscr}{U}{rsf}{m}{n}

\newtheorem{theorem}{Theorem}[section]
\newtheorem{lemma}[theorem]{Lemma}
\newtheorem{proposition}[theorem]{Proposition}

\newtheorem{xca}[theorem]{Exercise}
\newtheorem{conjecture}[theorem]{Conjecture}

\newtheorem{question}[theorem]{Question}
\theoremstyle{definition}
\newtheorem{definition}[theorem]{Definition}
\newtheorem{construction}[theorem]{Construction}
\newtheorem{example}[theorem]{Example}
\newtheorem{examples}[theorem]{Examples}

\theoremstyle{remark}

\numberwithin{equation}{section}

\newcommand {\Scatter} {\operatorname{{\mathsf{S}}}}
\newcommand {\fod}  {\mathfrak{d}}

\newcommand{\NN} {\mathbb{N}}
\newcommand{\ZZ} {\mathbb{Z}}
\newcommand{\QQ} {\mathbb{Q}}
\newcommand{\RR} {\mathbb{R}}
\newcommand{\VV} {\mathbb{V}}
\newcommand{\CC} {\mathbb{C}}

\newcommand{\PP} {\mathbb{P}}
\renewcommand{\AA} {\mathbb{A}}

\newcommand {\shL} {\mathcal{L}}
\newcommand {\shM} {\mathcal{M}}

\newcommand {\shO} {\mathcal{O}}

\newcommand {\shT} {\mathcal{T}}

\newcommand {\shP} {\mathcal{P}}

\newcommand {\shX} {\mathcal{X}}

\newcommand {\foD} {\mathfrak{D}}

\newcommand {\foM} {\mathfrak{M}}

\renewcommand {\Bar} {\operatorname{Bar}}

\newcommand {\coker} {\operatorname{coker}}

\newcommand {\dual} {{\vee}}

\newcommand {\gp} {{\operatorname{gp}}}

\newcommand {\Hom} {\operatorname{Hom}}

\newcommand {\id} {\operatorname{id}}

\renewcommand {\Im} {\operatorname{Im}}

\newcommand {\Int} {\operatorname{Int}}

\renewcommand {\ker } {\operatorname{ker}}

\newcommand {\M} {\mathcal{M}}

\renewcommand{\O} {\mathcal{O}}

\newcommand {\out} {\mathrm{out}}

\renewcommand{\P} {\mathscr{P}}

\newcommand {\rank} {\operatorname{rank}}

\newcommand {\Sing} {\operatorname{Sing}}

\newcommand {\Spec} {\operatorname{Spec}}

\newcommand {\Supp} {\operatorname{Supp}}

\newcommand {\Trop} {\mathrm{Trop}}

\newcommand{\bbfamily}{\fontencoding{U}\fontfamily{bbold}\selectfont}
\newcommand{\textbb}[1]{{\bbfamily#1}}
\newcommand {\lfor} {\mbox{\textbb{[}}}
\newcommand {\rfor} {\mbox{\textbb{]}}}
\newcommand {\T} {\shT}
\newcommand {\X} {\shX}

\newcommand {\cM} {\mathscr{M}}
\newcommand {\U} {\mathscr{U}}

\def\mydate{\ifcase\month \or January\or February\or March\or
April\or May\or June\or July\or August\or September\or October\or 
November\or December\fi \space\number\day,\space\number\year}

\newlength{\picwidth} 
\newlength{\miniwidth} 
\newlength{\nanowidth} 
\newlength{\melowidth} 
\newlength{\leftminiwidth} 
\newlength{\rightminiwidth} 
\newlength{\minipagewidth} 

\begin{document}
\def\mapright#1{\smash{
 \mathop{\longrightarrow}\limits^{#1}}}
\def\mapleft#1{\smash{
 \mathop{\longleftarrow}\limits^{#1}}}
\def\exact#1#2#3{0\to#1\to#2\to#3\to0}
\def\mapup#1{\Big\uparrow
  \rlap{$\vcenter{\hbox{$\scriptstyle#1$}}$}}
\def\mapdown#1{\Big\downarrow
  \rlap{$\vcenter{\hbox{$\scriptstyle#1$}}$}}
\def\dual#1{{#1}^{\scriptscriptstyle \vee}}
\def\invlim{\mathop{\rm lim}\limits_{\longleftarrow}}
\def\rto{\raise.5ex\hbox{$\scriptscriptstyle ---\!\!\!>$}}
\def\cy{\check y}
\input epsf.tex

\title[The Strominger-Yau-Zaslow conjecture]{
Mirror symmetry and the Strominger-Yau-Zaslow conjecture}

\author{Mark Gross} 
\address{UCSD Mathematics, 9500 Gilman Drive, La Jolla, CA 92093-0112, USA}
\email{mgross@math.ucsd.edu}
\thanks{This work was partially supported by NSF grant
1105871}

\date{}
\begin{abstract}
We trace progress and thinking about the Strominger-Yau-Zaslow conjecture
since its introduction in 1996. We begin with the original differential
geometric conjecture and its refinements, and explain how insights gained
in this context led to the algebro-geometric program developed by the author
and Siebert. The objective of this program is
to explain mirror symmetry by studying degenerations
of Calabi-Yau manifolds. This introduces
logarithmic and tropical geometry into the mirror symmetry story,
and gives a clear path towards a conceptual understanding of mirror
symmetry within an algebro-geometric context. After explaining the overall
philosophy, we explain how recent results fit into
this program.
\end{abstract}
\maketitle
\bigskip

\section*{Introduction.}

Mirror symmetry got its start in 1989 with work of Greene and Plesser
\cite{GrPl} and Candelas, Lynker and Schimmrigk \cite{CLS}. These
two works first observed the existence of pairs of Calabi-Yau
manifolds exhibiting an exchange of Hodge numbers. Recall that by Yau's
proof of the Calabi conjecture \cite{Yau}, a Calabi-Yau manifold
is an $n$-dimensional complex manifold $X$ with a nowhere vanishing 
holomorphic $n$-form
$\Omega$ and a Ricci-flat K\"ahler metric with K\"ahler form $\omega$.
Ricci-flatness is equivalent to $\omega^n=C \Omega\wedge\bar\Omega$
for a constant $C$.

The most famous example of a Calabi-Yau manifold is a smooth quintic
three-fold $X\subseteq\PP^4$. 
The Hodge numbers of $X$ are $h^{1,1}(X)=1$ and
$h^{1,2}(X)=101$, with topological Euler characteristic $-200$. 
The original construction of Greene and Plesser gave a mirror to
$X$, as follows. 
Let $Y\subseteq \PP^4$ be given by the equation
\[
x_0^5+\cdots+x_4^5=0,
\]
and let 
\[
G=\{(a_0,\ldots,a_4)\in\ZZ_5^5\,|\, \sum_i a_i=0\}.
\]
An element $(a_0,\ldots,a_4)\in G$ acts on $Y$ by
\[
(x_0,\ldots,x_4)\mapsto (\xi^{a_0}x_0,\ldots,\xi^{a_4}x_4)
\]
for $\xi$ a primitive fifth root of unity. The quotient $Y/G$
is highly singular, but there is a resolution of singularities
$\check X\rightarrow Y/G$ such that $\check X$ is also Calabi-Yau, and
one finds $h^{1,1}(\check X)=101$ and $h^{1,2}(\check X)=1$, so that $\check
X$ has topological Euler characteristic $200$.

The relationship between these two Calabi-Yau manifolds proved to be much
deeper than just this exchange of Hodge numbers. 
Pioneering work of Candelas,
de la Ossa, Greene and Parkes \cite{COGP} performed an amazing calculation,
following string-theoretic predictions which suggested
that certain enumerative calculations on $X$ should give
the same answer as certain period calculations on $\check X$. The calculations
on $\check X$, though subtle, could be carried out: these involved integrals
of the holomorphic form on $\check X$ over three-cycles as the complex
structure on $\check X$ is varied. On the other hand, the corresponding
calculations on $X$ involved numbers of rational curves on $X$ of each
degree. For example, the number of lines on a generic quintic threefold
is $2875$ and the number of conics is $609250$. String theory thus gave
predictions for these numbers for every degree, an astonishing feat given
that most of these numbers seemed far beyond the reach of algebraic
geometry at the time.

More generally, string theory introduced the concepts of the \emph{$A$-model}
and \emph{$B$-model}. The $A$-model involves the symplectic geometry
of Calabi-Yau manifolds. Properly defined, the counts of rational curves
are in fact symplectic invariants, now known as Gromov-Witten invariants.
The $B$-model, on the other hand, involves the complex geometry of
Calabi-Yau manifolds. Holomorphic forms of course depend on the complex
structure, so the period calculations mentioned above can be thought of
as $B$-model calculations. Ultimately, string theory predicts an isomorphism
between the $A$-model of a Calabi-Yau manifold $X$ and the $B$-model
of its mirror, $\check X$. The equality of numerical invariants is then
a consequence of this isomorphism.

Proofs of these string-theoretic predictions of curve-counting
invariants were given
by Givental \cite{Givental} and Lian, Liu and Yau \cite{LLY}, with successively
simpler proofs by many other researchers. However, all the proofs relied
on the geometry of the ambient space $\PP^4$ in which the quintic is contained.
Roughly speaking, one considers all rational curves in $\PP^4$, and tries
to understand how to compute how many of these are contained in a given
quintic hypersurface.

This raised the question: \emph{is there some underlying intrinsic
geometry to mirror symmetry?}  

Historically the first approach to an intrinsic formulation
of mirror symmetry is Kontsevich's
Homological Mirror Symmetry conjecture, stated in 1994 in \cite{KHMS}. This
made mathematically precise the notion of an isomorphism between the
$A$- and $B$-models.
The homological mirror symmetry conjecture posits an isomorphism
between two categories, 
the Fukaya category of Lagrangian
submanifolds of $X$ (the $A$-model) and 
the derived category of coherent
sheaves on the mirror $\check X$ (the $B$-model).
Morally, this states that the 
symplectic geometry of $X$ is the same as the complex geometry of
$\check X$. At the time this conjecture was made, however,
there was no clear idea
as to how such an isomorphism might be realised, nor did this conjecture
state how to construct mirror pairs.

The second approach is due to Strominger, Yau
and Zaslow in their 1996 paper \cite{SYZ}.
They made a remarkable proposal, based on new ideas in string theory,
which gave a very concrete geometric interpretation for mirror
symmetry. 

Let me summarize, very roughly, the physical argument they used
here. Developments
in string theory in the mid-1990s introduced the notion of \emph{Dirichlet
branes}, or $D$-branes. These are submanifolds of space-time, with
some additional data, which
serve as boundary conditions for open strings, i.e., we allow
open strings to propagate with their endpoints constrained to lie on
a $D$-brane. Remembering that space-time, according to string theory, 
looks like
$\RR^{1,3}\times X$, where $\RR^{1,3}$ is ordinary space-time and
$X$ is a Calabi-Yau three-fold, we can split a $D$-brane into a product
of a submanifold of $\RR^{1,3}$ and one on $X$. It turned out, simplifying
a great deal, that there are 
two particular types of submanifolds on $X$ of interest: 
\emph{holomorphic} $D$-branes,
i.e., holomorphic submanifolds with a holomorphic line bundle, and 
\emph{special Lagrangian} $D$-branes, which are 
\emph{special Lagrangian submanifolds} with 
flat $U(1)$-bundle: 

\begin{definition} Let $X$ be an $n$-dimensional
Calabi-Yau manifold with $\omega$
the K\"ahler form of a Ricci-flat metric on $X$ and $\Omega$
a nowhere vanishing holomorphic $n$-form. Then a submanifold $M\subseteq
X$ is \emph{special Lagrangian} if it is Lagrangian, i.e., 
$\dim_{\RR} M=\dim_{\CC} X$
and $\omega|_M=0$, and in addition, $\Im\Omega|_M=0$.
\end{definition}

Holomorphic $D$-branes can be viewed as $B$-model objects, and special
Lagrangian $D$-branes as $A$-model objects. The isomorphism between 
the $B$-model on $X$ and the $A$-model on $\check X$ then suggests
that the moduli space of holomorphic $D$-branes on $X$
should be isomorphic to the moduli space of special Lagrangian $D$-branes 
on $\check X$. (This is now seen as a physical manifestation of the
homological mirror symmetry conjecture).
Now $X$ itself is the moduli space of points on $X$. So each
point on $X$ should correspond to a pair $(M,\nabla)$, where $M\subseteq\check
X$ is a special Lagrangian submanifold and $\nabla$ is a flat 
$U(1)$-connection on $M$. 

A theorem of McLean \cite{McLean}
tells us that the tangent space to the moduli space of
special Lagrangian deformations of a special Lagrangian submanifold $M\subseteq
\check X$ is $H^1(M,\RR)$. Of course, 
the moduli space of flat $U(1)$-connections
modulo gauge equivalence on $M$ is the torus $H^1(M,\RR)/H^1(M,\ZZ)$.
In order for this moduli space to be of the correct dimension, we need
$\dim H^1(M,\RR)=n$, the complex dimension of $X$. This suggests that 
$X$ consists of a family of tori which are dual to a family of special 
Lagrangian tori on $\check X$. An elaboration of this argument yields the
following conjecture:

\begin{conjecture}
\emph{The Strominger-Yau-Zaslow conjecture}. If $X$ and $\check X$ are
a mirror pair of Calabi-Yau $n$-folds, then there exists fibrations
$f:X\rightarrow B$ and $\check f:\check X\rightarrow B$ whose fibres
are special Lagrangian, with general fibre an $n$-torus. Furthermore,
these fibrations are dual, in the sense that canonically
$X_b=H^1(\check X_b,\RR/\ZZ)$
and $\check X_b=H^1(X_b,\RR/\ZZ)$ whenever $X_b$ and $\check X_b$ 
are non-singular tori.
\end{conjecture}

This conjecture motivated a great deal of work in the five years following
its introduction in 1996, 
some of which will be summarized in the following sections. There was
a certain amount of success, as we shall see, with the conjecture proved
for some cases, including the quintic three-fold, at the topological level.
Further, the conjecture gave a solid framework for thinking about
mirror symmetry at an intuitive level. However, work of Dominic Joyce 
demonstrated that the conjecture was unlikely to be literally true. 
Nevertheless, it is possible that weaker limiting forms of the conjecture
still hold.

In the first several sections of this survey,
I will clarify the conjecture, review what is known about it, and state
a weaker form which seems accessible. Most importantly, I will explain
how the SYZ conjecture leads to the study of affine manifolds (manifolds
with transition functions being affine linear) and hence to an algebro-geometric
interpretation of the conjecture, developed by me and Bernd Siebert.
This removes the hard analysis, and gives a powerful framework for understanding
mirror symmetry at a conceptual level. 

The bulk of the paper is devoted to outlining this framework as developed
over the last ten years. I explain how affine manifolds are related
to degenerations of Calabi-Yau manifolds. Once one begins to consider
degenerations, log geometry of K.\ Kato and Fontaine--Illusie comes
into the picture. Conjecturally, the base of the SYZ fibration
incorporates key combinatorial information about log structures on degenerations
of Calabi-Yau manifolds. Log geometry then gives a connection with tropical
geometry and log Gromov-Witten theory, which theoretically allows a description
of $A$-model curve counting using tropical geometry. On the mirror side,
we explain how again tropical geometry is used to describe complex structures.
This identifies tropical geometry as the geometry underlying both sides
of mirror symmetry, and guides us towards a conceptual understanding
of mirror symmetry. We end with a description of recent work with
Pandharipande and Siebert \cite{GPS}  
which provides a snapshot of the relationship between the two sides of
mirror symmetry.

\medskip

I would like to thank the organizers of Current Developments in Mathematics
2012 for inviting me to take part in the conference, and 
Bernd Siebert, my collaborator on much of the work described here. 
Some of the material appearing in this article
was first published in my article 
``The Strominger-Yau-Zaslow conjecture: From
 torus fibrations to degenerations," in \emph{Algebraic Geometry: Seattle 2005},
edited 
by D. Abramovich, et al., Proceedings of Symposia in Pure Mathematics 
Vol. 80, part 1, 149-192, published by the American Mathematical Society. (c) 
2009 by the American Mathematical Society. Finally, I would like to thank
Lori Lejeune and the Clay Institute for Figure \ref{vanishingdisk}.

\bigskip

\section{Moduli of special Lagrangian submanifolds}
\label{modulisection}

The first step in understanding the SYZ conjecture is to 
examine the structures which arise on the base of a special Lagrangian
fibration. These structures arise from McLean's theorem on the moduli
space of special Lagrangian submanifolds \cite{McLean}, and these
structures and their relationships were explained by Hitchin in \cite{Hit}. 
We outline some of these ideas here.
McLean's theorem says that the moduli space of deformations of a compact
special Lagrangian submanifold
of a compact 
Calabi-Yau manifold $X$ is unobstructed. Further, the
tangent space at the point of
moduli space corresponding to a special Lagrangian $M\subseteq
X$ is canonically isomorphic to
the space of harmonic $1$-forms on $M$. This isomorphism is seen
explicitly as follows. Let $\nu\in\Gamma(M,N_{M/X})$ be a normal
vector field to $M$ in $X$. Then the restriction of the contractions
$(\iota(\nu)\omega)|_M$ and
$(\iota(\nu)\Im\Omega)|_M$ are both seen to be well-defined forms
on $M$: one needs to lift $\nu$ to a vector field but the choice is
irrelevant because $\omega$ and $\Im\Omega$ restrict to zero on $M$.
McLean shows that if $M$ is special Lagrangian then
\[
\iota(\nu)\Im\Omega=-*\iota(\nu)\omega,
\]
where $*$ denotes the Hodge star operator on $M$. Furthermore,
$\nu$ corresponds to an infinitesimal deformation preserving the
special Lagrangian condition if and only if $d(\iota(\nu)\omega)
=d(\iota(\nu)\Im\Omega)=0$. This gives the correspondence between
harmonic $1$-forms and infinitesimal special Lagrangian deformations.

Let $f:X\rightarrow B$ be a special Lagrangian fibration
with torus fibres, and assume for now that all fibres of $f$ are non-singular.
Then we obtain three structures on $B$, namely 
two affine structures and a metric, as we shall now see.

\begin{definition} 
\label{affine}
Let $B$ be an $n$-dimensional manifold.
An {\it affine structure} on $B$ is given by an atlas $\{(U_i,\psi_i)\}$
of coordinate charts $\psi_i:U_i\rightarrow \RR^n$,
whose transition functions $\psi_i\circ\psi_j^{-1}$ lie in ${\rm Aff}(\RR^n)$.
We say the affine structure is \emph{tropical} if the transition functions
lie in $\RR^n\rtimes GL(\ZZ^n)$, i.e., have integral linear part. We say
the affine structure is {\it integral} if the transition functions
lie in ${\rm Aff}(\ZZ^n)$. 

If an affine manifold $B$ carries a Riemannian metric $g$, then we say
the metric is \emph{affine K\"ahler} or \emph{Hessian} if $g$ is locally
given by $g_{ij}=\partial^2K/\partial y_i\partial y_j$ for some convex
function $K$ and $y_1,\ldots,y_n$ affine coordinates.
\end{definition}

We obtain the three structures as follows:

\emph{Affine structure 1.} For a normal vector field $\nu$ to a fibre $X_b$
of $f$, $(\iota(\nu)\omega)|_{X_b}$ is a well-defined $1$-form on $X_b$,
and we can compute its periods as follows.
Let $U\subseteq B$ be a small open set,
and suppose we have submanifolds $\gamma_1,\ldots,\gamma_n\subseteq
f^{-1}(U)$ which are families of 1-cycles over $U$ and such that
$\gamma_1\cap X_b,\ldots,\gamma_n\cap X_b$ 
form a basis for $H_1(X_b,\ZZ)$ for each
$b\in U$. Consider the $1$-forms $\omega_1,\ldots,\omega_n$ on $U$
defined by fibrewise integration:
\[
\omega_i(\nu)=\int_{X_b\cap\gamma_i} \iota(\nu)\omega,
\]
for $\nu$ a tangent vector on $B$ at $b$, which we can lift
to a normal vector field of $X_b$. We have $\omega_i=f_*(\omega|_{\gamma_i})$,
and since $\omega$ is closed, so is $\omega_i$. Thus there are locally
defined functions $y_1,\ldots,y_n$ on $U$ with $dy_i=\omega_i$. 
Furthermore, these functions are well-defined up to the choice of basis
of $H_1(X_b,\ZZ)$ and constants. Finally, they give well-defined coordinates,
as follows from the fact that 
$\nu\mapsto \iota(\nu)\omega$ yields an isomorphism of $\T_{B,b}$ with
$H^1(X_b,\RR)$ by McLean's theorem. Thus $y_1,\ldots,y_n$ define local
coordinates of a tropical affine structure on $B$.

\emph{Affine structure 2.} We can play the same trick with $\Im\Omega$: 
choose submanifolds 
\[
\Gamma_1,\ldots,\Gamma_n\subseteq f^{-1}(U)
\]
which are families of $n-1$-cycles over $U$ and such that
$\Gamma_1\cap X_b,\ldots,\Gamma_n\cap X_b$ form a basis for $H_{n-1}(X_b,
\ZZ)$. We define $\lambda_i$ by $\lambda_i=-f_*(\Im\Omega|_{\Gamma_i})$,
or equivalently,
\[
\lambda_i(\nu)=-\int_{X_b\cap\Gamma_i} \iota(\nu)\Im\Omega.
\]
Again $\lambda_1,\ldots,\lambda_n$ are closed $1$-forms, with
$\lambda_i=d\check y_i$ locally, and again $\check y_1,\ldots,\check
y_n$ are affine coordinates for a tropical affine structure on $B$.

\emph{The McLean metric.} The Hodge metric on $H^1(X_b,\RR)$ is given
by 
\[
g(\alpha,\beta)=\int_{X_b} \alpha\wedge *\beta
\]
for $\alpha$, $\beta$ harmonic $1$-forms, and hence induces a
metric on $B$, which can be written as 
\[
g(\nu_1,\nu_2)=-\int_{X_b}\iota(\nu_1)\omega\wedge \iota(\nu_2)\Im\Omega.
\]

\medskip

A crucial observation of Hitchin \cite{Hit} is that these structures are
related by the Legendre transform:

\begin{proposition}
\label{hessianmetric}
Let $y_1,\ldots,y_n$ be local affine coordinates
on $B$ with respect to the affine structure induced by $\omega$. 
Then locally there is a function $K$ on $B$
such that
\[
g(\partial/\partial y_i,\partial/\partial y_j)=\partial^2 K/\partial y_i
\partial y_j.
\]
Furthermore, $\cy_i=\partial K/\partial y_i$
form a system of affine coordinates with respect to the affine
structure induced by $\Im\Omega$, and if
\[
\check K(\cy_1,\ldots,\cy_n)=\sum \cy_i y_i-K(y_1,\ldots,y_n)
\]
is the Legendre transform of $K$, then
\[
y_i=\partial \check K/\partial\cy_i
\]
 and
\[
\partial^2\check K/\partial \cy_i\partial \cy_j=g(\partial/\partial\cy_i,
\partial/\partial\cy_j).
\]
\end{proposition}

\begin{proof}
Take families $\gamma_1,\ldots,\gamma_n,\Gamma_1,\ldots,\Gamma_n$
as above
over an open neighbourhood $U$ with the two bases being
Poincar\'e dual, i.e., $(\gamma_i\cap X_b)\cdot(\Gamma_j\cap X_b)=
\delta_{ij}$ for $b\in U$.
Let $\gamma_1^*,\ldots,\gamma_n^*$ and $\Gamma_1^*,\ldots,\Gamma_n^*$
be the dual bases for $\Gamma(U,R^1f_*\ZZ)$ and $\Gamma(U,R^{n-1}f_*\ZZ)$
respectively. From the choice of $\gamma_i$'s, we get local coordinates
$y_1,\ldots,y_n$ with $dy_i=\omega_i$, so in particular
\[
\delta_{ij}=\omega_i(\partial/\partial y_j)=\int_{\gamma_i\cap X_b}
\iota(\partial/\partial y_j)\omega,
\]
hence $\iota(\partial/\partial y_j)\omega$ defines the cohomology
class $\gamma_j^*$ in $H^1(X_b,\RR)$. Similarly,
let
\[
g_{ij}=-\int_{\Gamma_i\cap X_b}\iota(\partial/\partial y_j)\Im
\Omega;
\]
then $-\iota(\partial/\partial y_j)\Im\Omega$ defines the cohomology
class $\sum_i g_{ij}\Gamma_i^*$ in $H^{n-1}(X_b,\RR)$, and
$\lambda_i=\sum_j g_{ij}dy_j$.
Thus
\begin{eqnarray*}
g(\partial/\partial y_j,\partial/\partial y_k)&=&
-\int_{X_b}\iota(\partial/\partial y_j)\omega
\wedge \iota(\partial/\partial y_k)\Im\Omega\\
&=&g_{jk}.
\end{eqnarray*}
On the other hand, let $\cy_1,\ldots,\cy_n$ be coordinates with
$d\cy_i=\lambda_i$. Then
\[
{\partial\cy_i/\partial y_j}=g_{ij}=g_{ji}={\partial\cy_j/
\partial y_i},
\]
so $\sum\cy_i dy_i$ is a closed 1-form. Thus there exists
locally a function $K$ such that $\partial K/\partial y_i=\cy_i$ and
$\partial^2 K/\partial y_i\partial y_j=g(\partial/\partial y_i,
\partial/\partial y_j)$.
A simple calculation then confirms that $\partial\check K/\partial \cy_i
=y_i$. On the other hand,
\begin{eqnarray*}
g(\partial/\partial\cy_i,\partial/\partial\cy_j)&=&
g\left(\sum_k {\partial y_k\over\partial\cy_i}{\partial\over
\partial y_k},\sum_\ell {\partial y_\ell\over\partial\cy_j}
{\partial\over\partial y_\ell}\right)\\
&=&\sum_{k,\ell}{\partial y_k\over\partial\cy_i}{\partial y_\ell\over\partial
\cy_j} g(\partial/\partial y_k,\partial/\partial y_\ell)\\
&=&\sum_{k,\ell} {\partial y_k\over\partial\cy_i}{\partial y_\ell\over
\partial\cy_j}{\partial\cy_k\over\partial y_\ell}\\
&=&{\partial y_j\over\partial\cy_i}={\partial^2\check K\over
\partial\cy_i\partial\cy_j}.
\end{eqnarray*}
\end{proof}

Thus we introduce the notion of the \emph{Legendre transform} of an
affine manifold with a multi-valued convex function.

\begin{definition}
\label{multivaluedconvex}
Let $B$ be an affine manifold. A \emph{multi-valued} function $K$
on $B$ is a collection of functions on an open cover $\{(U_i,K_i)\}$
such that on $U_i\cap U_j$, $K_i-K_j$ is affine linear. We say $K$
is \emph{convex} if the Hessian
$(\partial^2 K_i/\partial y_j\partial y_k)$ is positive definite for
all $i$, in any, or equivalently all, affine coordinate systems $y_1,
\ldots,y_n$.

Given a pair $(B,K)$ of affine manifold and convex multi-valued
function, the \emph{Legendre transform} of $(B,K)$ is a pair $(\check B,
\check K)$ where $\check B$ is an affine structure on the underlying
manifold of $B$ with coordinates given locally
by $\check y_i=\partial K/\partial y_i$, and $\check K$ is defined
by
\[
\check K_i(\check y_1,\ldots,\check y_n)=\sum \check y_j y_j
-K_i(y_1,\ldots,y_n).
\]
\end{definition}

\begin{xca}
Check that $\check K$ is also convex, and that the Legendre transform
of $(\check B,\check K)$ is $(B,K)$.
\end{xca}

Curiously, this Legendre transform between affine manifolds
with Hessian metric seems to have first appeared
in a work in statistics predating mirror symmetry, see \cite{Amari}.

\section{Semi-flat mirror symmetry}
\label{semiflatsection}

Let's forget about special Lagrangian fibrations for the
moment. Instead, we will look at how the structures found on $B$ in the
previous section give a toy
version of mirror symmetry.

\begin{definition} Let $B$ be a tropical affine manifold.
\begin{enumerate}
\item
Denote by $\Lambda\subseteq\T_B$  the local system of lattices
generated locally by $\partial/\partial y_1,\ldots,\partial/\partial y_n$, where
$y_1,\ldots,y_n$ are local affine coordinates. This is well-defined because
transition maps are in $\RR^n\rtimes GL_n(\ZZ)$. Set 
\[
X(B):=\T_B/\Lambda.
\]
This is a torus bundle over $B$. In addition, $X(B)$ carries a complex
structure defined locally as follows. Let $U\subseteq B$ be
an open set with affine coordinates $y_1,\ldots,y_n$, so $\T_U$
has coordinate functions $y_1,\ldots,y_n$, $x_1=dy_1,\ldots,x_n=dy_n$.
Then 
\[
q_j=e^{2\pi i(x_j+iy_j)}
\]
gives a system of holomorphic coordinates on $T_U/\Lambda|_U$, and
the induced complex structure is 
independent of the choice of affine coordinates. This is called the
\emph{semi-flat} complex structure on $X(B)$.

Later we will need a variant of this: for $\epsilon>0$, set
\[
X_{\epsilon}(B):=\T_B/\epsilon\Lambda.
\]
This has a complex structure with coordinates given by
\[
q_j=e^{2\pi i(x_j+iy_j)/\epsilon}.
\]
(As we shall see later, the limit $\epsilon\rightarrow 0$ 
corresponds to a ``large complex structure limit.'')
\item
Define $\check\Lambda\subseteq\T^*_B$ to be the local system of
lattices generated locally by $dy_1,\ldots,dy_n$, with $y_1,\ldots,y_n$
local affine coordinates. Set
\[
\check X(B):=\T^*_B/\check\Lambda.
\]
Of course $\T^*_B$ carries a canonical symplectic structure, and this
symplectic structure descends to $\check X(B)$.
\end{enumerate}
\qed
\end{definition}

We write $f:X(B)\rightarrow B$ and $\check f:\check X(B)\rightarrow B$
for these torus fibrations; these are clearly dual.

Now suppose in addition we have a Hessian metric $g$ on $B$,
with local potential function $K$. Then the following propositions
show that in fact
both $X(B)$ and $\check X(B)$ become K\"ahler manifolds. 

\begin{proposition} $K\circ f$ is a (local) K\"ahler potential on
$X(B)$, defining a K\"ahler form $\omega=2i\partial\bar\partial(K\circ
f)$. This metric is
Ricci-flat if and only if $K$ satisfies
the real Monge-Amp\`ere equation
\[
\det {\partial^2 K\over \partial y_i\partial y_j}=constant.
\]
\end{proposition}

\begin{proof}
Working locally with affine coordinates $(y_j)$ and 
complex coordinates 
\[
z_j={1\over 2\pi i}\log q_j=x_j+i y_j,
\]
we compute $\omega=2i\partial\bar\partial(K\circ f)={i\over 2}
\sum {\partial^2 K\over \partial y_j\partial y_k} dz_j\wedge
d\bar z_k$
which is clearly positive. Furthermore,
if $\Omega=dz_1\wedge\cdots\wedge dz_n$, then
$\omega^n$ is proportional to $\Omega\wedge\bar\Omega$ if and only if
$\det (\partial^2 K/\partial y_j\partial y_k)$ is constant.
\end{proof}

We write this K\"ahler manifold as $X(B,K)$.

Dually we have

\begin{proposition}
In local canonical coordinates $y_i,\check x_i$ on $\T^*_B$, the 
complex coordinate functions
$z_j=\check x_j+i\partial K/\partial y_j$ on $\T^*_B$ induce a well-defined
complex structure on $\check X(B)$, with respect to which the canonical
symplectic form $\omega$ is the K\"ahler form of a metric. Furthermore
this metric is Ricci-flat if and only if $K$ satisfies the real
Monge-Amp\`ere equation
\[
\det {\partial^2 K\over \partial y_j\partial y_k}=constant.
\]
\end{proposition}

\begin{proof}
It is easy to see that an affine linear change in the coordinates
$y_j$ (and hence an appropriate change in the coordinates $\check x_j$)
results in a linear change of the coordinates $z_j$, so they induce
a well-defined complex structure invariant under $\check x_j\mapsto \check x_j+1$, and hence
a complex structure on $\check X(B)$. Then one computes that
\[
\omega=\sum d\check x_j\wedge dy_j={i\over 2}\sum g^{jk} dz_j\wedge d\bar z_k
\]
where $g_{ij}=\partial^2 K/\partial y_j\partial y_k$. Then the metric
is Ricci-flat if and only if $\det(g^{jk})=constant$, if and only if
$\det(g_{jk})=constant$.
\end{proof}

As before, we call this K\"ahler manifold $\check X(B,K)$.

This motivates the definition

\begin{definition} An affine manifold with metric of Hessian form
is a \emph{Monge-Amp\`ere manifold} if the local potential function $K$
satisfies the Monge-Amp\`ere equation $\det(\partial^2K/\partial y_i\partial
y_j)=constant$.
\end{definition}

Hessian and Monge-Amp\`ere manifolds were first studied by
Cheng and Yau in \cite{ChengYau}.

\begin{xca}
\label{caniso}
Show that the identification of $\T_B$ and $\T^*_B$ given by a Hessian metric
induces a canonical isomorphism $X(B,K)\cong\check X(\check B,\check K)$
of K\"ahler manifolds, where $(\check B,\check K)$ is the Legendre transform
of $(B,K)$.
\end{xca}

There is a key extra parameter which appears in mirror symmetry known as
the $B$-field. This appears as a field in the non-linear sigma model
with Calabi-Yau target space, and is required mathematically to make
sense of mirror symmetry. Mirror
symmetry roughly posits an isomorphism between the complex moduli space of a
Calabi-Yau manifold $X$ and the K\"ahler moduli space of 
$\check X$. If one interprets the K\"ahler moduli space to mean the space
of all Ricci-flat K\"ahler forms on $\check X$, then one obtains only
a real manifold as moduli space, and one needs a complex manifold to
match up with the complex moduli space of $X$. The $B$-field is interpreted
as an element ${\bf B}\in H^2(\check X,\RR/\ZZ)$, and one views
${\bf B}+i\omega$ as a complexified K\"ahler class on $\check X$
for $\omega$ a K\"ahler class on $\check X$.

In the context of our toy version of mirror symmetry, 
we view the $B$-field as an element ${\bf B}\in
H^1(B,\Lambda_{\RR}/\Lambda)$, where $\Lambda_{\RR}=\Lambda\otimes_{\ZZ}
\RR$. This does not quite agree with the above definition of the $B$-field,
as this group does not necessarily coincide with $H^2(\check X,\RR/\ZZ)$. 
However,
in many important cases, such as for simply connected
Calabi-Yau threefolds with torsion-free
integral cohomology, these two groups do coincide. More generally, including
the case of K3 surfaces and abelian varieties, one would
need to pass to generalized
complex structures \cite{HitGen}, \cite{Gual}, \cite{Oren}, \cite{Clay}, 
\cite{Huy}, which we do not wish to do here.

Noting that a section of $\Lambda_{\RR}/\Lambda$ over an open
set $U$ can be viewed as a section of $\T_U/\Lambda|_U$, such a section
acts on $\T_U/\Lambda|_U$ via translation, and this action is in fact
holomorphic with respect to the semi-flat complex structure.
Thus a \v Cech 1-cocycle $(U_{ij},\beta_{ij})$
representing ${\bf B}$ allows us to reglue
$X(B)$ via translations over the intersections $U_{ij}$. This is done
by identifying the open subsets $f^{-1}(U_{ij})\subseteq f^{-1}(U_i)$
and $f^{-1}(U_{ij})\subseteq f^{-1}(U_j)$ via the automorphism
of $f^{-1}(U_{ij})$ given by translation by the section $\beta_{ij}$.
This gives a new
complex manifold $X(B,{\bf B})$. If in addition there is a multi-valued
potential function $K$ defining a metric, these translations preserve
the metric and yield a K\"ahler manifold $X(B,{\bf B},K)$.

\medskip

Thus the full toy version of mirror symmetry is as follows:

\begin{construction}[The toy mirror symmetry construction]
Suppose given an affine manifold $B$ with potential $K$ and $B$-fields ${\bf B}
\in H^1(B,\Lambda_{\RR}/\Lambda)$, $\check {\bf B}\in H^1(B,
\check\Lambda_{\RR}/\check\Lambda)$. It is not difficult to see, and
you will have seen this already if you've done Exercise \ref{caniso}, that
the local system $\check\Lambda$ defined using the affine structure on $B$
is the same as the local system $\Lambda$ defined using the affine
stucture on $\check B$. So we say
the pair
\[
(X(B,{\bf B},K),\check{\bf B})
\]
is mirror to
\[
(X(\check B,\check {\bf B},\check K),\bf B).
\]
\end{construction}

This provides a reasonably fulfilling picture of mirror symmetry
in a simple context. Many more aspects of mirror symmetry can be
worked out in this semi-flat context, see \cite{Leung} and \cite{Clay},
Chapter 6. This semi-flat case is an ideal testing ground for concepts
in mirror symmetry. However,
ultimately this only sheds limited insight into the general case.
The only compact Calabi-Yau manifolds with semi-flat Ricci-flat metric
which arise in this way are complex tori (shown by Cheng and Yau
in \cite{ChengYau}).
To deal with more interesting cases, we need to allow singular fibres,
and hence, singularities in the affine structure of $B$. The existence
of singular fibres are fundamental for the most interesting aspects of
mirror symmetry.

\section{Affine manifolds with singularities}
\label{affsection}

To deal with singular fibres, we define

\begin{definition}
A \emph{(tropical, integral) affine manifold with singularities}
is a $(C^0)$ manifold $B$ with an open subset $B_0\subseteq B$ which carries
a (tropical, integral) affine structure, and such that $\Gamma:=B
\setminus B_0$ is a locally finite union of locally closed submanifolds
of codimension $\ge 2$. 
\end{definition}

Here we will give a relatively simple construction of such affine manifolds
with singularities; a broader class of examples is given in 
\cite{GBB}; see also \cite{HZ} and \cite{HZ3}.

Let $\Delta$ be a reflexive polytope in $M_{\RR}=M\otimes_{\ZZ}\RR$,
where $M=\ZZ^n$. This means that $\Delta$ is a lattice polytope with
a unique interior integral point $0\in\Delta$, and the polar dual
polytope
\[
\nabla:=\{n\in N_{\RR}|\hbox{$\langle m,n\rangle\ge -1$ for all $m\in\Delta$}
\}
\]
is also a lattice polytope.  

Let $B=\partial\Delta$, and let $\P$ be a decomposition of $B$ into
lattice polytopes, i.e., $\P$ is a set of lattice polytopes contained
in $B$ such that (1) $B=\bigcup_{\sigma\in\P} \sigma$; (2) $\sigma_1,\sigma_2
\in \P$ implies $\sigma_1\cap\sigma_2$ lies in $\P$ and is a face
of both $\sigma_1$ and $\sigma_2$; (3) if $\sigma\in\P$, any face of
$\sigma$ lies in $\P$.

We now define a structure of integral affine manifold with singularities
on $B$, with discriminant locus $\Gamma\subseteq B$ defined as follows.
Let $\Bar(\P)$ denote the first barycentric subdivision of
$\P$ and let $\Gamma\subseteq B$ be the union of all simplices of
$\Bar(\P)$ not containing a vertex of $\P$ (a zero-dimensional cell)
or intersecting the interior of a maximal cell of $\P$. Setting
$B_0:=B\setminus\Gamma$, we define an affine structure on $B_0$
as follows. 
$B_0$ has an open cover
\[
\{W_{\sigma}|\hbox{$\sigma\in\P$ maximal}\}\cup
\{W_v|\hbox{$v\in\P$ a vertex}\}
\]
where $W_{\sigma}=\Int(\sigma)$, the interior of $\sigma$, and 
\[
W_v=\bigcup_{\tau\in\Bar(\P)\atop v\in\tau}\Int(\tau)
\]
is the (open) star of $v$ in $\Bar(\P)$. We define an affine chart
\[
\psi_{\sigma}:W_{\sigma}\hookrightarrow\AA_{\sigma}\subseteq N_{\RR}
\]
given by the inclusion of $W_{\sigma}$ in
$\AA_{\sigma}$, which denotes the unique $(n-1)$-dimensional
affine hyperplane in $N_{\RR}$ containing $\sigma$. Also, take
$\psi_v:W_v\rightarrow M_{\RR}/\RR v$ to be the projection.
One checks easily that for
$v\in\sigma$, $\psi_{\sigma}\circ\psi_v^{-1}$ is integral affine linear
(integrality follows from reflexivity of $\Delta$!)
so $B$ is an integral affine manifold with singularities. 

\begin{example}
\label{quintic}
Let $\Delta\subseteq\RR^4$ be the convex hull of the points
\begin{align*}
&(-1,-1,-1,-1),\cr & (4,-1,-1,-1),\cr & (-1,4,-1,-1),\cr & (-1,-1,4,-1),\cr &
(-1,-1,-1,4).
\end{align*}
Choose a triangulation $\P$ of $B=\partial\Delta$ into standard simplices;
this can be done in a regular way so that the restriction of $\P$ to
each two-dimensional face of $\Delta$ is as given by the light lines
in Figure \ref{quinticdisc}.
This gives a discriminant locus $\Gamma$ depicted by the dark lines
in the figure; the line segments coming out of the boundary of
the two-face are meant to illustrate the pieces of discriminant locus
contained in adjacent two-faces. The discriminant locus there
is not contained in the plane of the two-face. In particular, the
discriminant locus is not planar at the vertices of $\Gamma$ on the edges
of $\Xi$ with respect to the affine structure
we define.
Note $\Gamma$ is a trivalent graph, with two types of trivalent
vertices, the non-planar ones just mentioned and the planar vertices
contained in the interior of two-faces.
\begin{figure}
\includegraphics{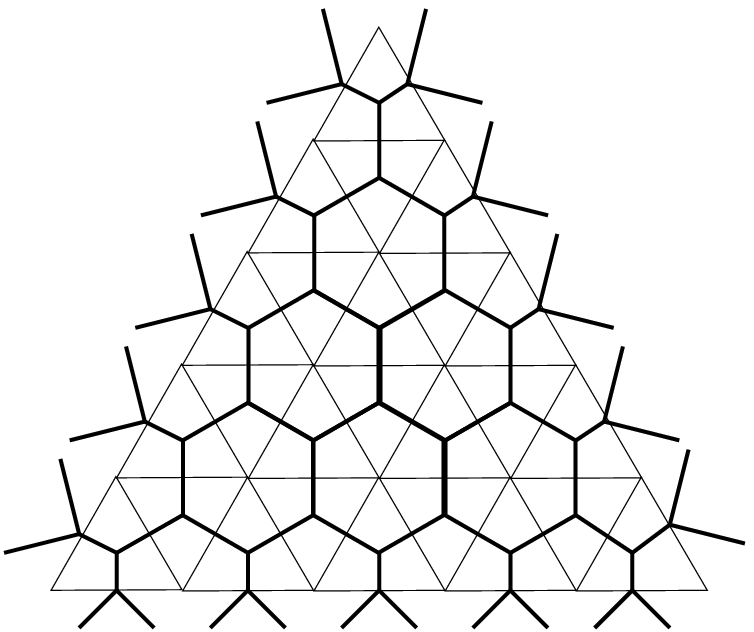}
\caption{}
\label{quinticdisc}
\end{figure}

For an affine manifold, the monodromy of the local system $\Lambda$ 
is an important feature of the affine structure. In this example,
it is very useful to analyze this monodromy around loops about the
discriminant locus. If $v$ is a vertex of $\Gamma$ contained in 
the interior of a two-face of $\Delta$, one can consider loops based
near $v$ in $B_0$ around the three line segments of $\Gamma$ adjacent
to $v$. It is an enjoyable exercise to calculate that these
monodromy matrices
take the form, in a suitable basis,
\[
T_1=\begin{pmatrix} 1&0&0\\1&1&0\\0&0&1\end{pmatrix},
T_2=\begin{pmatrix} 1&0&0\\0&1&0\\1&0&1\end{pmatrix},
T_3=\begin{pmatrix} 1&0&0\\-1&1&0\\-1&0&1\end{pmatrix}.
\]
They are computed by studying the composition of
transition maps between charts that a loop passes through.
These matrices can be viewed as specifying the obstruction to
extending the affine structure across a neighbourhood of $v$ in
$\Gamma$.
Of course, the monodromy
of $\check\Lambda$ is the transpose inverse of these matrices.
Similarly, if $v$ is a vertex of $\Gamma$ contained in an edge of 
$\Delta$, then the monodromy will take the form
\[
T_1=\begin{pmatrix} 1&-1&0\\0&1&0\\0&0&1\end{pmatrix},
T_2=\begin{pmatrix} 1&0&-1\\0&1&0\\0&0&1\end{pmatrix},
T_3=\begin{pmatrix} 1&1&1\\0&1&0\\0&0&1\end{pmatrix}.
\]
So we see that the monodromy of the two types of vertices are interchanged
between $\Lambda$ and $\check\Lambda$. 
\end{example}

One main result of \cite{TMS} is

\begin{theorem}
If $B$ is a three-dimensional tropical
affine manifold with singularities such that $\Gamma$ is
trivalent and the monodromy of $\Lambda$ at each vertex is one of the
above two types, then $f_0:X(B_0)\rightarrow B_0$ can be compactified to
a topological fibration $f:X(B)\rightarrow B$. Dually, $\check f_0:\check X(B_0)
\rightarrow B_0$ can be compactified to a topological fibration
$\check f:\check X(B)\rightarrow B$. Both $X(B)$ and $\check X(B)$ are
topological manifolds.
\end{theorem}

We won't give any details here of how this is carried out, but it is
not particularly difficult, as long as one restricts to the category
of topological (not $C^{\infty}$) manifolds. However, it is interesting
to look at the singular fibres we need to add.

If $b\in\Gamma$ is a point which is not a vertex of $\Gamma$, then $f^{-1}(b)$
is homeomorphic to $I_1\times S^1$, where $I_1$ denotes a Kodaira type $I_1$
elliptic curve, i.e., a pinched torus.

If $v$ is a vertex of $\Gamma$, with monodromy of the first type, then
$f^{-1}(v)=S^1\times S^1\times S^1/\sim$, with $(a,b,c)\sim (a',b',c')$
if $(a,b,c)=(a',b',c')$ or $a=a'=1$, where $S^1$ is identified with the unit
circle in $\CC$.
This is the three-dimensional analogue
of a pinched torus, and $\chi(f^{-1}(v))=+1$. We call this a \emph{positive}
fibre.

If $v$ is a vertex of $\Gamma$, with monodromy of the second type, then
$f^{-1}(v)$ can be described as $S^1\times S^1\times S^1/\sim$,
with $(a,b,c)\sim (a',b',c')$ if $(a,b,c)=(a',b',c')$ or $a=a'=1$, $b=b'$,
or $a=a',b=b'=1$.
The singular locus of this fibre is a figure eight, and $\chi(f^{-1}(v))=-1$.
We call this a \emph{negative} fibre.

So we see a very concrete local consequence of SYZ duality:
in the compactifications $X(B)$ and $\check X(B)$, the positive and
negative fibres are interchanged. Of course, this results in the
observation that the Euler characteristic changes sign under mirror symmetry
for Calabi-Yau threefolds.

\begin{example}
Continuing with Example \ref{quintic}, it was proved in \cite{TMS} 
that $\check X(B)$ is homeomorphic to the quintic and $X(B)$
is homeomorphic to the mirror quintic. Modulo a paper \cite{tori}
whose appearance
has been long-delayed because of other, more pressing, projects, 
the results of \cite{GBB} imply that the SYZ conjecture
holds for all complete intersections in toric varieties at a topological
level. 
\end{example}

W.-D.\ Ruan in 
\cite{Ruan} gave a description of \emph{Lagrangian} torus fibrations
for hypersurfaces in toric varieties using a symplectic flow argument, 
and his construction should
coincide with a \emph{symplectic} compactification of the symplectic
manifolds $\check X(B_0)$. In the three-dimensional case, such a 
symplectic compactification has been constructed by Ricardo
Casta\~no-Bernard and Diego Matessi \cite{CastMat}. If this compactification
is applied to the affine manifolds with singularities described here,
the resulting symplectic manifolds should be symplectomorphic
to the corresponding toric hypersurface, but this has not yet been shown.

\section{Tropical geometry}
\label{tropgeomsect}

Recalling that mirror symmetry is supposed to allow us to count curves,
let us discuss at an intuitive level how the picture so far gives us insight
into this question. Let $B$ be a tropical affine manifold. Then
as we saw, $X(B)$ carries the semi-flat complex structure, and it is
easy to describe some complex submanifolds of $X(B)$ 
as follows. Let $L\subseteq B$
be a linear subspace with rational slope, i.e., the tangent space $\shT_{L,b}$
to $L$ at any $b\in L$ can be written as $M\otimes_{\ZZ} \RR$
for some sublattice $M\subseteq \Lambda_b$. Then we obtain a submanifold
\[
X(L):=\shT_L/(\shT_L\cap \Lambda)\subseteq X(B).
\]
One checks easily that this is a complex submanifold. For example, if
$B=\RR^n$, so that $X(B)$ is just an algebraic torus $(\CC^*)^n$ with
coordinates $q_1,\ldots,q_n$, and $L\subseteq B$ is a codimension
$p$ affine linear subspace
defined by equations 
\[
\sum_j c_{ij}y_j=d_i,\quad 1\le i\le p,
\]
with $c_{ij}\in\ZZ$, $d_i\in\RR$, then the corresponding submanifold
of $X(B)$ is the subtorus given by the equations
\[
\prod_j q_j^{c_{ij}}= e^{-2\pi d_j},\quad 1\le i\le p.
\]

Of course, subtori of tori are not particularly interesting. How might we
build more complicated submanifolds? Let us focus on curves, where we
take the linear submanifolds of $B$ to be of dimension one. Then if we take
$L$ to be a line segment, ray, or line, $X(L)$ is a cylinder, with or without
boundary in the various cases. We can then try to glue
such cylinders together to obtain more complicated curves. For example,
imagine we are given rays  meeting at a point $b\in B=\RR^2$ as pictured
in Figure \ref{tropicalline}.
Take primitive integral tangent vectors $v_1,v_2,v_3\in\RR^2$ 
to $L_1,L_2$ and $L_3$
pointing outwards from the point $b$ where the three segments intersect.
Now we have the three cylinders $X(L_i)$ which do not match up over
$b$: the fibre $f^{-1}(b)=\RR^2/\ZZ^2$ intersects $X(L_i)$ in a circle
$\RR v_i/\ZZ v_i$. These circles are represented in $H_1(f^{-1}(b),\ZZ)=
\Lambda_b$ precisely by the vectors $v_1,v_2,v_3$, and so the condition
that the circles bound a surface in $f^{-1}(b)$ is that $v_1+v_2+v_3=0$. 
Thus, if this condition holds, we can glue in a surface $S$ contained
in $f^{-1}(b)$ so that $X(L_1)\cup X(L_2)\cup X(L_3)\cup S$ now has
no boundary at $b$. Of course, it is very far from being a holomorphic
submanifold. The expectation, however, is that this sort of object can
be deformed to a nearby holomorphic curve. 

\begin{figure}
\input{trivalent.pstex_t}
\caption{}
\label{tropicalline}
\end{figure}

Precisely, continuing with the above example, suppose $b=0$ and $v_1=(1,0)$,
$v_2=(0,1)$ and $v_3=(-1,-1)$. With holomorphic coordinates $q_1,q_2$
on $X(B)=(\CC^*)^2$, consider the curve $C\subseteq (\CC^*)^2$ defined
by $1+q_1+q_2=0$. Look at the
image of this curve under the map $f:X(B)\rightarrow B$, which here can be
written explicitly as $(q_1,q_2)\mapsto {-1\over 2\pi}(\log |q_1|,
\log |q_2|)$. One finds that one obtains a thickening of the trivalent
graph above, typically known as an amoeba. Further, if one considers
not $X(B)$ but $X_{\epsilon}(B)$, where now holomorphic coordinates are
given by $q_j=e^{2\pi i(x_j+iy_j)/\epsilon}$ and $f_{\epsilon}:X_{\epsilon}(B)
\rightarrow B$ is given by $(q_1,q_2)\mapsto -{\epsilon\over 2\pi}
(\log |q_1|,\log |q_2|)$, one finds that as $\epsilon\rightarrow 0$, 
$f_{\epsilon}(C)$ converges to the trivalent graph in the above figure.
In this sense the trivalent graph on $B$ is a limiting version of 
curves on a family of varieties tending towards a large complex structure
limit.

This basic picture for curves in algebraic tori is now very well studied.
In particular, this study spawned the subject of \emph{tropical
geometry}. The word \emph{tropical} is motivated by the role that the
\emph{tropical semiring} plays. This is the semiring $(\RR,\oplus, \odot)$
where addition and multiplication are given by
\begin{align*}
a\oplus b:= {} & \min(a,b)\\
a\odot b:= {} & a+b.
\end{align*}
The word ``tropical'' is used in honor of the Brazilian mathematician Imre
Simon, who pioneered use of this semi-ring.

We now consider polynomials over the tropical semiring, as follows.
Let $S\subseteq \ZZ^n$ be a finite subset, and
consider tropical polynomials on $\RR^n$ of the form
\[
g:=\sum_{(p_1,\ldots,p_n)\in S} c_{p_1\ldots p_n} x_1^{p_1}\cdots
x_n^{p_n}
\]
where the coefficients lie in $\RR$ and the operations are in the 
tropical semiring. Then $g$ is a convex piecewise linear function on
$\RR^n$, and the locus where $g$ is not linear is called a \emph{tropical
hypersurface}. In particular, in the case $n=2$, we obtain a tropical curve.
In the example of Figure \ref{tropicalline}, 
the relevant tropical polynomial could be taken to be
$0\oplus x_1\oplus x_2$.

While the tropical semiring has been used extensively in tropical geometry,
it is not so convenient for us to view our tropical curves on $B$ as being
defined by equations, since typically these curves will be of high codimension.
Instead, it is better to follow Mikhalkin \cite{Mik} and use parameterized
tropical curves.

The domain of a parameterized tropical curve will be a weighted
graph. In what follows, $\overline{\Gamma}$ will denote a connected graph.
Such a graph can be viewed in two different ways. First, it can be viewed
as a purely combinatorial object, i.e., a set $\overline{\Gamma}^{[0]}$
of vertices and a set $\overline{\Gamma}^{[1]}$ of edges consisting of
unordered pairs of elements of $\overline{\Gamma}^{[0]}$, indicating the
endpoints of an edge. We can also view $\overline{\Gamma}$ as the topological
realization of the graph, i.e., a topological space which is the union of
line segments corresponding to the edges. We shall confuse these two
viewpoints at will. We will then denote by $\Gamma$ the topological
space obtained from $\overline{\Gamma}$ by deleting the univalent vertices
of $\overline{\Gamma}$, so that $\Gamma$ may have some non-compact edges.

We also take $\overline{\Gamma}$ to come with a \emph{weight function}, a map
\[
w:\overline{\Gamma}^{[1]}\rightarrow \NN=\{0,1,2,\ldots\}.
\]

Replacing $\RR^n$ with a general tropical affine manifold
$B$, we now arrive at the following definition:

\begin{definition} A parameterized tropical curve in $B$ is a
continuous map
\[
h:\Gamma\rightarrow B
\]
where $\Gamma$ is obtained from a graph $\overline{\Gamma}$ as above,
satisfying the following two properties:
\begin{enumerate}
\item If $E\in\Gamma^{[1]}$ and $w(E)=0$, then $h|_E$ is
constant; otherwise $h|_E$ is a proper embedding of $E$ into $B$ as a
line segment, ray or line of rational slope.
\item  \emph{The balancing condition}. Let $V\in\overline{\Gamma}^{[0]}$
be a vertex with valency larger than $1$, with adjacent edges
$E_1,\ldots,E_{\ell}$. Let $v_i\in \Lambda_{h(V)}$ be a primitive
tangent vector to $h(E_i)$ at $h(V)$, pointing away from $h(V)$. Then
\[
\sum_{i=1}^{\ell} w(E_i)v_i=0.
\]
\end{enumerate}
\end{definition}

Here the balancing condition is just expressing the topological requirement
that the boundaries of the various cylinders $X(h(E_i))\subseteq X(B)$ 
can be connected
up with a surface contained in the fibre of $X(B)\rightarrow B$
over $h(V)$. The weights can be interpreted
as taking the cylinders $X(h(E_i))$ with multiplicity.

An important question then arises:

\begin{question}
When can a given parameterized tropical curve be viewed as a limit
of holomorphic curves in $X_{\epsilon}(B)$ as $\epsilon\rightarrow 0$?
\end{question}

This question has attracted a great deal of attention when $B=\RR^n$,
with completely satisfactory results in the case $n=2$ (Answer: always),
and less complete results when $n\ge 3$. The $n=2$ case was first treated
by Mikhalkin \cite{Mik}, and resuts in all dimensions were first obtained
by Nishinou and Siebert \cite{NS}. 
In particular, Mikhalkin proved that in this two-dimensional case,
one can calculate numbers of curves of a given degree and genus
passing through a fixed set of points, showing that difficult holomorphic
enumerative problems can be solved by a purely combinatorial approach.
This work gives hope that one can really count curves combinatorially
in much more general settings. In the two-dimensional case, again,
my own work \cite{GP2}
showed that the mirror side (for mirror symmetry for $\PP^2$)
could also be interpreted tropically, giving a completely tropical 
interpretation of mirror symmetry for $\PP^2$.

So far we have not considered the case that $B$ has singularities. In 
case $B$ has singularities, we expect that one should be able to relax
the balancing condition when a vertex falls inside of a point of the
singular locus, and in particular one can allow univalent vertices
which map to the singular locus. The reason for this is that once we
compactify $X(B_0)$ to $X(B)$, one expects to find holomorphic disks
fibering over line segments emanating from singular points: see Figure
\ref{vanishingdisk}
for a depiction of this when $B$ is two-dimensional, having isolated
singularities.

\begin{figure}
\centerline{\epsfbox{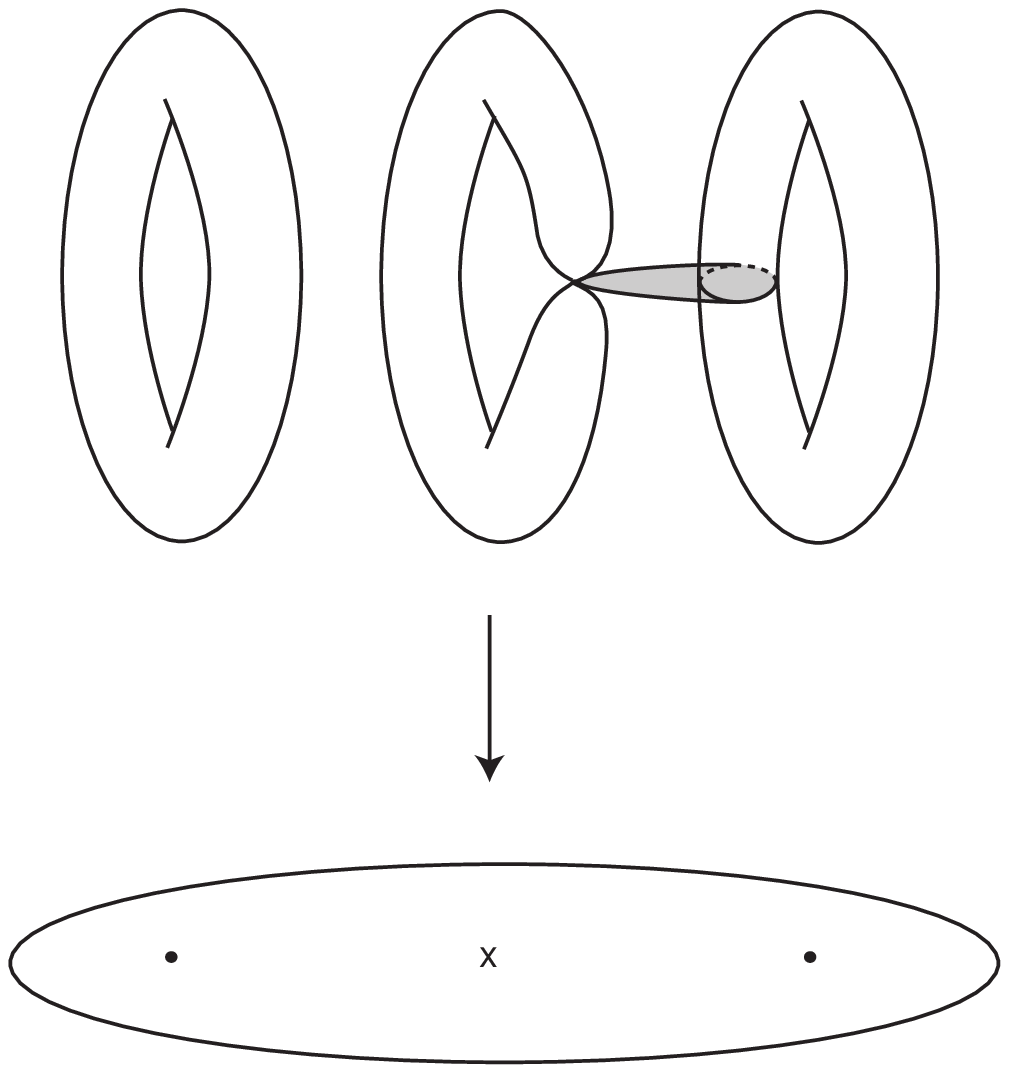}}
\caption{}
\label{vanishingdisk}
\end{figure}

We will avoid giving a precise definition of what a tropical curve should
mean in the case that $B$ has singularities, largely because it is not
clear yet what the precise definition should be. Hopefully, though, this
discussion makes it clear that at an intuitive level, counting curves
should be something which can be done on $B$.

\section{The problems with the SYZ conjecture, and how to get around them}
\label{problemsection}

The discussion of \S\ref{affsection}
demonstrates that the SYZ conjecture gives a beautiful
description of mirror symmetry at a purely topological level. This, by itself,
can often be useful, but fails to get at the original hard differential
geometric conjecture and fails to give insight into why mirror symmetry
counts curves. 

In order for the full-strength version of the SYZ conjecture to hold,
the strong version of duality for topological torus fibrations we saw
in \S\ref{affsection}
should continue to hold at the special Lagrangian
level. This would mean that a mirror pair $X,\check X$ would
possess special Lagrangian torus fibrations
$f:X\rightarrow B$ and $\check f:\check X\rightarrow B$ with codimension
two discriminant loci, and the discriminant loci of $f$ and
$\check f$ would coincide.
These fibrations would then be dual away from the discriminant locus.

There are examples of special Lagrangian fibrations on non-compact
toric varieties $X$ with discriminant locus looking very similar to
what we have described in the topological case. In particular, if 
$X$ is an $n$-dimensional Ricci-flat K\"ahler manifold with a 
$T^{n-1}$-action preserving the metric and holomorphic $n$-form,
then $X$ will have a very nice special Lagrangian fibration
with codimension two discriminant locus. (See \cite{SLAGex} and 
\cite{Gold}). However, Dominic Joyce (see \cite{Joyce} and other
papers cited therein) began studying some 
three-dimensional $S^1$-invariant
examples, and discovered quite different behaviour. There is an argument
in \cite{SlagII}
that if a special Lagrangian fibration is $C^{\infty}$, then the
discriminant locus will be (Hausdorff) codimension two. However, Joyce
discovered examples which were not differentiable, but only piecewise
differentiable, and furthermore, had a codimension one discriminant locus:

\begin{example}
Define $F:\CC^3\rightarrow \RR\times\CC$ by
$F(z_1,z_2,z_3)=(a,c)$ with
$2a=|z_1|^2-|z_2|^2$ and
\[
c=\begin{cases}
z_3&a=z_1=z_2=0\\
z_3-\bar z_1\bar z_2/|z_1|& a\ge 0, z_1\not=0\\
z_3-\bar z_1\bar z_2/|z_2|&a<0.
\end{cases}
\]
It is easy to see that if $a\not=0$, then $F^{-1}(a,c)$ is homeomorphic to
$\RR^2\times S^1$, while if $a=0$, then $F^{-1}(a,c)$ is a cone over $T^2$:
essentially, one copy of $S^1$ in $\RR^2\times S^1$ collapses to a point. 
In addition, all fibres of this map are special Lagrangian, and it is obviously
only piecewise smooth. The discriminant locus is the entire plane given
by $a=0$.
\end{example}

This example forces a reevaluation of the strong form of the SYZ conjecture.
In further work Joyce found evidence for a more likely picture
for general special Lagrangian fibrations in three dimensions. The discriminant
locus, instead of being a codimension two graph, will be a codimension one
blob. Typically the union of the singular points of singular fibres will
be a Riemann surface, and it will map to an amoeba-shaped set in $B$, i.e.,
the discriminant locus looks like the picture on the right rather than the 
left in Figure \ref{fatdisc},
and will be a fattening of the old picture of a codimension two discriminant.

\begin{figure}
\includegraphics{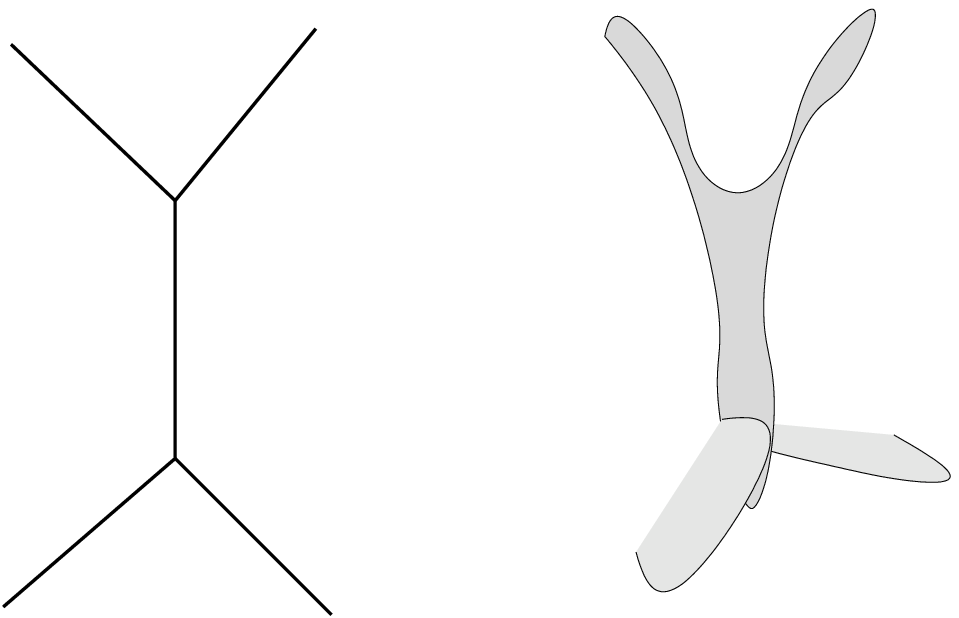}
\caption{}
\label{fatdisc}
\end{figure}

Joyce made some additional arguments to suggest that this fattened 
discriminant locus must look fundamentally different in a neighbourhood
of the two basic types of vertices we saw in \S\ref{affsection},
with the two types of vertices expected to appear pretty much as depicted in
Figure \ref{fatdisc}.
Thus the strong form of duality mentioned above, where we expect the
discriminant loci of the special Lagrangian fibrations on a mirror pair
to be the same, cannot hold. If this is the case, one needs to replace
this strong form of duality with a weaker form.

It seems likely that the best way to rephrase the SYZ conjecture is
in a limiting form. Mirror symmetry as we currently understand it has to
do with degenerations of Calabi-Yau manifolds. Given a flat family $f:\X
\rightarrow D$ over a disk $D$, with the fibre $\X_0$ over $0$
singular and all other fibres $n$-dimensional
Calabi-Yau manifolds, we say the family is \emph{maximally unipotent}
if the monodromy transformation $T:H^n(\X_t,\QQ)\rightarrow H^n(\X_t,\QQ)$
($t\in D$ non-zero) satisfies $(T-I)^{n+1}=0$ but $(T-I)^n\not=0$.
It is a standard expectation 
of mirror symmetry that mirrors should be associated
to maximally unipotent degenerations of Calabi-Yau manifolds. In particular,
given two different maximally unipotent degenerations in a single complex
moduli space for some Calabi-Yau manifold, one might obtain different
mirror manifolds. Such degenerations are usually called ``large complex
structure limits'' in the physics literature, although sometimes this
phrase is used to impose some additional conditions on the degeneration,
see \cite{Morr}.

We recall the definition of Gromov-Hausdorff convergence,
a notion of convergence of a sequence of metric spaces.

\begin{definition}
Let $(X,d_X)$, $(Y,d_Y)$ be two compact
metric spaces. Suppose there exists maps $f:X\rightarrow Y$
and $g:Y\rightarrow X$ (not necessarily continuous) such that
for all $x_1,x_2\in X$,
$$|d_X(x_1,x_2)-d_Y(f(x_1),f(x_2))|<\epsilon$$
and for all $x\in X$,
$$d_X(x,g\circ f(x))<\epsilon,$$
and the two symmetric properties for $Y$ hold. Then we say the
Gromov--Hausdorff distance between $X$ and $Y$ is at most $\epsilon$.
The Gromov--Hausdorff distance $d_{GH}(X,Y)$ is the infimum of all
such $\epsilon$.
\end{definition}

It follows from results of Gromov (see for example \cite{Petersen}, 
pg. 281, Cor.
1.11) that the space of compact Ricci-flat manifolds with diameter $\le C$
is precompact with respect to Gromov-Hausdorff distance, 
i.e., any sequence of such manifolds has a subsequence
converging with respect to the Gromov-Hausdorff distance to a metric
space. This metric space could be quite bad; this is quite outside
the realm of algebraic geometry! Nevertheless, this raises the following
natural question. Given a maximally unipotent
degeneration of Calabi-Yau manifolds $\X\rightarrow D$, take a 
sequence $t_i\in D$ converging to $0$, and consider a sequence $(\X_{t_i},
g_{t_i})$, where $g_{t_i}$ is a choice of Ricci-flat metric chosen so
that $Diam(g_{t_i})$ remains bounded. What is the Gromov-Hausdorff
limit of $(\X_{t_i},g_{t_i})$, or the limit of some convergent subsequence?

\begin{example} Consider a degenerating family of elliptic curves
parameterized by $t$, given by $\CC/(\ZZ+\ZZ\tau)$ where
$1$ and $\tau={1\over 2\pi i}\log t$ are periods of the elliptic curves. 
If we take $t$ approaching
$0$ along the positive real axis, then we can just view this as a family
of elliptic curves $\X_{\alpha}$
with period $1$ and $i\alpha$ with $\alpha\rightarrow\infty$.
If we take the standard Euclidean metric $g$ on $\X_{\alpha}$, then
the diameter of $\X_{\alpha}$ is unbounded. To obtain a bounded diameter,
we replace $g$ by $g/\alpha^2$; equivalently, we can keep $g$ fixed
on $\CC$ but change the periods of the elliptic curve to $1/\alpha, i$.
It then becomes clear that the Gromov-Hausdorff limit of such a sequence
of elliptic curves is a circle $\RR/\ZZ$.
\end{example}

This simple example motivates the first conjecture about maximally
unipotent degenerations, conjectured independently by myself and Wilson
on the one hand \cite{GrWi} and Kontsevich and Soibelman \cite{KS} on the other.

\begin{conjecture}
\label{GWKSconj}
Let $\X\rightarrow D$ be a maximally unipotent
degeneration of simply-connected 
Calabi-Yau manifolds with full $SU(n)$ holonomy, 
$t_i\in D$ with $t_i\rightarrow 0$,
and let $g_i$ be a Ricci-flat metric on $\X_{t_i}$ normalized to have
fixed diameter $C$. Then a convergent subsequence of $(\X_{t_i},g_i)$
converges to a metric space $(X_{\infty},d_{\infty})$, where $X_{\infty}$
is homeomorphic to $S^n$. Furthermore, $d_{\infty}$ is induced by
a Riemannian metric on $X_{\infty}\setminus\Gamma$, where $\Gamma\subseteq
X_{\infty}$ is a set of codimension two.
\end{conjecture}

Here the topology of the limit depends on the nature of the non-singular
fibres $\X_t$; for example,
if instead $\X_t$ was hyperk\"ahler, then we would expect
the limit to be a projective space. Also, even in the case of full
$SU(n)$ holonomy, if $\X_t$ is not simply connected, we would expect
limits such as $\QQ$-homology spheres to arise.

Conjecture \ref{GWKSconj} is directly inspired by the SYZ conjecture. Suppose
we had special Lagrangian fibrations $f_i:\X_{t_i}\rightarrow B_i$.
Then as the maximally unipotent degeneration is approached, one can
see that the volume of the fibres of these fibrations
goes to zero. This would suggest these fibres collapse, hopefully
leaving the base as the limit.

This conjecture was proved by myself and Wilson
in 2000 for K3 surfaces in \cite{GrWi}. The proof relied on a number
of pleasant facts about K3 surfaces. First, they are hyperk\"ahler manifolds,
and a special Lagrangian torus fibration becomes an elliptic fibration
after a hyperk\"ahler rotation of the complex structure. Since it is 
easy to construct elliptic fibrations on K3 surfaces, and indeed such
a fibration arises from the data of the maximally unipotent degeneration,
it is easy to obtain a special Lagrangian fibration. Once this is done,
one needs to carry out
a detailed analysis of the behaviour of Ricci-flat metrics in the limit.
This is done by creating good approximations to Ricci-flat metric,
using the existence of explicit local models for these metrics
near singular fibres of special Lagrangian fibrations in complex dimension
two.

Most of the techniques used are not available in higher dimension.
However, much more recently, weaker collapsing results in the hyperk\"ahler 
case were obtained in work with V.\ Tosatti and Y.\ Zhang in 
\cite{GTZ}, 
assuming the existence of abelian variety fibrations analogous to the
elliptic fibrations in the K3 case. Rather than getting an explicit
approximate Ricci-flat metric, we make use of a priori 
estimates of Tosatti in \cite{Tos}.  

In the general Calabi-Yau case, the only progress towards the conjecture
has been work of Zhang in \cite{Zhang} showing existence of special Lagrangian
fibrations in regions of Calabi-Yau manifolds with bounded injectivity
radius and sectional curvature and deduces local collapsing from the
existence of special Lagrangian fibrations. 

The motivation for Conjecture \ref{GWKSconj} 
from SYZ also provides a limiting form
of the conjecture. There are any number of problems with trying to prove
the existence of special Lagrangian fibrations on Calabi-Yau manifolds.
Even the existence of a single special Lagrangian torus near a maximally
unipotent degeneration is unknown, but we expect it should be easier to
find them as we approach the maximally unipotent point. Furthermore,
even if we find a special Lagrangian torus, we know that it moves
in an $n$-dimensional family, but we don't know its deformations fill
out the entire manifold. In addition, there is no guarantee that even if it
does, we obtain a foliation of the manifold: nearby special Lagrangian
submanifolds may intersect. (For an example, see \cite{Matessi}.)
So instead, we will just look at the moduli space
of special Lagrangian tori.

Given a maximally unipotent degeneration of Calabi-Yau manifolds of 
dimension $n$, it is known that the image of $(T-I)^n:H_n(\shX_t,\QQ)
\rightarrow H_n(\shX_t,\QQ)$ is a one-dimensional subspace $W_0$.
Suppose, given a sequence $t_i$ with $t_i \rightarrow 0$ as $i\rightarrow
\infty$, that for $t_i$ sufficiently close
to zero, there is a special Lagrangian $T^n$ which generates $W_0$.
This is where we expect to find fibres
of a special Lagrangian fibration associated to a maximally unipotent
degeneration. Let $B_{0,i}$ be the moduli space of deformations
of this torus; every point of $B_{0,i}$ corresponds to a smooth special
Lagrangian torus in $\X_{t_i}$. This manifold then comes
equipped with the McLean metric and affine structures
defined in \S 2. One can then
compactify $B_{0,i}\subseteq B_i$, (probably by taking the closure
of $B_{0,i}$ in the space of special Lagrangian currents; the details
aren't important here). This gives a series of metric spaces $(B_i,d_i)$
with the metric $d_i$ induced by the McLean metric. If the McLean
metric is normalized to keep the diameter of $B_i$ constant independent of $i$,
then we can hope that $(B_i,d_i)$ converges to a compact metric space
$(B_{\infty},d_{\infty})$. Here then is the limiting form of SYZ:

\begin{conjecture} If $(\X_{t_i},g_i)$ converges to $(X_{\infty},g_{\infty})$
and $(B_i,d_i)$ is non-empty for large $i$ and converges to
$(B_{\infty},d_{\infty})$, then $B_{\infty}$ and $X_{\infty}$
are isometric up to scaling. Furthermore, there is a subspace $B_{\infty,0}
\subseteq B_{\infty}$ with $\Gamma:=B_{\infty}\setminus B_{\infty,0}$ 
of Hausdorff codimension 2 in $B_{\infty}$
such that $B_{\infty,0}$ is a Monge-Amp\`ere manifold, with the Monge-Amp\`ere
metric inducing $d_{\infty}$ on $B_{\infty,0}$.
\end{conjecture}

Essentially what this is saying is that as we approach the maximally
unipotent degeneration, we expect to have a special Lagrangian fibration
on larger and larger subsets of $\X_{t_i}$. Furthermore, in the limit,
the codimension one discriminant locus suggested by Joyce converges to
a codimension two discriminant locus, and (the not necessarily
Monge-Amp\`ere, see \cite{Matessi}) Hessian metrics on $B_{0,i}$
converge to a Monge-Amp\`ere metric.

The main point I want to get at here is that
it is likely the SYZ conjecture is only ``approximately'' correct, and
one needs to look at the limit to have a hope of proving anything.
On the other hand, the above conjecture seems likely to be accessible
by currently understood techniques.  
I remain hopeful that this conjecture will be proved,
though much additional work will be necessary.

How do we do mirror symmetry using this modified version of the SYZ conjecture?
Essentially, we would follow these steps:
\begin{enumerate}
\item 
We begin with a maximally unipotent degeneration of Calabi-Yau manifolds
$\X\rightarrow D$, along with a choice of polarization. This gives
us a K\"ahler class $[\omega_t]\in H^2(\X_t,\RR)$ for each $t\in D\setminus
0$, represented by $\omega_t$ the K\"ahler form of a Ricci-flat metric
$g_t$. 
\item Identify the Gromov-Hausdorff limit of a sequence $(\X_{t_i}, r_ig_{t_i})$
where $t_i\rightarrow 0$ and $r_i$ is a scale factor which keeps the diameter
of $\X_{t_i}$ constant. The limit will be, if the above conjectures work,
an affine manifold with singularities $B$ along with a Monge-Amp\`ere metric.
\item Perform a Legendre transform to obtain a new affine manifold with
singularities $\check B$, though with the same metric.
\item Try to construct a compactification of $X_{\epsilon}(\check B_0)$
for small $\epsilon>0$ to obtain a complex manifold $X_{\epsilon}(\check B)$.
This will be the mirror manifold.
\end{enumerate}

As we shall see, we do not expect that we will need the full strength
of steps (2) and (3) to carry out mirror symmetry; some way of identifying the
base $B$ will be sufficient. Nevertheless, (2) is interesting
from the point of view of understanding the differential geomtry
of Ricci-flat K\"ahler manifolds.

Step (4), on the other hand, is crucial, and
we need to elaborate on this last step a bit more. The
problem is that while we expect that it should be possible in general
to construct symplectic compactifications of the symplectic manifold
$\check X(B_0)$ (and hence get the mirror as a symplectic manifold, see
\cite{CastMat} for the three-dimensional case), 
we don't expect to be able to compactify $X_{\epsilon}(\check B_0)$
as a complex manifold. Instead, the expectation is that a small deformation
of $X_{\epsilon}(\check B_0)$ is necessary before it can be compactified.
Furthermore, this small deformation is critically important in mirror
symmetry: \emph{it is this small deformation which provides the $B$-model
instanton corrections}. 

Because this last item is so important, let's give it a name:

\begin{question}[The reconstruction problem, Version I]
\label{reconstruct1}
Given a tropical affine manifold with singularities $B$, construct
a complex manifold $X_{\epsilon}(B)$ which is a compactification
of a small deformation of $X_{\epsilon}(B_0)$.
\end{question} 

We will return to this question later in the paper. However, 
I do not wish to dwell further on the differential-geometric versions
of the SYZ conjecture here. Instead I will move on to describing how
the above discussion motivated the algebro-geometric
program developed by myself and Siebert for understanding mirror symmetry,
and then describe recent work and ideas coming
out of this program.

\section{Gromov-Hausdorff limits, algebraic degenerations, and mirror
symmetry}

We now have two notions of limit: the familiar algebro-geometric 
notion of a degenerating family $\X\rightarrow D$ over a disk
on the one hand, and the Gromov-Hausdorff limit on the other. In 2000 Kontsevich
and Soibelman
had an important insight (see \cite{KS}) into the connection between these
two. In this section I will give a rough idea of how and why this works.

Very roughly speaking, the Gromov-Hausdorff limit $(\X_{t_i},g_{t_i})$
as $t_i\rightarrow 0$, or equivalently, the base of the putative SYZ
fibration, should coincide, topologically, with the \emph{dual
intersection complex} of the singular fibre $\X_0$. More precisely,
in a relatively simple situation,
suppose $f:\X\rightarrow D$ is relatively minimal (in the sense of Mori)
and normal crossings,
with $\X_0$ having irreducible components $X_1,\ldots,X_m$. The dual 
intersection complex of $\X_0$ is the simplicial complex with vertices
$v_1,\ldots,v_m$, and which contains a simplex $\langle v_{i_0},\ldots,
v_{i_p}\rangle$ if $X_{i_0}\cap\cdots\cap X_{i_p}\not=\emptyset$. The
idea that the dual intersection complex should play a role in describing
the base of the SYZ fibration was perhaps first suggested by Leung
and Vafa in \cite{LV}.

Let us explain roughly why this should be, first by looking at a standard
family of degenerating elliptic curves with periods $1$ and ${n\over 2\pi i}
\log t$ for $n$ a positive integer. Such a family over the punctured disk
is extended to a family over the disk by adding a Kodaira type 
$I_n$ (a cycle of $n$ rational curves) fibre over the origin.

Taking a sequence $t_i\rightarrow 0$ with $t_i$ real and positive gives
a sequence of elliptic curves of the form $X_{\epsilon_i}(B)$ where
$B=\RR/n\ZZ$ and
$\epsilon_i=-{2\pi\over\ln t_i}$. In addition, the metric on 
$X_{\epsilon_i}(B)$, properly scaled, comes from the constant Hessian
metric on $B$. So we wish to explain how $B$ is related to the geometry 
near the singular fibre. To this end,
let $X_1,\ldots,X_n$ be the irreducible components of $\X_0$; these
are all $\PP^1$'s. Let $P_1,\ldots,P_n$ be the singular points of $\X_0$.

We'll consider two sorts of open sets in $\X$. For the first type,
choose a coordinate $z$ on $X_i$, with $P_i$ given by $z=0$ and
$P_{i+1}$ given by $z=\infty$. Let $U_i\subseteq X_i$ be the open set
$\{z\,|\,\delta\le |z| \le 1/\delta\}$ for some small fixed $\delta$. Then one
can find a neighbourhood $\widetilde U_i$ of $U_i$ in $\X$ such that 
$\widetilde U_i$ is biholomorphic to $U_i\times D_{\rho}$ for $\rho>0$ sufficiently
small, $D_{\rho}$ a disk of radius $\rho$ in $\CC$, and
$f|_{\widetilde U_i}$ is the projection onto $D_{\rho}$. 

On the
other hand, each $P_i$ has a neighbourhood $\widetilde V_i$ in $\X$
biholomorphic to a polydisk
$\{(z_1,z_2)\in\CC^2\,|\,|z_1|\le \delta', |z_2|\le\delta'\}$ on which $f$
takes the form $z_1z_2$. 

If $\delta$ and $\delta'$ are chosen correctly,
then for $t$ sufficiently close to zero,
\[
\{\widetilde V_i\cap\X_t\,|\,1\le i\le n\}\cup \{\widetilde U_i\cap\X_t\,|\,1\le i\le n\}
\]
form an open cover of $\X_t$. Now each of the sets in this open cover
can be written as $X_{\epsilon}(U)$ for some $U$ a one-dimensional
(non-compact) affine manifold and
$\epsilon=-2\pi/\ln|t|$. If $U$ is an open interval $(a,b)\subseteq
\RR$, then $X_{\epsilon}(U)$ is biholomorphic to the annulus
\[
\{z\in\CC\,|\, e^{-2\pi b/\epsilon}\le |z|\le e^{-2\pi a/\epsilon}\}
\]
as $q=e^{2\pi i(x+i y)/\epsilon}$ is a holomorphic coordinate on
$X_{\epsilon}((a,b))$. 
Thus 
\[
\widetilde U_i\cap \X_t\cong X_{\epsilon}\left(\left({\epsilon\ln\delta\over 2\pi},
-{\epsilon\ln\delta\over 2\pi}\right)\right)
\]
with $\epsilon=-2\pi/\ln|t|$. As
$t\rightarrow 0$, the interval $(\epsilon\ln\delta/2\pi,
-\epsilon\ln\delta/2\pi)$ shrinks to a point. So $\widetilde U_i\cap
\X_t$ is a smaller and smaller open subset of $\X_t$ as $t\rightarrow 0$
when we view things in this way. This argument suggests that every
irreducible component should be associated to a point on $B$.

Now look at $\widetilde V_i\cap\X_t$. This is
\begin{eqnarray*}
\{(z_1,z_2)\in\CC^2\,|\,|z_1|,|z_2|<\delta', z_1z_2=t\}
&\cong&\{z\in\CC\,|\,|t|/\delta'\le |z|\le \delta'\}\\
&\cong& X_{\epsilon}\left({-\epsilon\over 2\pi}\ln\delta',
{\epsilon\over 2\pi} (\ln\delta'-\ln |t|)\right)
\end{eqnarray*}
with $\epsilon=-2\pi/\ln|t|$. This interval approaches the unit interval
$(0,1)$ as $t\rightarrow 0$. So the open set $\widetilde V_i\cap \X_t$ ends up being
a large portion of $\X_t$. We end up with $\X_t$, for small $t$, being
a union of open sets of the form 
$X_{\epsilon}((i+\epsilon',i+1-\epsilon'))$ (i.e., $\widetilde V_i\cap\X_{\epsilon}$)
and $X_{\epsilon}((i-\epsilon'',i+\epsilon''))$ (i.e., $\widetilde U_i\cap
\X_t$) for $\epsilon'$, $\epsilon''$ sufficiently small. These should glue,
at least approximately, to give $X_{\epsilon}(B)$. So we see
that irreducible components of $\X_0$ seem to coincide with points on $B$,
but intersections of components
coincide with lines. In this way we see the dual
intersection complex emerge.

\medskip

Let us make one more observation before beginning with rigorous
results in the next section. Suppose more generally we had a \emph{Gorenstein
toroidal
crossings} degeneration of Calabi-Yau manifolds $f:\X\rightarrow D$ (see
\cite{ss}). 
This means that every point $x\in\X$ has a neighbourhood isomorphic
to an open set in an affine Gorenstein (i.e., the canonical class is
a Cartier divisor) toric variety, with $f$ given
locally by a monomial which vanishes exactly
to order $1$ on each codimension one
toric stratum. This is a generalization of the notion of normal crossings.
Very roughly, the above argument suggests that each irreducible component
of the central fibre will correspond to a point of the Gromov-Hausdorff
limit. The following exercise shows what kind of contribution to $B$
to expect from a point $x\in\X_0$ which is a zero-dimensional stratum in
$\X_0$.

\begin{xca}
\label{gorensteinlimit}
Suppose that there is a point $x\in\X_0$ which has
a neighbourhood isomorphic to a neighbourhood of a dimension zero torus
orbit of an affine Gorenstein toric variety $Y_x$. 
Such an affine variety is
specified as follows.
Set $M=\ZZ^n$, $M_{\RR}=M\otimes_{\ZZ}\RR$, $N=\Hom_{\ZZ}(M,\ZZ)$,
$N_{\RR}=N\otimes_{\ZZ}\RR$ with $n=\dim\X_t$. Then
there is a lattice polytope $\sigma\subseteq M_{\RR}$, 
$C(\sigma):=\{(rm,r)\,|\, m\in\sigma,r\ge 0\}\subseteq
M_{\RR}\oplus\RR$, $P:=\dual{C(\sigma)}\cap (N\oplus\ZZ)$ the monoid
determined by the dual of the cone $C(\sigma)$, $Y_x
=\Spec \CC[P]$, and finally $f$ coincides with the monomial $z^{(0,1)}$.

Now let us take a small neighbourhood of $x$ of the form
\[
\widetilde 
U_{\delta}=\{y\in \Spec \CC[P]\,|\,\hbox{$|z^p|<\delta$ for all $p\in P$}\}.
\]
This is an open set as the condition $|z^p|<\delta$ can be tested on a
finite generating set for $P$, provided that $\delta<1$. 
Then show that for a given $t$, $|t|<1$ and $\epsilon=-2\pi/\log|t|$, if 
\[
\sigma_t:=\{m\in M_{\RR}\,|\,\hbox{$\langle p,(m,1)\rangle>{\log\delta\over
\log |t|}$ for all $p\in P$}\},
\]
then 
\[
f^{-1}(t)\cap \widetilde U_{\delta}\cong X_{\epsilon}(\sigma_t).
\]
Note that 
\[
\sigma:=\{m\in M_{\RR}\,|\,\hbox{$\langle p,(m,1)\rangle\ge 0$ for all $p\in P$}\},
\]
so $\sigma_t$ is an open subset of $\sigma$, and
as $t\rightarrow 0$, $\sigma_t$ converges to the interior of
$\sigma$. 
\qed
\end{xca}

This observation hopefully motivates the basic construction
of the next section.

\section{Toric degenerations, the intersection complex and its dual}

I will now introduce the basic objects of the program developed by
myself and Siebert to understand mirror symmetry in an algebro-geometric
context. This program was announced in \cite{Announce}, and has been developed
further in a series of papers \cite{PartI}, \cite{PartII}, \cite{Annals},
\cite{GBB}, \cite{JAMS}, \cite{GPS}.

The motivation for this program came from two different directions. The first,
which was largely my motivation, was the discussion of the
limiting form of the SYZ conjecture of the previous sections. The second
arose in work of
Schr\"oer and Siebert \cite{ssKod}, \cite{ss}, which led Siebert
to the idea that log structures on degenerations of Calabi-Yau manifolds
would allow one to view mirror symmetry as an operation performed on
degenerate Calabi-Yau varieties. Siebert observed that at a combinatorial
level, mirror symmetry exchanged data pertaining to the log structure
and a polarization. This will be explained more clearly in the following
section, when I introduce log structures. Together, Siebert and I realised
that the combinatorial data he was considering could be encoded naturally
in the dual intersection complex of the degeneration, which we saw in the
previous section appears to be the base of the SYZ fibration.
The combinatorial interchange of data necessary for mirror symmetry
then corresponded to a discrete Legendre transform on the dual
intersection complex. It became apparent that this approach provided
an algebro-geometrization of the SYZ conjecture.

To set this up properly, one has to consider what kind of degenerations
to allow. They should be maximally unipotent, of course, but there can
be many different birational models of degenerations. 
Below we define the notion of \emph{toric
degeneration}. The class of toric degenerations may seem rather restrictive,
but it appears to be the largest class of degenerations closed under
mirror symmetry: one can construct the mirror of a toric degeneration
as a toric degeneration. It does not appear that there is any other
natural family of degenerations with this property.
Much of the material in this section comes from \cite{PartI}, \S 4. 

Roughly put, a toric degeneration of Calabi-Yau varieties is a degeneration
whose central fibre is a union of toric varieties glued along toric strata,
and the total space of the degeneration is, off of some well-behaved
set $Z$ contained in the central fibre, locally toric with the family 
locally given
by a monomial. The precise technical definition is as follows.

\begin{definition}
\label{toricdegen}
Let $f:\X\rightarrow D$ be a proper flat family of relative dimension
$n$, where $D$ is a disk and $\X$ is a complex analytic space
(not necessarily non-singular). We say $f$ is a {\it toric degeneration}
of Calabi-Yau varieties if
\begin{enumerate}
\item
$\X_t$ is an irreducible normal Calabi-Yau variety with only
canonical singularities for $t\not=0$. (The reader may
like to assume $\X_t$ is smooth for $t\not=0$).
\item
If $\nu:\widetilde\X_0\to\X_0$ is the normalization,
then $\widetilde\X_0$ is a disjoint union of toric varieties,
the conductor locus $C\subseteq\widetilde\X_0$ is reduced,
and the map $C\to\nu(C)$ is unramified and generically
two-to-one. (The conductor locus is a naturally defined scheme structure
on the set where $\nu$ is not an isomorphism.)
The square
\[\begin{CD}
C@>>> \widetilde\X_0\\
@VVV @VV{\nu}V\\
\nu(C)@>>> \X_0
\end{CD}\]
is cartesian and cocartesian.
\item $\X_0$ is a reduced Gorenstein space
and the conductor locus $C$
restricted to each irreducible component of $\widetilde\X_0$ is the union
of all toric Weil divisors of that component. 
\item There exists a closed subset $Z\subseteq\X$ of relative
codimension $\ge 2$ such that $Z$ satisfies the following properties:
$Z$ does not contain the image under $\nu$
of any toric stratum of $\widetilde\X_0$,
and for any point $x\in \X\setminus Z$, there is a neighbourhood
$\widetilde U_x$ (in the analytic topology) of $x$, an
$n+1$-dimensional affine toric variety $Y_x$, a regular function
$f_x$ on $Y_x$ given by a monomial, and a commutative diagram
$$\begin{matrix}
\widetilde U_x&\mapright{\psi_x}&Y_x\cr
\mapdown{f|_{\widetilde U_{x}}}&&\mapdown{f_x}\cr
D'&\mapright{\varphi_x}&\CC\cr
\end{matrix}$$
where $\psi_x$ and $\varphi_x$ are open embeddings and $D'\subseteq D$. 
Furthermore,
$f_x$ vanishes precisely once on each toric divisor of $Y_x$.
\end{enumerate}
\end{definition}

\begin{example}
\label{quarticexample}
Take $\X$ to be defined by the equation $tf_4+z_0z_1z_2z_3=0$
in $\PP^3\times D$, where $D$ is a disk with coordinate
$t$ and $f_4$ is a general homogeneous quartic polynomial on $\PP^3$.
It is easy to see that $\X$ is singular at the locus
\[
\{t=f_4=0\}\cap Sing(\X_0).
\]
As $\X_0$ is the coordinate tetrahedron, the singular locus of
$\X_0$ consists of the 
six coordinate lines of $\PP^3$, and $\X$ has four singular points along
each such line, for a total of 24 singular points. 
Take $Z=Sing(\X)$. Then away from
$Z$, the projection $\X\rightarrow D$ is normal crossings, which
yields condition (4) of the definition of toric degeneration. It is easy
to see all other conditions are satisfied.
\end{example}

Given a toric degeneration $f:\X\rightarrow D$, 
we can build the \emph{dual intersection
complex} $(B,\P)$ of $f$, as follows. Here $B$ is an integral affine
manifold with singularities, and $\P$ is a \emph{polyhedral decomposition}
of $B$, i.e., a decomposition of $B$ into lattice polytopes. In fact, we 
will construct $B$ as a union of lattice polytopes.
Specifically,
let the normalisation of $\X_0$, $\widetilde \X_0$, be written as a disjoint
union $\coprod X_i$ of toric varieties $X_i$, $\nu:\widetilde\X_0\rightarrow\X_0$
the normalisation. The {\it strata} of $\X_0$ are the elements of
the set
\[
Strata(\X_0)=\{\nu(S)\,|\,
\hbox{$S$ is a toric stratum of $X_i$ for some $i$}\}.
\]
Here by toric stratum we mean the closure of a $(\CC^*)^n$ orbit.

Let $\{x\}\in Strata(\X_0)$ be a zero-dimensional stratum. 
Applying Definition \ref{toricdegen},(4), to a neighbourhood of $x$,
there is a toric variety $Y_x$ 
such that in
a neighbourhood of $x$, $f:\X\rightarrow D$ is locally isomorphic to
$f_x:Y_x\rightarrow\CC$, where $f_x$ is given by a monomial.
Now the condition that $f_x$ vanishes precisely
once along each toric divisor of $Y_x$ is the statement
that $Y_x$ is Gorenstein, and as such, it arises 
as in Exercise \ref{gorensteinlimit}. Indeed, let $M,N$ be given
in Exercise \ref{gorensteinlimit},
with $\rank M=\dim\X_0$. Then there
is a lattice polytope $\sigma_x\subseteq M_{\RR}$ such that $C(\sigma_x)
=\{(rm,r)|m\in\sigma, r\ge0\}$
is the cone defining the toric variety $Y_x$. As we saw in Exercise
\ref{gorensteinlimit}, a small neighbourhood of $x$ in $\X$ should contribute
a copy of $\sigma_x$ to $B$, which provides the motivation for our construction.
We can now describe how to construct $B$ by gluing together the polytopes
\[
\{\sigma_x\,|\, \{x\}\in Strata(\X_0)\}.
\]
We will do this in the case that every irreducible component
of $\X_0$ is in fact itself normal so that $\nu:X_i\rightarrow \nu(X_i)$ is an
isomorphism. The reader may be able to imagine the more
general construction.
With this normality assumption, there is a one-to-one inclusion reversing
correspondence between faces of
$\sigma_x$ and elements of $Strata(\X_0)$ containing $x$. We can then
identify faces of $\sigma_x$ and $\sigma_{x'}$ if they correspond
to the same strata of $\X_0$. Some argument is necessary to show that this
identification can be done via an integral affine transformation, but
again this is not difficult.

Making these identifications, one obtains $B$. One can then prove

\begin{lemma} If $\X_0$ is complex $n$-dimensional, then $B$ is an
real $n$-dimensional manifold.
\end{lemma}

See \cite{PartI}, Proposition 4.10 for a proof.

Now so far $B$ is just a topological manifold, constructed by gluing together
lattice polytopes. Let 
\[
\P=\{\sigma\subseteq B| \hbox{$\sigma$ is a face of $\sigma_x$ 
for some zero-dimensional stratum $x$}\}.
\]
There is a one-to-one inclusion reversing correspondence between strata
of $\X_0$ and elements of $\P$. 

It only remains to give $B$ an affine structure with singularities.
In fact, I shall describe somewhat more structure on $B$ derived
from $\X_0$ which in particular gives an affine structure with 
singularities on $B$.

First, for $\tau\in\P$, let
\[
U_{\tau}:=\bigcup_{\{\sigma\in\P\,|\,\tau\subseteq\sigma\}}\Int(\sigma).
\]
A \emph{fan structure  along $\tau\in\P$} is a continuous map $S_{\tau}:
U_{\tau}\rightarrow\RR^k$ such that
\begin{enumerate}
\item $S_{\tau}^{-1}(0)=\Int(\tau)$.
\item If $e:\tau\rightarrow\sigma$ is an inclusion then $S_{\tau}|_{\Int\sigma}$
is an integral affine submersion onto its image.
\item The collection of cones
\[
\{K_e:=\RR_{\ge 0} S_{\tau}(\sigma\cap U_{\tau})\,|\, e:\tau\rightarrow\sigma\}
\]
defines a finite fan $\Sigma_{\tau}$ in $\RR^k$.
\end{enumerate}
Two fan structures $S_{\tau},S'_{\tau}:U_{\tau}\rightarrow\RR^k$ are
considered \emph{equivalent} if they differ only by an integral linear
transformation of $\RR^k$.

If $S_{\tau}:U_{\tau}\rightarrow\RR^k$ is a fan structure along $\tau\in\P$
and $\sigma\supseteq \tau$ then $U_{\sigma}\subseteq U_{\tau}$.
The fan structure along $\sigma$ \emph{induced by $S_{\tau}$} is the
composition
\[
U_{\sigma}\mapright{} U_{\tau} \mapright{S_{\tau}}\RR^k\mapright{}
\RR^k/L_{\sigma}\cong \RR^{\ell}
\]
where $L_{\sigma}\subseteq\RR^k$ is the linear span of $S_{\tau}(\sigma)$.

\begin{definition}
An \emph{integral tropical manifold of dimension $n$} is a pair $(B,\P)$
as above along with a choice of fan structure $S_v$ at each vertex
$v$ of $\P$, with the property that if $v,w\in\tau$, then the fan structures
along $\tau$ induced by $S_v$ and $S_w$ are equivalent.
\end{definition}

Such data gives $B$ the structure of an integral affine manifold with
singularities. Let $\Gamma\subseteq B$ be the union of those cells of
$\Bar(\P)$ (the first barycentric subdivision of $\P$) which are not contained
in maximal cells of $\P$ nor contain vertices of $\P$. Then $B_0:=B\setminus
\Gamma$ can be covered by
\[
\{\Int(\sigma)\,|\,\sigma\in\P_{\max}\}\cup \{W_v\,|\,
\hbox{$v\in\P$ a vertex}\}
\]
for certain open neighbourhoods $W_v$ of $v\in B$ contained in $U_v$.
We define an affine structure
on $B_0$ by giving $\Int(\sigma)$ the natural affine structure given by
$\sigma$ being a lattice polytope, while $S_v:U_v\rightarrow \RR^n$
restricts to an affine chart on $W_v$.

Finally, the point is that the structure of $\X_0$ gives rise
to an integral tropical manifold structure on $(B,\P)$. Indeed, each
vertex $v\in \P$ corresponds to an irreducible component $X_v$ of 
$\X_0$ and this
irreducible component is a toric variety with fan $\Sigma_v$ in $\RR^n$.
Furthermore, there is a one-to-one
correspondence between $p$-dimensional
cones of $\Sigma_v$ and $p$-dimensional cells
of $\P$ containing $v$ as a vertex, as they both correspond to 
strata of $\X_0$ contained in $X_v$. There is then a continuous map
\[
\psi_v:U_v\rightarrow \RR^n
\]
which takes $U_v\cap\sigma$, for any $\sigma\in\P$ containing $v$ as a vertex,
into the corresponding cone of $\Sigma_v$
\emph{integral affine linearly}. Such a map is uniquely determined
by the combinatorial correspondence and the requirement that it
be integral affine linear on each cell. These maps define a fan structure
at each vertex. Furthermore, these fan structures are compatible in
the sense that if $v,w\in\tau$, the two induced fan structures on $U_{\tau}$
are equivalent. This follows because there is a well-defined fan
$\Sigma_{\tau}$ defining the stratum corresponding to $\tau$.

\begin{example} Let $f:\X\rightarrow D$ be a degeneration of elliptic
curves to an $I_n$ fibre. Then $B$ is the circle $\RR/n\ZZ$, decomposed
by $\P$ into $n$ line segments of length one.
\end{example}

\begin{example}
Continuing with Example \ref{quarticexample}, the dual intersection
complex is the boundary of a tetrahedron, with each face affine isomorphic
to a standard two-simplex, and the affine structure near each vertex 
makes the polyhedral decomposition look locally like the fan for $\PP^2$.
There is one singularity at the barycenter of each edge, and one can calculate
that the monodromy of $\Lambda$ about each of these singularities is
$\begin{pmatrix} 1&4\\ 0&1\end{pmatrix}$ in a suitable basis.
\end{example}

\begin{example}
Consider the polytope $\Delta$ of Example \ref{quintic}. The dual polytope
$\nabla$ is the convex hull of the points $(-1,-1,-1,-1),(1,0,0,0),\ldots,
(0,0,0,1)$. The corresponding projective toric variety $\PP_{\nabla}$
has a crepant resolution $X_{\Sigma}\rightarrow\PP_{\nabla}$ where $\Sigma$
is the fan consisting of cones over all elements of the decomposition
$\P$ of $\partial\Delta$ as described in Example \ref{quintic}. 
Consider in $\PP_{\nabla}\times\AA^1$ the
degenerating family $\shX\rightarrow\AA^1$ of Calabi-Yau manifolds given by
\[
s_0+t\sum_{m\in \nabla\cap \ZZ^4} c_ms_m=0
\]
where $s_m$ is the section of $\shO_{\PP_{\nabla}}(1)$ corresponding
to $m\in\nabla\cap\ZZ^4$. Let $\widetilde\shX$ be the proper transform
of $\shX$ in $X_{\Sigma}\times\AA^1$. 
Then the family $\widetilde\shX\rightarrow \AA^1$
is a toric degeneration with general fibre the mirror quintic, and its 
dual intersection complex is the affine manifold $B$ constructed in
Example \ref{quintic}.
\end{example}

Is the dual intersection complex the right affine manifold with singularities?
The following theorem provides evidence for this, and gives the connection
between this construction and the SYZ conjecture.

\begin{theorem}
\label{complextheorem}
Let $\X\rightarrow D$ be a toric degeneration, with dual intersection
complex $(B,\P)$. Then there is an open set $U\subseteq B$ such that 
$B\setminus U$ retracts onto the discriminant locus $\Gamma$ of $B$,
and an open subset $\U_t$ of $\shX_t$ which is biholomorphic
to a small deformation of a twist of $X_{\epsilon}(U)$,
where $\epsilon=O(-1/\ln|t|)$. 
\end{theorem}

We will not be precise here about what we mean by small deformation;
by twist, we mean a twist of the complex structure of $X_{\epsilon}(U)$
by a $B$-field. See \cite{Announce} 
for a much more precise statement; the above statement
is meant to give a feel for what is true. The proof, along with much more
precise statements, will eventually appear in \cite{tori}.

\medskip

If $\X\rightarrow D$ is a \emph{polarized} toric degeneration, i.e., if there
is a relatively ample line bundle $\shL$ on $\X$, then we can construct
another integral tropical manifold
$(\check B,\check\P)$, which we call the \emph{intersection complex},
as follows.

For each irreducible component $X_i$ of $\X_0$, $\shL|_{X_i}$ is an ample
line bundle on a toric variety. Let $\check\sigma_i\subseteq N_{\RR}$
denote the Newton polytope of this line bundle. There is then
a one-to-one inclusion preserving correspondence between strata of
$\X_0$ contained in $X_i$ and faces of $\check\sigma_i$. We can then glue together
the $\check\sigma_i$'s in the obvious way: if $Y$ is a codimension one
stratum of $\X_0$, it is contained in two irreducible components $X_i$ and
$X_j$, and defines faces of $\check\sigma_i$ and $\check\sigma_j$. 
These faces are affine
isomorphic because they are both the Newton polytope of $\shL|_Y$, and
we can then identify them in the canonical way. Thus we obtain a topological
space $\check B$ with a polyhedral decomposition $\check\P$. 

To define the fan structure at a vertex $v\in\P$, note that such a vertex
corresponds to a zero-dimensional stratum of $\X_0$, giving rise
to a maximal cell $\sigma_v$ of the dual intersection complex. Take the
fan structure at $v$ to be defined using the normal fan 
$\check\Sigma_v$ to $\sigma_v$. Then there is
a one-to-one inclusion preserving correspondence between cones in
$\check\Sigma_v$ and strata of $\X_0$ containing the stratum corresponding
to $v$. This correspondence allows us to define a fan structure
\[
S_v:U_v\rightarrow \RR^n
\]
which takes $U_v\cap\check\sigma$, 
for any $\check\sigma\in\check\P$ containing $v$ as a vertex,
into the corresponding cone of $\check\Sigma_v$. One checks easily that
this set of fan structures satisfies the definition of integral tropical
manifold, 
and hence defines the intersection complex $(\check B,\check\P)$.

Analogously to Theorem \ref{complextheorem}, we expect

\begin{conjecture}
Let $\X\rightarrow D$ be a polarized toric degeneration, 
with intersection complex
$(\check B,\check\P)$. 
Let $\omega_t$ be a K\"ahler form on $\X_t$ representing the first Chern
class of the polarization.
Then there is an open set $\check U\subseteq \check B$ 
such that $\check B\setminus \check U$ retracts onto the discriminant locus
$\Gamma$ of $\check B$, such that $\X_t$ is a symplectic compactification
of $\check X(\check U)$ for any $t$.
\end{conjecture}

I don't expect this to be particularly difficult: it should be amenable
to the techniques of W.-D.\ Ruan \cite{RuanJSG}, but such an approach
has not been carried out in general.

\medskip

The relationship between the intersection complex and the dual intersection
complex can be made more precise by introducing multi-valued piecewise
linear functions, in analogy with the multi-valued convex functions of
Definition \ref{multivaluedconvex}.

\begin{definition}
Let $(B,\P)$ be an integral tropical manifold.
Then a \emph{multi-valued piecewise linear function}
$\varphi$ on $B$ is a collection of continuous functions on an open cover
$\{(U_i,\varphi_i)\}$ such that $\varphi_i$ is affine linear on each
cell of $\P$ intersecting $U_i$, and on $U_i\cap U_j$, $\varphi_i-\varphi_j$
is affine linear. Furthermore, for any $\tau\in\P$, let $S_{\tau}:U_{\tau}
\rightarrow \RR^k$ be the induced fan structure. Then there is a piecewise
linear function $\varphi_{\tau}$ on the fan $\Sigma_{\tau}$ such that
on $U_i\cap U_{\tau}$, $\varphi_i-\varphi_{\tau}\circ S_{\tau}$ is affine
linear. Here we will always assume that each linear part of $\varphi_i$
has differential in $\check\Lambda$, i.e., $\varphi_i$ has integral slopes.
\end{definition}

The rather technical condition on the local behaviour of each $\varphi_i$
on $U_{\tau}$ comes from the idea that such a multi-valued piecewise linear
function is really just a collection of piecewise linear functions on
the fans $\Sigma_{\tau}$ given by the fan structure of $(B,\P)$. These
functions need to satisfy some compatibility conditions, and this compatibility
is motivated by the following discussion.

Suppose we are given a 
polarized toric degeneration $\X\rightarrow D$. We in fact obtain a 
multi-valued piecewise
linear function $\varphi$ on the dual intersection complex $(B,\P)$
as follows. Restricting to any toric stratum $X_{\tau}$, $\shL|_{X_{\tau}}$
is determined completely by an integral piecewise linear function
$\varphi_{\tau}$ on $\Sigma_{\tau}$, well-defined up to
a choice of linear function. Pulling back this piecewise linear function
via $S_{\tau}$ to $U_{\tau}$, we obtain a collection of piecewise
linear functions $\{(U_{\tau},\varphi_{\tau}\circ S_{\tau})\,|\,\tau\in\P\}$.
The fact that $(\shL|_{X_{\tau}})|_{X_{\sigma}}=\shL|_{X_{\sigma}}$
for $\tau\subseteq\sigma$ implies that on overlaps $\varphi_{\sigma}
\circ S_{\sigma}$
and $\varphi_{\tau}\circ S_{\tau}$ differ by at most an affine linear function. 
So $\{(U_{\tau},
\varphi_{\tau}\circ S_{\tau})\}$ 
defines a multi-valued piecewise linear function.
The last condition in the definition of multi-valued piecewise linear
function then reflects the need for the function to be locally
a pull-back of a function via $S_{\sigma}$ in a neighbourhood of $\sigma$.

If $\shL$ is ample, then the piecewise linear function determined by
$\shL|_{X_{\sigma}}$ is strictly convex. So we say a multi-valued piecewise
linear function is \emph{strictly convex} if $\varphi_{\tau}$ is
strictly convex for each $\tau\in\P$. 

As a consequence, if $\shX\rightarrow D$ is a polarized toric degeneration,
we will write $(B,\P,\varphi)$ for the data of the dual intersection
complex and the induced multi-valued function $\varphi$. We call this
triple the dual intersection complex of the polarized degeneration.

\medskip

Now suppose we are given abstractly a triple $(B,\P,\varphi)$ with $(B,\P)$
an integral tropical manifold and $\varphi$ a strictly convex multi-valued
piecewise linear function on $B$. Then we construct the \emph{discrete
Legendre transform} $(\check B,\check\P,\check\varphi)$ of $(B,\P,\varphi)$
as follows.

$\check B$ will be constructed by gluing together Newton polytopes.
If we view, for $v$ a vertex of $\P$, the fan $\Sigma_v$ as living
in $M_{\RR}$, then the Newton polytope of $\varphi_v$ is
\[
\check v:=\{x\in N_{\RR}\,|\,\langle x,y\rangle\ge-\varphi_v(y)
\quad\forall y\in M_{\RR}\}.
\]
There is a one-to-one inclusion reversing correspondence between faces of
$\check v$ and cells of $\P$ containing $v$. Furthermore, if
$\sigma$ is the smallest cell of $\P$ containing two vertices $v$ and $v'$,
then the corresponding faces of $\check v$ and $\check v'$ are integral 
affine isomorphic, as they are both isomorphic to the Newton polytope
of $\varphi_{\sigma}$. Thus we can glue $\check v$ and $\check v'$
along this common face. After making all these identifications, we obtain a cell
complex $(\check B,\check\P)$, which is really just the dual cell complex
of $(B,\P)$. This is given an integral tropical structure by taking the
fan structure at a vertex $\check\sigma$, for $\sigma\in\P_{\max}$,
to be given by the normal fan to $\sigma$.

Finally, the function $\varphi$ has a discrete Legendre transform
$\check\varphi$ on $(\check B,\check\P)$. We have no choice but to
define $\check\varphi$ in a neighbourhood of a vertex $\check\sigma\in
\check\P$ dual to a maximal cell $\sigma\in\P$ to be a piecewise
linear function whose Newton polytope is $\sigma$, i.e.,
\[
\check\varphi_{\check\sigma}(y)
=-\inf\{\langle y,x\rangle\,|\,x\in\sigma\subseteq M_{\RR}\}.
\]
This gives $(\check B,\check\P,\check\varphi)$, the discrete
Legendre transform of $(B,\P,\varphi)$. If $B$ is $\RR^n$, then
this coincides with the classical notion of discrete Legendre
transform. The discrete Legendre transform has several relevant 
properties:
\begin{itemize}
\item The discrete Legendre transform of $(\check B,\check\P,\check\varphi)$
is $(B,\P,\varphi)$.
\item If we view the underlying topological spaces $B$ and $\check B$
as identified by being the underlying space of dual cell complexes,
then $\Lambda_{B_0}\cong \check\Lambda_{\check B_0}$ and
$\check\Lambda_{B_0}\cong\Lambda_{\check B_0}$, where the subscript
denotes which affine structure is being used to define $\Lambda$ or
$\check\Lambda$. 
\end{itemize}

This hopefully makes it clear that the discrete Legendre transform
is a suitable replacement for the duality provided by the
Legendre transform of \S 2.

\medskip

Note in particular that if $\shX\rightarrow D$ is a polarized toric 
degeneration, with dual intersection complex $(B,\P,\varphi)$, then
the discrete Legendre transform $(\check B,\check\P,\check\varphi)$
satisfies the condition that $(\check B,\check\P)$ is the intersection
complex of the polarized degeneration. The function $\check\varphi$
is some extra information on $\check B$, which from the definition
of discrete Legendre transform encodes the cells of $\P$. These cells
of the dual intersection complex were defined using the local toric
structure of $\shX\rightarrow D$.
So $\check\varphi$ can be seen as carrying information
about this local toric structure. 
We will say $(\check B,\check\P,\check\varphi)$
is the intersection complex of the polarized toric degeneration
$\shX\rightarrow D$.

So we see that for $(B,\P,\varphi)$, $\P$ carries information about the
log structure and $\varphi$ carries information about the polarization,
but for $(\check B,\check\P,\check\varphi)$, $\check\P$ carries information
about the polarization and $\check\varphi$ carries information about
the log structure. Mirror symmetry interchanges these two pieces of
information!

We can now state an \emph{algebro-geometric
SYZ procedure}. In analogy with the procedure suggested in \S 5, we
could follow these steps:
\begin{enumerate}
\item We begin with a toric degeneration of Calabi-Yau
manifolds $\X\rightarrow D$ with an ample polarization.
\item Construct the dual intersection complex $(B,\P,\varphi)$ 
from this data, as explained above.
\item Perform the discrete Legendre transform to obtain $(\check B,
\check\P,\check\varphi)$.
\item Try to construct a polarized degeneration of Calabi-Yau
manifolds $\check\X\rightarrow D$ whose dual intersection
complex is $(\check B,\check\P,\check\varphi)$, or whose intersection
complex is $(B,\P,\varphi)$.
\end{enumerate}

\begin{example} The discrete Legendre transform enables us to reproduce 
Batyrev duality \cite{Bat}. Let $\Delta\subseteq M_{\RR}$ be a reflexive polytope,
$\nabla\subseteq N_{\RR}$ the polar dual, and assume $0\in\Delta$ is
the unique interior point. We then obtain two toric degenerations
given by the equations
\[
s_0+t\sum_{m\in M\cap\Delta} c_ms_m=0,\quad\quad s_0+t\sum_{n\in N\cap
\nabla} c_ns_n=0
\]
in $\PP_{\Delta}\times\AA^1$ and $\PP_{\nabla}\times\AA^1$ respectively,
with $s_m$ ($s_n$) the section of $\shO_{\PP_{\Delta}}(1)$ corresponding
to $m$ (the section of $\shO_{\PP_{\nabla}}(1)$ corresponding to $n$).
It is easy to check that the dual intersection complexes of these
two degenerations are given as follows. For the first degeneration, 
$B=\partial\nabla$ with polyhedral decomposition given by the proper faces of
$\nabla$. The fan structure at each vertex $v$ is given by projection
$U_v\hookrightarrow N_{\RR}\rightarrow N_{\RR}/\RR v$. For the second
degeneration, one uses $\Delta$ instead of $\nabla$. One can then check
that if one polarizes the two degenerations using $\shO_{\PP_{\Delta}}(1)$
and $\shO_{\PP_{\nabla}}(1)$ respectively, then the corresponding triples
$(B,\P,\varphi)$ are Legendre dual. Thus Batyrev duality is a special
case of this general approach to a mirror construction.

For a much more general construction which works for the Batyrev-Borisov
construction \cite{BB} of mirrors of complete intersection Calabi-Yaus in toric
varieties, see \cite{GBB}.
\end{example}

The only step missing in this mirror symmetry algorithm is the last: 
\begin{question}[The reconstruction problem, Version II]
\label{reconstruct2}
Given $(B,\P,\varphi)$, is it possible to construct a polarized toric
degeneration $\X\rightarrow D$ whose intersection complex is
$(B,\P,\varphi)$?
\end{question}

One could hope to solve this problem via naive deformation theory,
by constructing the central fibre $\X_0$
from the data $(B,\P,\varphi)$, and
then deforming this to find a smoothing. However, 
as initially observed in the normal crossings case by Kawamata and
Namikawa in \cite{KN}, one needs to put some additional structure on $\X_0$
before it has good deformation theory. This structure is a \emph{log
structure}, and introducing log structures allows us to study
many aspects of mirror symmetry directly on the degenerate fibre itself.
So let us turn to a review of the theory of logarithmic structures.

\section{Log structures}

We review the notion of  log structures of Fontaine-Illusie and Kato
(\cite{Illu}, \cite{K.Kato}). These play a key role in trying to understand
mirror symmetry via degenerations.

\begin{definition}
A log structure on a scheme (or analytic space) $X$ is a (unital) homomorphism
$$\alpha_X:\shM_X\rightarrow \O_X$$
of sheaves of (multiplicative and commutative) monoids inducing an isomorphism
$\alpha_X^{-1}(\O_X^{\times})\rightarrow \O_X^{\times}$. The monoid
structure on $\O_X$ is given by multiplication. The
triple $(X,\shM_X,\alpha_X)$ is then called a {\it log space}.
We often write the whole package as $X^{\dagger}$.
\end{definition}

A morphism of log spaces $F:X^{\dagger}\rightarrow Y^{\dagger}$ consists
of a morphism $\underline{F}:X\rightarrow Y$ of underlying
spaces together with a homomorphism $F^{\#}:\underline{F}^{-1}(\shM_Y)
\rightarrow\shM_X$ commuting with the structure homomorphisms:
$$\alpha_X\circ F^{\#}=\underline{F}^*\circ\alpha_Y.$$

The key examples:

\begin{examples}
\label{logexamples}
(1) Let $X$ be a scheme and $Y\subseteq X$ a closed subset of
codimension one. Denote by $j:X\setminus Y\rightarrow X$
the inclusion. Then the inclusion
$$\alpha_X:\shM_X=j_*(\O_{X\setminus Y}^{\times})\cap\O_X\rightarrow
\O_X$$
of the sheaf of regular functions invertible off of $Y$ is a log
structure on $X$. This is called a \emph{divisorial log structure} on
$X$.

(2) A {\it prelog structure}, i.e., an arbitrary homomorphism of
sheaves of monoids $\varphi:\shP\rightarrow\O_X$, defines
an associated log structure $\shM_X$ by
$$\shM_X=(\shP\oplus\O_X^{\times})/\{(p,\varphi(p)^{-1})\,|\,p\in
\varphi^{-1}(\O_X^{\times})\}$$
and $\alpha_X(p,h)=h\cdot\varphi(p)$.

(3) If $f:X\rightarrow Y$ is a morphism of schemes and $\alpha_Y:\shM_Y
\rightarrow\O_Y$ is a log structure on $Y$, then the prelog structure
$f^{-1}(\shM_Y)\rightarrow\O_X$ given as the composition of
$\alpha_Y:f^{-1}(\shM_Y)\rightarrow f^{-1}\shO_Y$ and $f^*:f^{-1}\shO_Y
\rightarrow \shO_X$ defines an associated log structure
on $X$, the {\it pull-back log structure}.

(4) In (1) we can pull back the log structure on $X$ to $Y$ using
(3). Thus in particular, if $\X\rightarrow D$ is a toric
degeneration, the inclusion $\X_0\subseteq\X$ gives a log
structure on $\X$ and an induced log structure on $\X_0$. Similarly
the inclusion $0\in D$ gives a log structure on $D$ and
an induced one on $0$. Here $\M_0=\CC^{\times}\oplus\NN$,
where $\NN$ is the (additive) monoid of natural (non-negative) numbers,
and
$$\alpha_0(h,n)=\begin{cases}h& n=0\\ 0&n\not=0.\end{cases}$$
$0^{\dagger}$ is usually called the \emph{standard log point}.

We then have log morphisms $\X^{\dagger}\rightarrow D^{\dagger}$ and
$\X_0^{\dagger}\rightarrow 0^{\dagger}$.

(5) If $\sigma\subseteq M_{\RR}=\RR^n$ is a
strictly convex rational polyhedral cone, $\dual{\sigma}\subseteq
N_{\RR}$ the dual cone, let $P=\dual{\sigma}\cap N$: this is a monoid
under addition.
The affine toric variety defined by $\sigma$ can be written as
$X=\Spec \CC[P]$. 
We then have a pre-log structure induced by the homomorphism of
monoids
$$P\rightarrow \CC[P]$$
given by $p\mapsto z^p$. There is then an associated log
structure on $X$. This is in fact the same as the log structure
induced by $\partial X\subseteq X$, where $\partial X$
is the toric boundary of $X$, i.e., the union of toric divisors of $X$.

If $p\in P$, then the monomial $z^p$ defines a map
$f:X\rightarrow \Spec \CC[\NN]=\AA^1$ which is a log morphism
with the log structure on $\Spec \CC[\NN]$ induced similarly by
$\NN\rightarrow\CC[\NN]$.
The fibre $X_0=\Spec \CC[P]/(z^p)$ is a subscheme of $X$,
there is an induced log structure on $X_0$, and a map $X_0^{\dagger}
\rightarrow 0^{\dagger}$ as in (4). The log morphism $f$ is an example of a 
\emph{log smooth} morphism, see Definition \ref{logsmooth}.

Condition (4) of Definition \ref{toricdegen} in fact implies
that locally, away from $Z$, $\X^{\dagger}$ and $\X_0^{\dagger}$ are
of the above form. So we should view $\X^{\dagger}\rightarrow D^{\dagger}$
as log smooth away from $Z$, and from the log point of view, $\X_0^{\dagger}$
can be treated much like a non-singular scheme away from $Z$. 

(6) Given a monoid $P$ as in (5) and a morphism
$X\rightarrow \Spec\CC[P]$, we can pull back the log structure defined
above on $\Spec\CC[P]$ to $X$. If $X^{\dagger}$ is a log scheme which
\'etale locally can be described in this way, we say $X^{\dagger}$ is a 
\emph{fine saturated log scheme}. The adjective ``fine'' tells us it is locally
described via maps to schemes of the form $\Spec\CC[P]$ where $P$ is
a finitely generated 
\emph{integral} monoid, i.e., the canonical homomorphism $P\rightarrow
P^{\gp}$ is an injection.
The adjective ``saturated'' tells us the monoid $P$ is saturated.
This means that $P$ is integral and whenever
$p\in P^{\gp}$ satisfies $mp\in P$ for some $m>0$, $p\in P$.
Such monoids arise, e.g., as
the intersection of a rational polyhedral cone with a lattice.

Most of the literature on log geometry tends to apply only to 
fine log structures.
In the key example of $\shX_0^{\dagger}\rightarrow 0^{\dagger}$,
the log structure is fine saturated away from the set $Z$. However, it is not
in general fine
along $Z$, and this tends to cause many technical problems as new techniques
have to be developed to deal properly with the log structure along $Z$.
\qed
\end{examples}

The notion of log smoothness generalizes the morphisms of Examples
\ref{logexamples}, (5): 

\begin{definition}
\label{logsmooth}
A morphism $f:X^{\dagger}\rightarrow Y^{\dagger}$ of
fine log schemes is \emph{log smooth}
if \'etale locally on $X$ and $Y$
it fits into a commutative diagram
\[
\xymatrix@C=30pt
{
X\ar[r]\ar[d]&\Spec \ZZ[P]\ar[d]\\
Y\ar[r]&\Spec\ZZ[Q]
}
\]
with the following properties:
\begin{enumerate}
\item The canonical log structure on $\Spec \ZZ[P]$ and
$\Spec\ZZ[Q]$ of Examples \ref{logexamples}, (5), pull-back to the
log structures on $X$ and $Y$ respectively.
\item The induced morphism
\[
X\rightarrow Y\times_{\Spec\ZZ[Q]}\Spec\ZZ[P]
\]
is a smooth morphism of schemes.
\item The right-hand vertical arrow is induced by a monoid homomorphism
$Q\rightarrow P$ with $\ker(Q^{\gp}\rightarrow P^{\gp})$ and the
torsion part of $\coker(Q^{\gp}\rightarrow P^{\gp})$ finite groups
of orders invertible on $X$. Here $P^{\gp}$ denotes the Grothendieck
group of $P$.
\end{enumerate}
\end{definition}

Log smooth morphisms include, in the simplest case, normal crossings
morphisms.

\medskip

On a log scheme $X^{\dagger}$ there is always an exact sequence
\[
1\mapright{} \O_X^{\times}\mapright{\alpha^{-1}}\M_X\mapright{}
\overline{\M}_X\mapright{}0,
\]
where we write the quotient sheaf of monoids $\overline{\M}_X$
additively. We call $\overline{\M}_X$ the \emph{ghost sheaf}
of the log structure. I like to view $\overline{\M}_X$ as specifying
the combinatorial information associated to the log structure. For 
example, if $X^{\dagger}$ is induced by the Cartier divisor $Y\subseteq
X$ with $X$ normal, 
then the stalk $\overline{\M}_{X,x}$ at $x\in X$ is the monoid 
of effective Cartier divisors on a neighbourhood of $x$ supported
on $Y$.

It is useful for understanding pull-backs of log structures
to note that if $f:Y\rightarrow X$ is a morphism
with $X$ carrying a log structure, and $Y$ is given the pull-back log
structure, then $\overline{\M}_Y=f^{-1}\overline{\M}_X$. In the case
that $\M_X$ is induced by an inclusion
of $Y\subseteq X$, $\overline{\M}_X$ is supported on $Y$, so we can
equate $\overline{\M}_X$ and $\overline{\M}_Y$, the ghost sheaves
for the divisorial log structure on $X$ and its restriction to $Y$.

\begin{xca}
\label{ghostexercise}
Show that in Example \ref{logexamples}, (5), $\overline{\M}_{X,x}=
P$ if $\dim\sigma=\dim M_{\RR}$ and $x$ is the unique zero-dimensional torus 
orbit
of $X$. More generally, 
\[
\overline{\M}_{X,x}={\dual{\tau}\cap N\over \tau^{\perp}\cap N}
=\Hom_{monoid}(\tau\cap M,\NN),
\]
when $x\in X$ is in the torus orbit 
corresponding to a face $\tau$ of $\sigma$. In particular, $\tau$
can be recovered as $\Hom_{monoid}(\overline{\M}_{X,x},\RR_{\ge 0})$,
where $\RR_{\ge 0}$ is the additive monoid of non-negative real
numbers.
\qed
\end{xca}

In the sections which follow, the key logarithmic spaces we consider
will be those arising from toric degenerations $\shX\rightarrow D$.
As above, the central fibre $\shX_0\subseteq\shX$ induces a divisorial
log structure on $\shX$, and restricting gives a log scheme 
$\shX_0^\dagger$ along with a morphism $\shX_0^{\dagger}\rightarrow
0^{\dagger}$ which is log smooth off of the bad set $Z\subseteq \shX_0$.

We can now elaborate on the philosophy we wish to take with the following
diagram:
\begin{center}
\input{philo1.pstex_t}
\end{center}
There are two sides to mirror symmetry. 
The $A$-model side involves counting curves: we wish to count curves
in the general fibre of a toric degeneration $\shX\rightarrow D$.
There are good reasons to believe that this count can in fact be performed
on $\shX_0^{\dagger}$, using a theory of logarithmic Gromov-Witten
invariants: see \S\ref{Amodelsection}. 
The hope is that $\shX_0$ is a sufficiently
combinatorial object so that such a count can be carried out in a combinatorial
manner.

The $B$-side involves deformations of complex structure. The idea is
that to understand deformations of complex structure, we should start
with the central fibre $\shX_0^{\dagger}$ and try to construct
smoothings, i.e., construct a toric degeneration with this central fibre.
The log structure is necessary to find a unique smoothing. 
If this smoothing can be described sufficiently explicitly, 
then again one should be able
to extract the necessary periods for the $B$-model calculations purely
in terms of combinatorics.

So log geometry will play an important role on both sides of mirror
symmetry, but as the above suggests, there should be some combinatorial
objects underlying both calculations.

In fact, log geometry is closely related to tropical geometry. We will
explore in the following sections how tropical geometry controls both
the $A$- and $B$-model sides of the above picture, completing the above 
diagram:
\bigskip
\begin{center}
\input{philo2.pstex_t}
\end{center}

\section{The $A$-model and tropical geometry}
\label{Amodelsection}

The first link between log geometry and tropical geometry comes from an
elementary combinatorial construction. Given a log scheme $X^{\dagger}$, 
we can construct the \emph{tropicalization} of $X^{\dagger}$, as follows.
For each geometric point $\bar\eta$ of $X$, we have a monoid 
$\overline\M_{X,\bar\eta}$, and hence a cone $C_{\bar \eta}:=
\Hom(\overline\M_{X,\bar\eta},\RR_{\ge 0})$ (where here $\Hom$ denotes
monoid homomorphisms and $\RR_{\ge 0}$ is given the additive monoid
structure). Further, if $\bar\eta$ is in the closure of $\bar\eta'$, there
is a generization map $\overline\M_{X,\bar\eta}\rightarrow
\overline\M_{X,\bar\eta'}$.\footnote{Since we need to work in the
\'etale topology, there can actually be a number of generization maps.
For example, if $X$ is a nodal cubic, then there are two generization maps
from $\bar\eta$ the node to $\bar\eta'$ the generic point.} 
Dualizing, this gives maps $C_{\bar\eta'}\rightarrow C_{\bar\eta}$.
If the log structure on $X$ is fine,
then these maps are inclusions of faces of strictly
convex rational polyhedral cones. We can then form a cell complex by
making identifications given by these inclusions of faces, obtaining
a polyhedral cone complex $\Trop(X^{\dagger})$. Actually, in general this 
may not really make sense as a cell complex because the generization maps
may induce many strange self-identifications on faces, but in the situations
we want to describe here, this will not cause a problem.

This construction is functorial, so if $f:X^{\dagger}\rightarrow Y^{\dagger}$
is a morphism of log schemes, then we obtain $\Trop(f):
\Trop(X^{\dagger})\rightarrow \Trop(Y^{\dagger})$.

For example, consider the case of a toric degeneration $\shX\rightarrow D$.
As we saw in the previous section, this gives a morphism of log schemes
$\shX_0^{\dagger}\rightarrow 0^{\dagger}$. The bad set $Z\subseteq
\shX_0$ is precisely the locus where the log structure on $\shX_0$ is
not fine. Thus we can apply the above 
tropicalization construction to $\shX_0^{\dagger}\setminus Z\rightarrow
0^{\dagger}$. Now $\Trop(0^{\dagger})=\RR_{\ge 0}$ is a ray. On the other hand,
if $x\in\shX_0$ is a zero-dimensional stratum, locally a neighbourhood
of $x$ looks like the fibre over $0$ of $f_x:Y_x\rightarrow\CC$ where 
$Y_x$ is a toric variety defined by $C(\sigma_x)\subseteq
M_{\RR}\oplus\RR$ for a lattice polytope
$\sigma_x\subseteq M_{\RR}$, and the morphism $Y_x\rightarrow\CC$ is given
by the projection $M_{\RR}\oplus\RR\rightarrow\RR$. Then 
$\overline\M_{\shX_0,x}=C(\sigma_x)^{\vee}\cap (N\oplus \ZZ)$, as
follows from Exercise \ref{ghostexercise}, and $C_x=C(\sigma_x)$.
In particular, the induced map $\Trop(f):
C_x\rightarrow \Trop(0^{\dagger})$ has fibre $\Trop(f)^{-1}(1)=\sigma_x$.
From this, one checks easily that 
\[
\Trop(f):\Trop(\shX_0^{\dagger}\setminus
Z)\rightarrow \Trop(0^{\dagger})=\RR_{\ge 0}
\]
has fibre
\[ 
\Trop(f)^{-1}(1)=B,
\]
with $B$ coming with the polyhedral decomposition $\P$. So the dual
intersection complex comes from a very general construction. In particular,
note that $B$ only depends on $\shX_0^{\dagger}$, not on $\shX$ (although
this is obvious without knowing this general construction).

Now let us turn to the $A$-model, which for the purposes of this discussion
means counting curves on Calabi-Yau manifolds. Suppose we have a toric
degeneration $\shX\rightarrow D$. We would like to count curves on the general
fibre. Can we do so by counting curves on $\shX_0$ instead, where the
problem might have a more combinatorial nature?

This question has a long history. The first work on this kind of question
was due to Li and Ruan \cite{LR} and Ionel and Parker \cite{IP1},\cite{IP2}. 
Essentially they
considered a situation where one has a degeneration $\shX\rightarrow D$
where the special fibre $\shX_0=X_1\cup X_2$ is a normal crossings
union of two smooth irreducible components. They showed that there
was a theory of Gromov-Witten invariants of $\shX_0$, and that it
gave the same answer as Gromov-Witten theory on a general fibre. Further,
they gave gluing formulas, which stated that the Gromov-Witten invariants
of $\shX_0$ could be computed using the Gromov-Witten invariants of
the two pairs $(X_i,X_1\cap X_2)$, $i=1,2$. Here the Gromov-Witten theory
associated to a pair $(X,D)$ where $D\subseteq X$ is a smooth divisor
is the theory of \emph{relative} Gromov-Witten invariants, where one
considers curves in $X$ with some imposed orders of tangency at points on
the curve with $D$. This gluing formula has proven to be a very powerful
tool in Gromov-Witten theory.

In 2001, Bernd Siebert \cite{STalk} proposed using log geometry to generalize
these results. Meanwhile, Jun Li was working on an algebro-geometric
approach to the Li-Ruan and Ionel-Parker theories (which were carried
out using symplectic techniques). He gave a satisfactory algebro-geometric
definition of relative Gromov-Witten invariants and reproved the gluing
formula, using a few techniques from log geometry. However, the theory
possesses a technical difficulty. In Gromov-Witten theory, it is standard
that one allows the domain curves to develop bubbles. But in relative
Gromov-Witten theory, it is also necessary to allow the target space
$X$ to develop bubbles. This occurs when an irreducible component of
the domain curve falls into the divisor $D$, so that the order of tangency
with $D$ becomes meaningless. So the actual target space for a relative stable
map might be $X$ with a chain of $\PP^1$-bundles over $D$ glued to $D\subseteq 
X$. This often
makes the analysis more difficult, and was a major stumbling block for
extending these techniques to more complicated degenerations.

Several solutions to this problem were completed in 2011. Brett Parker
in \cite{Park1}, \cite{Park2} 
provided a completely new category, the category of
\emph{exploded manifolds}, in which to study Gromov-Witten theory. 
These manifolds carry
information similar to log spaces, but is a somewhat more flexible
and ``softer'' category in which to work. In \cite{Park2} 
he provides a definition
of Gromov-witten invariants in this setting and gives a gluing formula. 
Also,
Siebert and I \cite{JAMS} 
completed a theory of logarithmic Gromov-Witten invariants,
as did Abramovich and Chen \cite{Chen},\cite{AC}, 
working with Siebert's original
suggestion. I will summarize the basic ideas here.

\begin{definition} 
A \emph{log curve} over a fine saturated 
log scheme $W^{\dagger}$ is a fine saturated log scheme 
$C^{\dagger}$ with a
morphism $C^{\dagger} \rightarrow W^{\dagger} $ 
which is flat of relative dimension one, log smooth, and with all
geometric fibres reduced. 
\end{definition}

Here log smoothness implies that the geometric fibres of $C\rightarrow W$
are nodal curves, which is pleasant as this is precisely the sort of
curve which is allowed as the domain of a stable map. The log structure
can also be viewed as incorporating marked points. For example,
given a smooth curve $C$ over $W=\Spec \CC$, one can take a finite
number of points $x_1,\ldots,x_k\in C$ and give $C$ the divisorial
log structure associated to the subset $\{x_1,\ldots,x_k\}\subseteq C$.
Then $C^{\dagger}$ is log smooth over $W$ with the trivial log structure
$\M_W=\O_W^{\times}$.

\begin{definition}
Let $X^{\dagger}\rightarrow S^{\dagger}$ be a morphism of
fine saturated log schemes.
A \emph{log curve in $X^{\dagger}$ with base $W^{\dagger}$} is a log curve 
$C^{\dagger}/W^{\dagger}$
together with a morphism $f:C^{\dagger}\rightarrow X^{\dagger}$ 
fitting into a commutative diagram of log schemes
\[
\xymatrix@C=30pt
{ C^{\dagger}\ar[r]^f\ar[d] &X^{\dagger}\ar[d]\\
W^{\dagger}\ar[r]&S^{\dagger}}
\]
A log curve in $X^{\dagger}$ is 
a \emph{stable log map} if for every geometric point
$\bar w\rightarrow W$, the restriction of $f$ to the underlying
marked curve $C_{\bar w}\rightarrow\bar w$ is an ordinary stable map.
We write the data as $(C^{\dagger}/W^{\dagger},f)$.

This definition can be further decorated in the usual way by labelling
marked points.
\end{definition}

The main work of \cite{JAMS} is to construct a well-behaved moduli space
of stable log maps. There is a technical issue which arises whenever
one tries to construct a moduli space of log objects; this 
was explored by Martin Olsson in his thesis \cite{Ols}. The problem
is as follows. Suppose we are given a stable log map with domain
$\pi:(C,\M_C)\rightarrow (W,\M_W)$. Then $\pi':(C,\M_C\oplus \NN^r)\rightarrow
(W,\M_W\oplus\NN^r)$ also gives the domain of a stable log map. Here
the structure map $\alpha_C$ (or $\alpha_W$) takes the value $0$ on
the non-zero elements of the 
constant sheaf $\NN^r$, and the map $\pi'$ acting on monoids
just takes $\NN^r$ isomorphically to $\NN^r$. The new map $f^{\#}$
is the composition of the old $f^{\#}:f^{-1}\M_X\rightarrow \M_C$ and
the inclusion $\M_C\rightarrow \M_C\oplus\NN^r$.
As a result, a single stable log map gives rise to
a countable number of other maps, so the stack of
stable log maps has no chance of being finite type, and hence cannot
be proper. 

The solution is to identify log structures on $W$ which are universal
in a suitable sense. In the above example, all the log curves
in question arise as a cartesian diagram of log schemes:
\[
\xymatrix@C=30pt
{
(C,\M_C\oplus\NN^r)\ar[r]\ar[d]&(C,\M_C)\ar[d]\\
(W,\M_W\oplus\NN^r)\ar[r]&(W,\M_W)
}
\]
Thus all these extraneous log curves can be viewed as obtained by pull-back
from the initial choice of log curve via a logarithmic base-change.

To solve this problem, we introduce a property of stable log
maps called \emph{basic}. I do not wish to give the definition here, as
it is very involved, but the important properties of basic stable log
maps are universality and boundedness, as expressed in the following
two theorems, a summation of the main results of \cite{JAMS}:

\begin{theorem}
Given a stable log map $(C^{\dagger}/W^{\dagger},f)$, there is a basic
stable log map $(C_b^{\dagger}/W_b^{\dagger},f_b)$ fitting into a commutative
diagram
\[
\xymatrix@C=30pt
{
C^{\dagger}\ar[r]\ar[d]&C_b^{\dagger}\ar[r]\ar[d]&
X^{\dagger}\ar[d]\\
W^{\dagger}\ar[r]&W_b^{\dagger}\ar[r]&S^{\dagger}
}
\]
where the left-hand square is cartesian in the category of fine
saturated log schemes and
the maps $W\rightarrow W_b$ and $C\rightarrow C_b$ of underlying schemes
are isomorphisms. Furthermore, $(C_b^{\dagger}/W_b^{\dagger},f_b)$ and
the maps in the above diagram are determined by
$(C^{\dagger}/W^{\dagger},f)$  uniquely up to unique isomorphism. 
\end{theorem}

\begin{theorem}
Let $\cM(X^{\dagger}/S^{\dagger})$ denote the stack of basic stable log maps in
$X^{\dagger}$ over $S^{\dagger}$. Then:
\begin{enumerate}
\item $\cM(X^{\dagger}/S^{\dagger})$ is a Deligne-Mumford stack.
\item Let $\beta$ denote a choice of genus $g$, number of marked points $k$,
homology class in $H_2(X,\ZZ)$, along with a collection of tangency data
for the marked points (this notion can be made precise). Let
$\cM(X^{\dagger}/S^{\dagger},\beta)$ denote the substack of 
$\cM(X^{\dagger}/S^{\dagger})$ of basic stable log maps
of curves of genus $g$ and $k$ marked points, representing the given homology
class, and satisfying the given tangency conditions. Then modulo some
technical hypotheses on $X^{\dagger}$, $\cM(X^{\dagger}/S^{\dagger},\beta)$ is 
proper over $S$ if $X$ is proper over $S$.
\item Assuming further that $X^{\dagger}\rightarrow S^{\dagger}$ is
log smooth, $\cM(X^{\dagger}/S^{\dagger},\beta)$ 
carries a virtual fundamental class,
allowing for the definition of logarithmic Gromov-Witten invariants.
\end{enumerate}
\end{theorem}

Similar results were also obtained by Abramovich and Chen in 
\cite{AC},\cite{Chen}.

This is a promising start to the problem of understanding the $A$-model
by working entirely on the central fibre of a toric degeneration. There are,
however, still two major gaps in the theory which need to be filled. 

First,
one needs an analogue of the gluing formula. This should allow us to
break down a calculation of curves on the central fibre of a degeneration
into simpler pieces. This is expected to be quite subtle, however, and
is still work in progress. I will say a bit more shortly about what one
expects such a formula to look like.

Second, as observed earlier, the central fibre of a toric degeneration
$\shX_0^{\dagger}\rightarrow 0^{\dagger}$ is only fine saturated off of 
the set $Z$. As a result
none of the above theorems about stable log maps apply. It is quite
likely that even the definition of stable log map is not the correct
one in this case. So the theory still needs to be extended. This is also
work in progress of Michael Kasa.

Let us return to the tropicalization functor. Suppose we have a degeneration
$q:\shX\rightarrow D$, which we assume to be log smooth (say a normal crossings
degeneration), so that we don't
have to worry about the singular set $Z$ where the log structure on
$\shX$ is not fine. As usual, this gives $\shX_0^{\dagger}
\rightarrow 0^{\dagger}$. Suppose we have a basic stable log map over
a point, i.e., a diagram
\[
\xymatrix@C=30pt
{C^{\dagger}\ar[r]^f\ar[d]_\pi& \shX_0^{\dagger}\ar[d]^q\\
W^{\dagger}=(\Spec\CC, Q\oplus\CC^{\times})\ar[r]_(0.8)g&0^{\dagger}
}
\]
Here $Q$ is a monoid given by $Q=\sigma_Q\cap \ZZ^n$ for a strictly convex
rational polyhedral cone $\sigma_Q$, and the log structure on $W$ is given
by $\alpha:Q\oplus\CC^\times\rightarrow\CC$ defined by
\[
\alpha(p,s)=\begin{cases}
s & p=0\\
0 & p\not=0
\end{cases}
\]
Here, the monoid $Q$ is determined by the fact the curve is basic. We then
tropicalize this, so get a diagram
\[
\xymatrix@C=30pt
{\Trop(C^{\dagger})\ar[r]^{\Trop(f)}\ar[d]_{\Trop(\pi)}& 
\Trop(\shX_0^{\dagger})\ar[d]^{\Trop(q)}\\
\Trop(W^{\dagger})=\sigma_Q^{\vee}\ar[r]_{\Trop(g)}&
\Trop(0^{\dagger})=\RR_{\ge 0}
}
\]
The fibres of $\Trop(\pi)$ are in general one-dimensional 
graphs, while $\Trop(q)^{-1}(1)$ is the dual intersection complex $B$ of
$\shX_0^{\dagger}$. (In general, this is only a polyhedral complex and does
not carry an affine structure in codimension one, unlike the case of a toric
degeneration.) Thus $\Trop(g)^{-1}(1)\subseteq \sigma_Q^{\vee}$ can
be viewed as a space parameterizing maps from graphs (fibres of $\Trop(\pi)$)
into $B$. These will be tropical curves. In fact, where $B$ does carry
an affine structure, these curves satisfy the tropical balancing condition.

The fundamental property that the monoid $Q$ associated with the basic
log structure must satisfy is that $\Trop(g)^{-1}(1)$ must parameterize
all tropical curves in $B$ of the same ``combinatorial type''. This
makes precise the correspondence between tropical curves and log curves.

We can also describe the expected shape of a gluing formula, in keeping
with the formula developed by Brett Parker in his setting \cite{Park2}. 
One considers
tropical curves in $B$ as above. These in general move in families, but
there will be, for any given set of data $\beta$, a finite number of tropical
curves representing $\beta$
which cannot be deformed without changing the domain graph. We call
such tropical curves \emph{rigid}. The actual moduli space 
$\cM(\shX_0^{\dagger}/0^{\dagger},\beta)$ can then be viewed to have a
``decomposition into virtual irreducible components'' indexed by these rigid
curves. Furthermore, the ``virtual irreducible component'' associated to any 
rigid curve can be further related to moduli spaces of curves associated
to each vertex of the tropical curve. This should ultimately allow an
expression for the Gromov-Witten invariants of $\shX_0^{\dagger}/0^{\dagger}$,
and hence the Gromov-Witten invariants of a smoothing of $\shX_0^{\dagger}$,
in terms of much simpler invariants. This is an ongoing joint project
with Abramovich, Chen and Siebert.

\section{The $B$-model and tropical geometry}
\label{Bmodelsect}

Let us turn to the $B$-model, and understand how tropical geometry
may be visible in the variation of complex structures which is necessary for
$B$-model computations.

This problem is closely related to the reconstruction problem, as 
stated in Question \ref{reconstruct2}. If given $(B,\P,\varphi)$,
one can find an explicit description of a toric degeneration 
$\shX\rightarrow D$, with dual intersection complex $(B,\P,\varphi)$,
then one could use this explicit description
to calculate periods and extract $B$-model predictions for the mirror.

Before describing the solution to this problem, let me give a bit of history
of the reconstruction problem. The version as stated in Question
\ref{reconstruct1} was first studied by Fukaya in \cite{Fukaya}. 
There he considered
directly the question of perturbing the complex structure on $X_{\epsilon}(B_0)$
by looking at the Kodaira-Spencer equation governing deformations of
complex structure. Arguing informally in the case that $\dim B=2$, 
he suggested that the perturbations should be concentrated along trees made
of gradient flow lines, with the lines emanating initially from singular
points of $B$. This gave the first hint that a nice solution to the 
reconstruction problem might actually see something related to curves.
However, Fukaya's work contained no definite theorems, and the analysis
looked likely to be very difficult.

In 2004, Siebert and I were considering how to
solve the reconstruction problem using our program. Given $(B,\P)$, we
had shown in \cite{PartI} how to construct log schemes $X_0(B,\P,s)^{\dagger}$
along with a morphism to $0^{\dagger}$ which had all the properties
one would want for a central fibre of a toric degeneration $\shX_0^{\dagger}
\rightarrow 0^{\dagger}$. Our original hope was that a generalization of
the Bogomolov-Tian-Todorov unobstructedness theorem would allow us to
show such log schemes smoothed. 
In particular, Kawamata and Namikawa \cite{KN} had had success with this point 
of view in the normal crossings case.
While this approach works easily in dimension 2,
we couldn't make it work in higher dimension. Furthermore, this approach fails
to give the explicit description of the smoothing which would be needed
to describe the $B$-model. As a consequence, we turned towards a more
explicit approach, which involved gluing together explicit local models.

While we were working on this approach,
Kontsevich and Soibelman in \cite{KS2} got around the difficult analysis
of Fukaya's approach
by replacing complex manifolds with rigid
analytic manifolds. They were able to show that given a tropical affine
surface $B$ with $24$ singularities of focus-focus type (the simplest
type of singularity which occurs in affine surfaces, to be described
shortly)
one could construct a rigid analytic K3 surface $X^{an}(B)$. This was
done by gluing together standard pieces via automorphisms attached to
lines on $B$. These lines were given as gradient flow lines, giving
a similar, but much more precise, picture to the one given by Fukaya.

Combining our approach of gluing local models with
one of the central ideas of Kontsevich and Soibelman's work \cite{KS2}, we
were then able to complete a construction in all dimensions, giving a
satisfactory solution to the reconstruction problem within algebraic geometry.
This was carried out in \cite{Annals}.

Before surveying this approach, let me make a philosophical remark. Note
that when we were discussing the $A$-model, we observed that tropical
curves on $B$ should correspond to holomorphic curves on $X(B)$. If we want
to see these same tropical curves playing a role on the $B$-model of the
mirror, then we should think of the $B$-model not on a complex manifold
of the form $X(B)$, but rather on $\check X(B)$. This is a slightly
confusing reversal of roles. Normally counting curves is done in the symplectic
category, here expected to mean on the symplectic manifold $\check X(B)$,
while anything having to do with complex structures should be done on
$X(B)$. This reversal can be explained as follows. If we were to study
pseudo-holomorphic curves on $\check X(B)$, we would need to put an almost
complex structure on $\check X(B)$. One way to do this is to choose a metric
on $B$; this induces an almost complex structure on $\check X(B)$ constant
on fibres of the torus fibration, generalizing the construction of
a complex structure from a Hessian metric described in \S\ref{semiflatsection}.
Then in a suitable adiabatic limit where the almost complex structure is 
rescaled,
pseudo-holomorphic curves are expected to tend towards trees of gradient
flow lines. If the chosen metric was in fact Hessian, these gradient
flow lines would in fact be straight lines with respect to the Legendre
dual affine structure, so these trees of gradient flow lines can be 
viewed as a generalization of tropical curves. However, tropical geometry
is linear and much easier to control. We take the attitude that we should
work on the side in which tropical geometry appears. Indeed, this turns
out to be very helpful.

Given this, we can then present a somewhat revised version of the mirror
symmetry program:
\begin{enumerate}
\item We begin with a toric degeneration of Calabi-Yau manifolds
$\shX\rightarrow D$ with an ample polarization.
\item Construct the dual intersection complex 
$(B,\P,\varphi)$ from this data.
\item Construct a new toric degeneration $\check\shX\rightarrow D$
whose \emph{intersection complex} is $(B,\P,\varphi)$. This degeneration
should be controlled by tropical data.
\item Understand genus $0$ holomorphic curves (or whatever other aspect of
the $A$-model one is interested in) on the general fibre of $\shX\rightarrow
D$ in terms of tropical geometry of $B$.
\item Understand the variation of Hodge structures for $\check\shX\rightarrow
D$ in terms of tropical geometry of $B$.
\item Use the fact that the $A$- and $B$-models of $\shX$ and $\check\shX$
respectively are controlled by the same tropical geometry on $B$ to
prove mirror symmetry. 
\end{enumerate}

Here we outline the completion of step (3) as carried out in 
\cite{Annals}.

The first step is as follows. Given $(B,\P,\varphi)$, we wish to construct
the central fibre $\shX_0^{\dagger}$ of the degeneration. This in fact
was carried out in \S 5 of \cite{PartI}, assuming certain genericity
assumptions on the singular locus of $B$. As a scheme, it is fairly
obvious what $\shX_0$ should be. For each maximal cell $\sigma$,
one has an associated projective toric variety $\PP_{\sigma}$ with
Newton polytope $\sigma$. Any face $\tau\subseteq\sigma$ specifies
a toric strata $\PP_{\tau}\subseteq\PP_{\sigma}$, and given
$\tau=\sigma_1\cap\sigma_2$ for $\sigma_1,\sigma_2$ maximal, we can glue
together the toric strata $\PP_{\tau}\subseteq\PP_{\sigma_1},\PP_{\sigma_2}$
in a torus equivariant manner. There is of course a whole family of
possible gluings, parameterized by what we call \emph{closed gluing data}
in \cite{PartI}. Given closed gluing data $s$, we obtain a scheme
$\check X_0(B,\P,s)$.

Now $\check X_0(B,\P,s)$ cannot be a central fibre of a 
toric degeneration unless
it carries a log structure of the correct sort. There are many reasons
this may not happen. If $s$ is poorly chosen, there may be
zero-dimensional strata of $\check X_0(B,\P,s)$ which do not have neighbourhoods
locally \'etale isomorphic to the toric boundary of an affine toric variety;
this is a minimal prerequisite. As a result, we have to
restrict attention to closed
gluing data induced by what we call \emph{open gluing data}. 
Explicitly, each vertex $v$ of $\P$ defines local models $V(v)\subseteq U(v)$
as follows. The piecewise linear function
$\varphi$ is defined locally up to affine linear functions. Choose a
representative $\varphi_v$ for $\varphi$ in a neighbourhood of $v$ which
takes the value $0$ at $v$. By extending the function linearly on each
cell, we can view this as a piecewise linear function on the fan
$\Sigma_v$, viewed as a fan in some $\RR^n$. We can then set
\[
P_v:=\{(m,r)\in \ZZ^n\times\ZZ\,|\,r\ge \varphi_v(m)\}.
\]
Noting that $(0,1)\in P_v$,  we set 
\begin{align*}
U(v):={} & \Spec \CC[P_v],\\
V(v):={} & \Spec \CC[P_v]/(z^{(0,1)}).
\end{align*}
Note that $z^{(0,1)}$ vanishes to order one on every toric divisor
of $U(v)$, so in fact $V(v)$ is the toric boundary of $U(v)$.
It turns out,
as we show in \cite{PartI}, that a necessary condition for $\check X_0(B,\P,s)$
to be the central fibre of a toric degeneration is that it is obtained
by dividing out $\coprod_{v\in\P} V(v)$ by an equivalence relation. 
In other words, we are gluing together the $V(v)$'s along Zariski
open subsets to obtain a scheme.\footnote{\cite{PartI} allowed the case that
the cells of $\P$ self-intersect. As a consequence, the equivalence relation
is merely \'etale and one obtains an algebraic space.}
Again, there is some choice of gluing, but now the gluing data are given
by equivariant identifications of open subsets of the various $V(v)$'s. 
We call this \emph{open gluing data}.

The advantage of using open gluing data is that each $V(v)$
carries a log structure induced by the
divisorial log structure $V(v)\subseteq U(v)$. 
These log structures are not identified under the open gluing maps, but
the ghost sheaves of the log structures are isomorphic. 
So the ghost sheaves $\overline{\M}_{V(v)}$
glue to give a ghost sheaf of monoids $\overline{\M}_{\check X_0(B,\P,s)}$.
Thus we see how $\varphi$ influences the log structure.

One then tries to construct a log structure with this ghost sheaf. This is
done in \cite{PartI} by building suitable extensions of the ghost sheaf with
$\shO_{\check X_0(B,\P,s)}^{\times}$, 
and this extension depends on some moduli (which may in general
be empty). The good situation is that one can find a closed
subset $Z\subseteq\check X(B,\P,s)$ of complex codimension at least two
and a log structure on $\check X_0(B,\P,s)$ along
with a morphism $\check X_0(B,\P,s)^{\dagger}\rightarrow 0^{\dagger}$ which
is log smooth away from $Z$. Furthermore, the ghost sheaf on
$\check X_0(B,\P,s)\setminus Z$ should be the given ghost sheaf of
monoids $\overline{\M}_{\check X_0(B,\P,s)}$ restricted to
$\check X_0(B,\P,s)\setminus Z$. We call such a log scheme with morphism
to $0^{\dagger}$ a \emph{log Calabi-Yau space}.

The technical heart of \cite{PartI} is an explicit classification of
log Calabi-Yau spaces with given intersection complex $(B,\P,\varphi)$,
modulo some assumptions on the singularities of $B$ called \emph{simplicity}.
The definition of simplicity is rather involved, so we will not give it
here, but it essentially says that not too much topology of $\check X(B)$
(or $X(B)$) can be hiding over the singular locus of $B$.

A main result of \cite{PartI}, (Theorem 5.4) is then

\begin{theorem}
Given $(B,\P,\varphi)$ simple, the set of
log Calabi-Yau spaces with intersection complex $(B,\P,\varphi)$ 
modulo isomorphism preserving $B$ (i.e., does not interchange irreducible 
components) 
is $H^1(B,i_*\check \Lambda\otimes \CC^{\times})$.
An isomorphism is said to preserve $B$ if it induces the identity on
the intersection complex.
\end{theorem}

So the moduli space is an algebraic torus (or a disjoint union of algebraic
tori) of dimension equal to $\dim_\CC H^1(B,i_*\check\Lambda\otimes \CC)$.
In \cite{PartII}, we in fact show the dimension of this torus is the
dimension of $H^1(\shX_t,\shT_{\shX_t})\cong H^{n-1,1}(\shX_t)$ for
a smooth fibre $\shX_t$ of a smoothing $\shX\rightarrow D$ of 
$X_0(B,\P,s)^{\dagger}$. This is the expected dimension, as this
latter vector space is the tangent space to the moduli space of
$\shX_t$.

Now assume given a log Calabi-Yau space $X_0:=X_0(B,\P,s)^{\dagger}
\rightarrow 0^{\dagger}$. Our goal is to use the log structure to
provide ``initial conditions'' to produce $k$-th order deformations $X_k
\rightarrow\Spec\CC[t]/(t^{k+1})$, order by order. To do so, we will glue
together standard thickenings of ``pieces'' of $X_0$, modifying standard
gluings by a complicated system of data we call a \emph{structure}. 

First, the ``pieces'' of $X_0$ we consider are toric open affine subsets
of strata of $X_0$. Recall that strata of $X_0$ are indexed by cells
$\tau\in\P$, corresponding to a projective toric variety $\PP_{\tau}$. 
Recall also that if $\omega\subseteq\tau$, the normal cone to $\tau$
along $\omega$ is a cone in the fan defining $\PP_{\tau}$ and hence
defines an open affine subset of $\PP_{\tau}$. We call this open
affine subset $V_{\omega,\tau}\subseteq\PP_{\tau}$; note
\[
V_{\omega,\tau}=\PP_{\tau}\setminus \bigcup_{\rho\subseteq\tau\atop
\omega\not\subseteq\rho} \PP_{\rho}.
\]
For example, if
$\omega$ is a vertex of $\tau$, then $V_{\omega,\tau}$ is the standard
toric open affine subset of $\PP_{\tau}$ containing the zero-dimensional
stratum of $\PP_{\tau}$ corresponding to $\omega$. 

Second, what are the thickenings of the sets $V_{\omega,\tau}$? These
can be described explicitly as follows. Choose a point $x$ in the interior
of $\omega$ not contained in the singular locus $\Gamma$ of $B$. We obtain
a fan $\Sigma_x$ in the tangent space $\Lambda_x\otimes_{\ZZ}\RR$
of not necessarily strictly convex cones consisting of
the tangent cones at $x$ of each cell $\sigma$ containing $\omega$. 
We can choose a representative $\varphi_x$ for $\varphi$ in a small
neighbourhood of $x$ which is zero along $\omega$, and this can then be
extended linearly on each cone of $\Sigma_x$ to view $\varphi_x$ as
a piecewise linear function $\varphi_x:\Lambda_x\otimes\RR\rightarrow\RR$.
This in turn defines a monoid
\[
P_x:=\{(m,r)\in \Lambda_x\times \ZZ\,|\, r\ge \varphi_x(m)\},
\]
completely analogous to the definition of $P_v$. 

For each maximal cell $\sigma$ containing $\tau$, let $n_{\sigma}
\in\check\Lambda_x$
denote the slope of $\varphi_x$ restricted to the tangent cone of
$\sigma$. We then define a monomial ideal in the ring $\CC[P_x]$
given by
\[
I_{\omega,\tau}^{>k}=\langle z^{(m,r)}\,|\, \hbox{
$(m,r)\in P_x$, $r-\langle n_{\sigma},m\rangle > k$ for some
$\sigma\in\P_{\max}$ with $\sigma\supseteq\tau$}\rangle.
\]
Then the desired standard thickening of $V_{\omega,\tau}$ is
\[
V^k_{\omega,\tau}:=\Spec\CC[P_x]/I_{\omega,\tau}^{>k}.
\]
One checks easily that if $k=0$, this recovers $V_{\omega,\tau}$, and if
$k>0$, then the reduced space of $V^k_{\omega,\tau}$ is $V_{\omega,\tau}$.
Thus this is indeed a thickening of $V_{\omega,\tau}$.

There is one point we have to be quite careful about. This definition 
would appear to depend on the point $x$, and identifications of
different tangent spaces $\Lambda_x$, $\Lambda_{x'}$ via parallel
transport depend on the path because of the presence of the singular locus.
We deal with this issue not by choosing a specific point $x$, but choosing
a specific maximal reference cell $\sigma$ containing $\tau$. We then
can identify any $\Lambda_x$ with $\Lambda_{\sigma}$, the well-defined
tangent space to $\sigma$, via parallel transport from $x$ directly into
$\sigma$. We will notate this additional choice of reference cell by 
writing $V^k_{\omega,\tau,\sigma}$. A different choice of reference
cell $\sigma'$ gives a space
$V^k_{\omega,\tau,\sigma'}$ abstractly, but not canonically, isomorphic to
$V^k_{\omega,\tau,\sigma}$. This will prove important
below. We also use the notation for the coordinate rings
\[
R^k_{\omega,\tau,\sigma}:=\CC[P_x]/I_{\omega,\tau}^{>k},
\]
again keeping in mind this choice of reference cell.

There are also natural gluings between these various thickened schemes.
One notes that given $\tau_1\subseteq\tau_2\subseteq\tau_3$ there are
natural surjections
\[
R^k_{\tau_1,\tau_3,\sigma}\rightarrow R^k_{\tau_1,\tau_2,\sigma}
\]
giving a closed embedding $V^k_{\tau_1,\tau_2,\sigma}\rightarrow
V^k_{\tau_1,\tau_3,\sigma}$, and natural inclusions
\[
R^k_{\tau_1,\tau_3,\sigma}\rightarrow R^k_{\tau_2,\tau_3,\sigma},
\]
giving open embeddings $V^k_{\tau_2,\tau_3,\sigma}
\rightarrow V^k_{\tau_1,\tau_3,\sigma}$.

If $B$ has no singularities, then the reference cell $\sigma$ is
not important, and we drop this from the notation in this case.
In particular, it is easy to check that if we take, say, $\tau_1$ to be
a fixed vertex $v$, and we take the limit of the directed system
$\{V^k_{v,\tau}\,|\,v\in\tau\}$ of schemes, we obtain a $k$-th order
thickening $V^k(v)$ of $V(v)$ given by $V^k(v)=U(v)\times_{\AA^1} 
\Spec\CC[t]/(t^{k+1})$, with $U(v)\rightarrow\AA^1$ the morphism given
by $z^{(0,1)}$. This is precisely the kind of vanilla smoothing the
log structure leads us to expect. Note we can write this direct limit
of schemes as
\[
\Spec \lim_{\longleftarrow\atop\tau} R^k_{v,\tau}.
\]

The basic idea then will be to modify the various maps above by some
additional data.

To understand why we need these modifications, let us consider the single 
most important example, that of an isolated singularity of focus-focus
type in a two-dimensional $B$.

We suppose $\P$ contains two maximal cells $\sigma_1,\sigma_2$, with
$\sigma_1\cap\sigma_2=\tau$, as depicted in Figure \ref{TwoTriangles}.
Note that the intersection of the two coordinate charts is
$(\sigma_1\cup\sigma_2)\setminus\tau$, and the transition map is then
the identity on $\sigma_1\setminus\tau$ and is given by the linear 
transformation $\begin{pmatrix} 1&0\\ 1&1\end{pmatrix}$ on $\sigma_2\setminus
\tau$. Together,
these two charts define an integral affine structure on 
$(\sigma_1\cup\sigma_2)\setminus\Gamma$, where $\Gamma=\{p\}$ 
is the common point of the two cuts.

\begin{figure}
\input{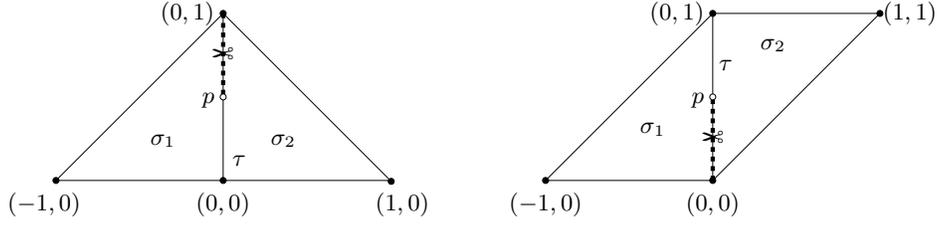}
\caption{The fundamental example. The diagram shows the affine embeddings
of two charts, obtained by cutting the union of two triangles as indicated
in two different ways. Each triangle is a standard simplex.}
\label{TwoTriangles}
\end{figure}

We then take $\varphi$ to be single-valued, identically $0$ on $\sigma_1$
and taking the value $1$ at the right-hand vertex. 

One now finds
\begin{align*}
R^k_{\tau,\sigma_1,\sigma_1}= {} & \CC[x,y,w^{\pm 1}]/(y^{k+1})\\
R^k_{\tau,\sigma_2,\sigma_2}= {} & \CC[x,y,w^{\pm 1}]/(x^{k+1})\\
R^k_{\tau,\tau,\sigma_i}= {} & \CC[x,y,w^{\pm 1}]/(x^{k+1},y^{k+1}).
\end{align*}
Here, if we use the chart on the left, i.e., choose a point $s$
below $p$ and work in $P_s\subseteq\Lambda_s\oplus\ZZ$, 
the variables $x,y$ and $w$
are identified with elements of $\CC[P_s]$ as
\[
x=z^{(-1,0,0)},\quad y=z^{(1,0,1)},\quad w=z^{(0,1,0)}.
\]
We have the natural surjections $R^k_{\tau,\sigma_i,\sigma_i}
\rightarrow R^k_{\tau,\tau,\sigma_i}$, and we identify 
$R^k_{\tau,\tau,\sigma_1}$ with $R^k_{\tau,\tau,\sigma_2}$ by identifying
$\Lambda_{\sigma_1}$ and $\Lambda_{\sigma_2}$ by parallel
transport through $s$. Since we have
written everything in the left-hand chart, where $\Lambda_{\sigma_1}$
and $\Lambda_{\sigma_2}$ are identified via parallel transport through
$s$, this identification  is the trivial one. We can thus glue together
the coordinate rings of the thickenings as
\[
R^k_{\tau,\sigma_1,\sigma_1}\times_{R^k_{\tau,\tau,\sigma_i}}
R^k_{\tau,\sigma_2,\sigma_2}.
\]
This fibred product of rings is easily seen to be isomorphic to the ring
\[
\CC[X,Y,W^{\pm 1},t]/(t-XY, t^{k+1}),
\]
where $X=(x,x)$, $Y=(y,y)$,
$W=(w,w)$, and $t=(xy,xy)$ as elements of the Cartesian product of rings.

On the other hand, suppose we instead identified $R^k_{\tau,\tau,\sigma_1}$
and $R^k_{\tau,\tau,\sigma_2}$ by parallel transport through a point
$s'$ lying above $p$. To do this, we can work in the right-hand
chart. Again, $x,y$ and $w$ are defined using the tangent vectors
$(-1,0), (1,0)$ and $(0,1)$ in $\sigma_1$, and these
are transported to the same tangent vectors in $\sigma_2$ in the second
chart. However, to compare this with our original description of $R^k_{\tau,
\tau,\sigma_2}$, we need to think of these as tangent vectors in $\sigma_2$
in the original chart, i.e., the left-hand chart. There, these tangent
vectors are $(-1,1)$, $(1,-1)$ and $(0,1)$ respectively. Thus we obtain
an isomorphism $R^k_{\tau,\tau,\sigma_1}\rightarrow R^k_{\tau,\tau,\sigma_2}$
given by
\begin{equation}
\label{auto1}
x\mapsto xw,\quad y\mapsto yw^{-1}, \quad w\mapsto w.
\end{equation}
Using this identification, we obtain a composed map $R^k_{\tau,\sigma_1,
\sigma_1}\rightarrow R^k_{\tau,\tau,\sigma_1}\rightarrow R^k_{\tau,\tau,
\sigma_2}$, leading to a fibred product
\[
R^k_{\tau,\sigma_1,\sigma_1}\times_{R^k_{\tau,\tau,\sigma_2}}
R^k_{\tau,\sigma_2,\sigma_2}
\cong \CC[X,Y,W^{\pm 1},t]/(XY-tW, t^{k+1}),
\]
where now
\[
X=(x,xw),\quad Y=(yw,y),\quad W=(w,w), \quad t=(xy,xy).
\]
Note that while this new ring is abstractly isomorphic to the previous
ring, there is no isomorphism as $\CC[t]/(t^{k+1})$-algebras.

So the gluing is not well-defined, and this is caused by the
singularities of $B$. The correct smoothing in this case will depend
on the choice of log structure, but in any event we expect
it should be
a family of the form $\Spec \CC[X,Y,W^{\pm 1},t]/(XY-f(W)t)$
for some function $f(W)$ which vanishes along the $W$-axis precisely
at the points where the given log structure on $X_0(B,\P,s)$ is not fine.
Clearly $f$ is then determined by the log structure up to invertible
functions. Let us take
for the sake of this example the function $f(W)=1+W$, noting that
$f(W)=1+W^{-1}$ would do just as well. We can now modify the gluings using
Figure \ref{TwoTriangles2}.

\begin{figure}
\input{TwoTriangles2.pstex_t}
\caption{}
\label{TwoTriangles2}
\end{figure}

In this figure, we have drawn two rays contained in $\tau$ emanating
from the singular point, and labelled these two arrows with the functions
$1+w^{-1}$ and $1+w$ respectively. These rays tell us
that if we try to identify $R^k_{\tau,\tau,\sigma_1}$ with $R^k_{\tau,\tau,
\sigma_2}$ using parallel transport between the two maximal cells, we need
to modify the identification via an automorphism given by the crossing
of one of these rays. Here, we will get different automorphisms depending
on whether we cross above or below the singularity $p$. If we cross below,
the ray tells us to use an automorphism of $R^k_{\tau,\tau,\sigma_1}$ given by
\begin{equation}
\label{auto2}
x\mapsto x(1+w),\quad y\mapsto y(1+w)^{-1},\quad w\mapsto w,
\end{equation}
while if we cross above the singularity, we use the automorphism
\begin{equation}
\label{auto3}
x\mapsto x(1+w^{-1}),\quad y\mapsto y(1+w^{-1})^{-1},\quad w\mapsto w.
\end{equation}
Actually, note that $1+w$ or $1+w^{-1}$ is not invertible in
$R^k_{\tau,\tau,\sigma_i}$, so we need to modify this ring by localizing it at
$1+w$ (or equivalently $1+w^{-1}$).
Let's see how this affects the fibred products 
$R^k_{\tau,\sigma_1,\sigma_1}
\times_{R^k_{\tau,\tau,\sigma_2}} R^k_{\tau,\sigma_2,\sigma_2}$.

If we use parallel transport below the singular point, then the map
$R^k_{\tau,\sigma_1,\sigma_1}\rightarrow R^k_{\tau,\tau,\sigma_2}$ is
just given by \eqref{auto2}, while $R^k_{\tau,\sigma_2,\sigma_2}
\rightarrow R^k_{\tau,\tau,\sigma_2}$ remains the canonical one. One then
finds
\[
R^k_{\tau,\sigma_1,\sigma_1}
\times_{R^k_{\tau,\tau,\sigma_2}} R^k_{\tau,\sigma_2,\sigma_2}
\cong \CC[X,Y,W^{\pm},t]/(XY-(1+W)t, t^{k+1}),
\]
with
\[
X=(x,x(1+w)),\quad Y=(y(1+w),y),\quad W=(w,w),\quad t=(xy,xy).
\]
On the other hand, if we use parallel transport above the singular point,
we need to compose the automorphism \eqref{auto3} with the isomorphism
\eqref{auto1}, giving a map $R^k_{\tau,\sigma_1,\sigma_1}
\rightarrow R^k_{\tau,\tau,\sigma_2}$ given by
\[
x\mapsto xw(1+w^{-1})=x(1+w),\quad y\mapsto yw^{-1}(1+w^{-1})^{-1}
=y(1+w)^{-1},\quad w\mapsto w.
\]
Thus this map is exactly the same as \eqref{auto2},
and hence we get the same fibred product. The
glued thickenings are independent of choices. 
The introduction of the extra automorphisms removes the problems caused
by monodromy.

This is a very local situation. The next problem which arises is that
more globally, we need to propagate the automorphisms attached to the
rays. Indeed, imagine now that the picture we are looking at is contained
in a more complex situation, as on the left-hand side of Figure
\ref{proprays}. Here we have two singularities, and rays emanate in each
direction from the singularity. Let us follow the rule that any identification
of rings which involves parallel transport through a ray must be modified
by the appropriate automorphism as described above.
Then looking at the vertex $v_1$, say, we need to glue together five
irreducible components, but only one of these gluings is modified.
These gluings would not be compatible. To correct for this, one can
extend the ray indefinitely, and ``parallel transport'' the automorphism
along the ray. There is a precise sense in which this can be done. This is shown
on the right-hand picture in Figure \ref{proprays}, with the dotted lines
showing the extension of the rays. Now if crossing a ray in one direction
produces
the inverse of the automorphism given by crossing the ray the other direction,
one finds that gluing at the vertices $v_1$ and $v_2$ have now become
compatible.

A new problem arises, however, at the intersection point of the two rays.
Again, when we try to identify various rings using parallel transport
and automorphisms induced by crossing rays, we don't want the choice
to depend on the particular path we take. Because in general the two
automorphisms attached to the rays don't commute, we again have trouble
at the point of intersection. 

This is in fact where our thinking stood in early 2004, shortly before
the release of Kontsevich and Soibelman's paper \cite{KS2}. 
The solution to this problem, really the key part of Kontsevich and 
Soibelman's argument,
is to add new rays emanating from the point of intersection
of the old rays, as depicted in Figure \ref{ksrays}. These rays
are added in such a way as to guarantee that the composition of automorphisms
given by a loop around the intersection point is in fact the identity,
and thus the identifications will be independent of the choice of path.

\begin{figure}
\input{proprays.pstex_t}
\caption{}
\label{proprays}
\end{figure}

\begin{figure}
\input{ksrays.pstex_t}
\caption{}
\label{ksrays}
\end{figure}

The description here is somewhat vague, but demonstrates the basic idea.
We've seen how we obtain our degeneration by
gluing together basic pieces. 
Other than these different basic pieces,
in two dimensions, the main distinction between our approach
and the one taken by Kontsevich and Soibelman in \cite{KS2} is that
we work in the affine structure dual to the one \cite{KS2} works with.
They propogate automorphisms along gradient flow lines, but
we are able to propogate automorphisms along straight lines with respect
to the affine structure. This saves a great deal of trouble in higher
dimensions, where gradient flow lines will be much more difficult
to control. That makes it possible for us to obtain results in all
dimensions.

We of course have not made it particularly clear how we really encode
automorphisms and how they propagate, but we will make at least the first
point clearer in the next section. For the second point, the main thing
is that they propagate along straight lines; this in fact is crucial
for guaranteeing that the automorphisms don't start to involve monomials
with poles on irreducible components of $X_0$. 
So here we see something which looks tropical
already, with the union of rays looking like a tropical tree. Again, we
will make this more precise in the next section.

In higher dimensions, the argument becomes much more
subtle. Instead of rays carrying automorphisms, codimension one wall 
carry automorphisms, and one needs to be very careful about how these walls
propagate. Furthermore, there are great technical difficulties concerning
convergence of the algorithm near the discriminant locus. This was handled
in \cite{KS2} in two dimensions via an argument showing new rays added
can be guaranteed to avoid a neighbourhood of each singularity, but this
is done by choosing the metric carefully. In higher dimensions, this is
not true, and instead we used algebraic methods to prove convergence.
All these difficulties were overcome in \cite{Annals}.

In \cite{Gbook} I wrote down a complete
version of the proof in two dimensions; this has the advantage of avoiding 
most of the really technical issues. Hopefully, \cite{Gbook} provides a
gentler entry point into the ideas outlined here than the main paper
\cite{Annals}.

We now turn to a more precise description of the automorphisms involved,
and give evidence that the description of the explicit deformations
(which we view as $B$-model information) really encodes $A$-model
information on the mirror.

\section{The tropical vertex}

To simplify the discussion, we will work in this section only with the
simplest rings which occur in the previous section, of the form
$R^k_{\sigma,\sigma,\sigma}$ where $\sigma$ is a maximal (two-dimensional)
cell. This ring is isomorphic to $\CC[x^{\pm 1},y^{\pm 1},t]/(t^{k+1})$.
Let us work formally instead, setting 
\[
R=\CC[x^{\pm 1},y^{\pm 1}]\lfor t\rfor.
\]
This is the ring of formal power series in $t$ with coefficients
Laurent polynomials in $x$ and $y$.
Let $f\in R$ be of the form
\[
f=1+tx^ay^b\cdot g(x^ay^b,t),\quad g(z,t)\in \CC[z]\lfor t\rfor.
\]
Then this defines an automorphism $\theta_{(a,b),f}$ of $R$ as a 
$\CC\lfor t\rfor$-algebra given by
\[
\theta_{(a,b),f}(x)= x \cdot f^{b},\quad \theta_{(a,b),f}(y)=y\cdot f^{-a}.
\]
Note that $\theta_{(a,b),f}^{-1}=\theta_{(a,b),f^{-1}}$. These automorphisms
have the further property that they preserve the holomorphic symplectic
form ${dx\over x}\wedge {dy\over y}$.

We define the \emph{tropical vertex group} $\VV$ to be the completion
with respect to the maximal ideal $(t)\subseteq\CC\lfor t\rfor$
of the  subgroup of 
$\CC\lfor t\rfor$-algebra automorphisms of $R$ generated by all
such automorphisms. Note that infinite products are defined in
$\VV$ only if only finitely many factors are non-trivial modulo $t^k$ for every
$k>0$.
This is a slight modification of a group originally introduced by 
Kontsevich and Soibelman in \cite{KS2}. 

We now describe a local version of the rays described in the previous section.
For convenience, set $M=\ZZ^2$, $M_{\RR}=M\otimes_{\ZZ}\RR$, 
and identify $\CC[x^{\pm 1},y^{\pm 1}]$
with $\CC[M]$. 

\begin{definition} A \emph{ray} or \emph{line} in $M_{\RR}$ is a pair
$(\fod,f_{\fod})$ for some $\fod=\RR_{\le 0}m$ if $\fod$ is a ray and
$\fod=\RR m$ if $\fod$ is a line, where $m\in M\setminus\{0\}$.
Furthermore, 
\[
f_{\fod}=1+tz^m\cdot g(z^m,t)\in R,\quad g(z,t)\in \CC[z]\lfor t\rfor,
\]

A \emph{scattering diagram} $\foD$ is a collection
of rays and lines $\{(\fod,f_{\fod})\}$ with the property that for any
$k>0$, $f_{\fod}\equiv 1 \mod t^k$ for all but a finite number of
elements of $\foD$.
\end{definition}

Given a scattering diagram $\foD$, let 
\[
\Supp\foD=\bigcup_{(\fod,f_{\fod})\in\foD} \fod.
\]
If we are given a path $\gamma:[0,1]\rightarrow M_{\RR}\setminus\{0\}$
with $\gamma(0),\gamma(1)\not\in\Supp(\foD)$ and $\gamma$ being transversal
to each ray it crosses, then we can define the \emph{path-ordered product}
$\theta_{\gamma,\foD}\in\VV$ which is a composition of automorphisms
associated to each ray that $\gamma$ crosses. We define the path-ordered
product for each power $k>0$, and take the limit. For any given $k>0$, 
we can find numbers
\[
0<t_1\le t_2\le \cdots\le t_s<1
\]
and elements $\fod_i\in\foD$ with $f_{\fod_i}\not\equiv 1\mod t^k$
such that $\gamma(t_i)\in\fod_i$, $\fod_i\not=\fod_j$ if $t_i=t_j$, and
$s$ taken as large as possible. For each $i$ define $\theta_{\fod_i}$
to be the automorphism defined as follows.
Let $n\in N=\Hom(M,\ZZ)$ be
the unique primitive element which vanishes on $\fod_i$ and is negative
on $\gamma'(t_i)$. Then define $\theta_{\fod_i}$ to be the automorphism
\[
\theta_{\fod_i}(z^m)=z^mf_{\fod_i}^{\langle n,m\rangle}.
\]
Note this is of the form $\theta_{(a,b),f_{\fod}}$ for suitable choice of
$(a,b)$. We then define
\[
\theta^k_{\gamma,\foD}=\theta_{\fod_s}\circ\cdots\circ\theta_{\fod_1}.
\]
Note that
if $\gamma$ crosses two rays at the same time, the order doesn't matter
as one checks easily that two automorphisms commute if they are associated
with the same underlying $\fod\subseteq M_{\RR}$. Finally, we define
\[
\theta_{\gamma,\foD}=\lim_{k\rightarrow\infty} \theta^k_{\gamma,\foD}.
\]

We can then express the essential lemma of \cite{KS2} in this context:

\begin{proposition}
\label{KSLemma}
Let $\foD$ be a scattering diagram. Then there is a scattering diagram
$\Scatter(\foD)$ such that $\Scatter(\foD)\setminus\foD$ consists
just of rays and $\theta_{\gamma,\Scatter(\foD)}$ is the identity for
any loop $\gamma$ around the origin.
\end{proposition}

The proof is very simple and algorithmic; I give a quick outline. One
constructs a sequence of scattering diagrams $\foD_1=\foD,\foD_2,\ldots$
with the property that $\theta_{\gamma,\foD_k}\equiv\id \mod t^{k}$. 
This is clearly true for $\foD_1$, so we proceed inductively, assuming
we have constructed $\foD_k$. Then one shows (by looking at the Lie algebra
of $\VV$) that
\begin{align*}
\theta_{\gamma,\foD_k}(x)= {} & x\sum_{i=1}^n b_ic_it^kx^{a_i}y^{b_i}\\
\theta_{\gamma,\foD_k}(y)= {} & -y\sum_{i=1}^n a_ic_it^kx^{a_i}y^{b_i}
\end{align*}
for integers $a_i,b_i$ (with $a_i,b_i$ not both zero) and $c_i\in\CC$.
Then one obtains $\foD_{k+1}$ by adding rays 
\[
(\RR_{\le 0}(a_i,b_i), 1\pm c_it^kx^{a_i}y^{b_i}),\quad 1\le i\le n
\]
with the sign chosen so that when $\gamma$ crosses this ray, it produces
the automorphism
\[
x\mapsto x(1-b_ic_it^kx^{a_i}y^{b_i})\mod t^{k+1},\quad
y\mapsto y(1+a_ic_it^kx^{a_i}y^{b_i})\mod t^{k+1}.
\]
Since this automorphism will commute with all other automorphisms in
$\foD_k$ modulo $t^{k+1}$, inserting these rays will precisely
cancel out the contributions to $\theta_{\gamma,\foD_k}$ to order $k$,
and thus $\theta_{\gamma,\foD_{k+1}}\equiv \id\mod t^{k+1}$.

It is very easy to program this algorithm and explore these scattering
diagrams. They appear to have a very rich and fascinating structure. The
following simple examples show their complexity.

\begin{example}
\label{scatdiagexample}
Consider the case that
\[
\foD=\{(\RR (1,0), (1+tx^{-1})^\ell), (\RR (0,1), (1+ty^{-1})^\ell)\}
\]
for $\ell$ some positive integer. For $\ell=1$, it is easy to check that
\[
\Scatter(\foD)\setminus\foD=\{(\RR_{\ge 0}(1,1),1+t^2x^{-1}y^{-1})\}.
\]
Figure \ref{oneone} shows explicitly what the automorphisms are as 
one traverses the depicted loop; the reader can easily check that the 
composition of the five automorphisms is the identity.

\begin{figure}
\input{oneone.pstex_t}
\caption{$\Scatter(\foD)$ for $\ell=1$. Here the automorphisms
are given
explicitly, and the identity $\theta_{\gamma,\Scatter(\foD)}$ is just the composition of the
given automorphisms.}
\label{oneone}
\end{figure}

If $\ell=2$, then one finds
\begin{eqnarray*}
\Scatter(\foD)
\setminus\foD=&&\{(\RR (n+1,n),(1+t^{2n+1}x^{-(n+1)}y^{-n})^2)| 
n\in\ZZ, n\ge 1\}\\
&\cup&
\{(\RR (n,n+1),(1+t^{2n+1}x^{-n}y^{-(n+1)})^2)| n\in\ZZ, n\ge 1\}\\
&\cup&\{(\RR (1,1), (1-t^2x^{-1}y^{-1})^{-4})\}.
\end{eqnarray*}
This was first found experimentally by myself and Siebert via a computer
program, and the first verification of this was given in \cite{GMN}.
It also follows immediately from the results of \cite{GPS} which will be
explained in what follows.

If $\ell=3$, the situation becomes even more complicated. First, as
noticed by Kontsevich, $\Scatter(\foD)$ has a certain periodicity. Namely,
\[
(\RR_{\ge 0} (m_1,m_2),f(x^{-m_1}y^{-m_2}))
\in\Scatter(\foD)
\]
if and only if
\[
(\RR_{\ge 0} (3m_1-m_2,m_1),f(x^{-(3m_1-m_2)}y^{-m_1}))
\in\Scatter(\foD),
\]
provided that $m_1,m_2$ and $3m_1-m_2$ are all positive.
In addition, there are rays with support
$\RR_{\ge 0} (3,1)$ and $\RR_{\ge 0} (1,3)$, hence by the
periodicity, there are also rays with support
\[
\RR_{\ge 0}(8,3),\ \RR_{\ge 0}(21,8),\ \ldots \ \ \
\text{and}\  \ \  \RR_{\ge 0}(3,8),\ \RR_{\ge 0}(8,21),\ \ldots 
\]
which converge to
the rays of slope $(3\pm\sqrt{5})/2$, corresponding to the two distinct
eigenspaces of the linear transformation
$\begin{pmatrix}3&-1\\1&0\end{pmatrix}$. Each of these rays
is of the form
\[
\big(\RR_{\ge 0}(m_1,m_2), (1+t^{m_1+m_2}x^{-m_1}y^{-m_2})^3\big).
\]
These are the only rays appearing outside of the cone generated by
the rays of slope $(3\pm\sqrt{5})/2$. 
On the other hand, inside this cone, every rational slope occurs, and
the attached functions are very complicated. For example, the function
attached to the line of slope 1 is
\[
\left(\sum_{n=0}^{\infty} {1\over 3n+1}\begin{pmatrix}4n\\ n\end{pmatrix}
t^{2n}x^{-n}y^{-n}\right)^9.
\]
Again, Siebert and 
I found this form via computer experiment, but it was verified
by Reineke in \cite{Rei}. Recently, Kontsevich has shown the functions
attached to all these rays are algebraic. For example, if $g$ denotes 
the $9$-th root of the above function, it satisfies the equation
\[
t^2x^{-1}y^{-1}g^4-g+1=0.
\]

This series of examples also makes contact with a number of other interesting
objects. On the one hand, Reineke in \cite{Rei} gave an interpretation
of the attached functions in terms of Euler characteristics of moduli
spaces of representaions of the Kronecker $\ell$-quiver, the quiver with
two vertices and $\ell$ arrows between them. On the other hand, these
diagrams are also closely related to the cluster algebras defined
by these quivers. This connection will be studied in more detail in
work with Hacking, Keel, and Kontsevich \cite{GHKK}.
\end{example}

We will now explain the enumerative interpretation for the functions
which arise in $\Scatter(\foD)$. To motivate this, let us return to
the tropical interpretation of \S\ref{tropgeomsect}. Begin, say,
with a tropical manifold $B$ which corresponds to a K3 surface, as depicted
in Figure \ref{tropicaldisk}, along with what we will call a \emph{tropical
disk}. This is almost a tropical curve, but it just ends at the point
$P$ without any balancing condition at $P$; meanwhile, it has other legs
terminating at the singularities of $B$. This is legal behaviour as
explained at the end of \S\ref{tropgeomsect}. Following the description
at the end of \S4, one can imagine disks over each leg terminating at
a singular point. Where these legs meet, one would like to glue these
disks together and continue along a cylinder over the segment adjacent
to $P$. Terminating at $P$, we roughly obtain a disk in $X(B)$
with boundary contained in the torus fibre over $P$, as depicted.
It is natural to ask how many ways the initial disks (possibly taking
multiple covers of these disks) can be glued together to give a new disk.

\begin{figure}
\input{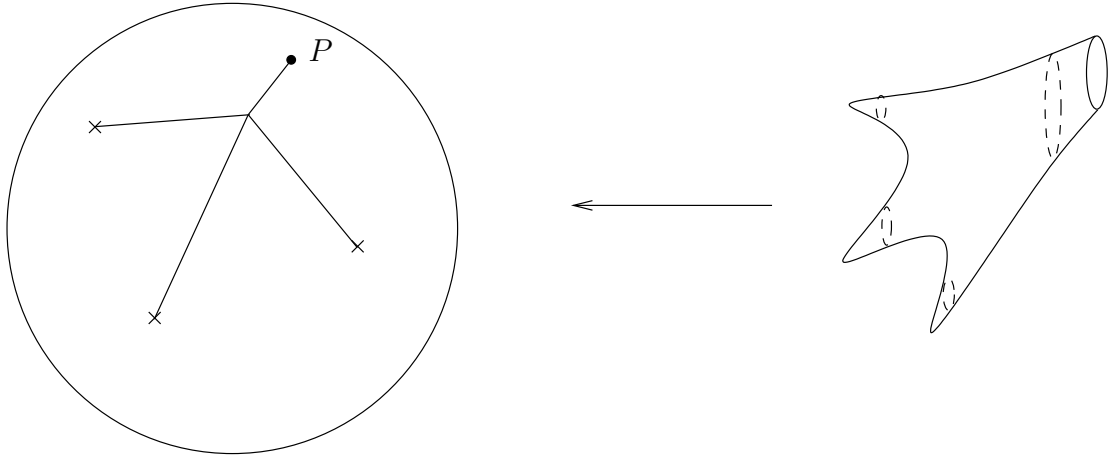}
\caption{A tropical disk on an affine K3 surface. Here the $\times$'s indicate
singular points, while the disk ``ends'' at the point $P$.}
\label{tropicaldisk}
\end{figure}

Now compare this picture with what we have seen on the mirror side. Our
explicit degeneration really gives, as generic fibre, something like
$\check X(B)$. However, it is controlled by similar tropical information:
rays emanate from the singularities in the monodromy invariant direction,
just as in the case of the tropical curves. They collide, and the
Kontsevich-Soibelman result in Proposition \ref{KSLemma} 
gives new rays. So one may hope that 
this process precisely reflects holomorphic disks in $X(B)$ with boundary
on fibres of $X(B)\rightarrow B$. 

It is also worth mentioning work of Auroux \cite{Aur}, which makes more precise
the notion that the complex structure on one side should be 
determined by holomorphic disks on the other. This also provides
a posteriori support for the idea that there must be an enumerative
interpretation for the process of generating new rays.

It is usually difficult to work with holomorphic disks. It is often 
easier to translate problems involving holomorphic disks into problems
involving genuine Gromov-Witten invariants. 
We can do so for the problems being discussed here. 
Here then is the enumerative interpretation, in the simplest
situation, as explained in \cite{GPS}.

Suppose we are given distinct non-zero primitive vectors $m_1,\ldots,m_p
\in M$ and positive integers $\ell_1,\ldots,\ell_p$.
Consider the scattering diagram
\[
\foD=\{(\RR m_i, (1+tz^{-m_i})^{\ell_i})\,|\,1\le i\le p\}.
\]
Let $(\fod,f_{\fod})\in\Scatter(\foD)\setminus\foD$. We can always
assume that this is the only ray in $\Scatter(\foD)\setminus\foD$ with
a given underlying ray $\fod$. This is because if there are rays
$(\fod_1,f_{\fod_1}),(\fod_2,f_{\fod_2}),\ldots$ in $\Scatter(\foD)
\setminus\foD$ with $\fod_1=\fod_2=\cdots$,
we can replace this collection of rays with a single ray
$(\fod_1,\prod f_{\fod_i})$ without affecting $\theta_{\gamma,\Scatter(\foD)}$.
With this assumption, $f_{\fod}$ is uniquely determined by $\foD$. We wish
to interpret $f_{\fod}$ enumeratively.

To do this, consider a complete fan $\Sigma$ in $M_{\RR}$ whose one-dimensional
rays are 
\[
\RR_{\le 0}m_1,\ldots,\RR_{\le 0}m_p, \fod.
\] 
Assume for the sake
of simplicity in this discussion that $\fod$ does not coincide with the
other $p$ rays.
Let $X$ be the toric variety defined by $\Sigma$, with toric divisors
$D_1,\ldots,D_p,D_{\out}$ corresponding to the above rays. Next, choose
$\ell_i$ general points
on the divisor $D_i$, say labelled $P_{i1},\ldots,P_{i\ell_i}$. 
Let $\nu:\widetilde X\rightarrow X$ be the blow-up of these
$\sum_i \ell_i$ points, with
exceptional divisor $E_{ij}$ over $P_{ij}$. Let $\widetilde D_i,\widetilde D_{\out}$
denote the proper transforms of $D_i,D_{\out}$.

In what follows, we will use the notation ${\bf P}_i=(p_{i1},\cdots,
p_{i\ell_i})$
for a partition of length $\ell_i$ of some non-negative integer $|{\bf P}_i|
=p_{i1}+\cdots+p_{i\ell_i}$, allowing
some of the $p_{ij}$'s to be zero. Fix a class $\beta\in H^2(X,\ZZ)$
with the property that $a_i:=\beta\cdot D_i$ are non-negative
and $k:=\beta\cdot D_{\out}$ is positive.
It is an easy exercise in toric geometry that this implies a relationship
\[
\sum_{i=1}^p a_i m_i=km_{\out},
\]
where $m_{\out}$ is a primitive generator of $\fod$. If one chooses
a collection of partitions ${\bf P}=({\bf P}_1,\ldots,{\bf P}_p)$
where ${\bf P}_i$ is a partition of $a_i$,  
let
\[
\beta_{\bf P}:=\nu^*\beta-\sum_{i=1}^p \sum_{j=1}^{\ell_i} p_{ij}E_{ij}.
\]
This can be thought of as the class of a curve on $X$ which passes through
the point $P_{ij}$ precisely $p_{ij}$ times.

We would now like to associate a number to this cohomology class.
This will be a Gromov-Witten count 
of one-pointed rational curves in $\widetilde X$ which
(1) represent the class
$\beta_{\bf P}$; (2) are tangent to $\widetilde D_{\out}$ at
the marked point with order $k$; and  (3) are otherwise disjoint from any of
the divisors $\widetilde D_i$. 
This is a relative Gromov-Witten invariant.
However, the classical theory of relative Gromov-Witten invariants
works relative to a 
smooth divisor, and of course the union of the boundary divisors
here is singular. One can instead
encode the above conditions using log Gromov-Witten theory.
At the time \cite{GPS} was written, log Gromov-Witten theory
was not yet available, and as a consequence, we used a technical work-around to
reduce to the classical theory. I give this description here since it
does not require knowing log Gromov-Witten theory.

One defines $\widetilde X^o:=\widetilde X\setminus \bigcup_{i=1}^p \widetilde D_i$.
One then considers the moduli space $\foM(\widetilde X^o/\widetilde D^o_{\out},
\beta_{\bf P})$ of relative stable maps of genus
zero with target space $\widetilde X^o$, relative to the divisor
$\widetilde D^o_{\out}=\widetilde D_{\out}\cap \widetilde X^o$. These curves
have one marked point with order of tangency $k$ with $\widetilde D_{\out}^o$.
The only problem is that the target space is non-proper, but one shows
this doesn't cause any problems because nevertheless the moduli space
is proper. One finds it is virtual dimension zero, and since it carries
a virtual fundamental class, we can define
\[
N_{\bf P}:=\int_{[\foM(\widetilde X^o/\widetilde D^o,\beta_{\bf P})]^{vir}} 1
\in\QQ.
\]
We can then state the enumerative result (\cite{GPS}):

\begin{theorem}
\label{mainGPStheorem}
We have
\[
\log f_{\fod}=\sum_{\beta}\sum_{\bf P} k(\beta) N_{\bf P}
t^{\sum_i |{\bf P}_i|} z^{-k(\beta)m_{\out}},
\]
where the sum is over all $\beta\in H^2(X,\ZZ)$ with $\beta\cdot D_i\ge 0$,
$k(\beta):=\beta\cdot D_{\out}>0$, and partitions ${\bf P}$ with 
$|{\bf P}_i|=\beta\cdot D_i$.
\end{theorem}

\begin{example}
Returning to Example \ref{scatdiagexample}, consider the function $f_{\fod}$
attached to the ray of slope $1$ for the cases $\ell=1,2$ and $3$. In each
case, the surface $X$ is $\PP^2$, with coordinate axes $D_1,D_2$ and $D_{\out}$.
Then $\widetilde X$ is obtained by blowing up $\ell$ points on each of $D_1$ and $D_2$.

Considering first the case of $\ell=1$, we note that for $\beta=dH$, the class
of a degree $d$ curve in $\PP^2$, the only relevant choice of ${\bf P}$
is ${\bf P}_1=d$, ${\bf P}_2=d$, and thus we have
\[
\beta_{\bf P}=d\nu^* H - d E_{11}-d E_{21}.
\]
This represents the class of a curve of degree $d$ passing through
the two blown-up points $d$ times each. It is easy to see that the only
choice for such a curve is a $d$-fold cover of a line passing through
the two points. Furthermore, this cover must be totally ramified over
$D_{\out}$ to guarantee the required order of tangency with $D_{\out}$.
This requires a virtual count, and the relevant localization calculations
are carried out in \cite{GPS}, giving a value of $N_{{\bf P}}=(-1)^{d+1}/d^2$.
Thus we get
\[
\log f_{\fod}=\sum_{d=1}^{\infty} d \left({(-1)^{d+1}\over d^2}\right) t^{2d}x^{-d}y^{-d}.
\]
Exponentiating one finds $f_{\fod}=1+t^2x^{-1}y^{-1}$, agreeing with Example 
\ref{scatdiagexample}. So here we are just counting the one line through two
points in $\PP^2$ along with certain multiple covers of this line.

Going to $\ell=2$, and $\beta=dH$, one finds four choices for the partition
in the case $d=1$,
${\bf P}=(1+0,1+0), (1+0,0+1), (0+1,1+0)$, and $(0+1,0+1)$. Each corresponds
to a choice of one point on each of $D_1$, $D_2$, and one has one line
through each of these pairs of points. Thus $N_{\bf P}=1$ for each choice
of such ${\bf P}$. As in the case $\ell=1$, each of these lines also
contributes to higher degree via multiple covers, with, say, ${\bf P}=(d+0,d+0)$
contributing $N_{\bf P}=(-1)^{d+1}/d^2$. For $d=2$, one sees there are no
curves for ${\bf P}=(2+0,1+1)$, say, as this would require a conic with 
a node on $D_1$ and tangent to $D_{\out}$; such does not exist. But with
${\bf P}=(1+1,1+1)$, we look at conics passing through all four points
and tangent to $D_{\out}$. It is very easy to see there are two such
conics. 

One can then check that the only other curves contributing are multiple
covers of one of the four lines or two conics. The multiple
cover contribution for conics is actually different than for lines,
because the order of tangency with $D_{\out}$ is different. It turns
out the correct contribution for a $d$-fold cover of a conic is $1/d^2$.
Hence we find
\[
\log f_{\fod}=4\sum_{d=1}^{\infty} d\left({(-1)^{d+1}\over d^2}\right)t^{2d}x^{-d}y^{-d}
+2\sum_{d=1}^{\infty} 2d\left(1\over d^2\right)t^{4d}x^{-2d}y^{-2d}
\]
and exponentiating we get
\[
f_{\fod}={(1+t^2xy)^4\over (1-t^4x^{-2}y^{-2})^4}=(1-t^2x^{-1}y^{-1})^{-4}.
\]

In the case that $\ell=3$, one expects $3\times 3=9$ lines, as there is
one line passing through each pair of choices of one point on $D_1$ and one
point on $D_2$. For conics, one has double covers of these lines, for a contribution
of $-9/4$, and $2\times 3\times 3=18$ conics. Here one needs to choose two
points on $D_1$ and two points on $D_2$, and then there are two conics
passing through these four points tangent to $D_{\out}$.

For cubics, there is the contribution of triple covers of lines, for a total
of $9/9$, and a number of contributions from plane cubics. It turns out
that for ${\bf P}=(1+1+1,1+1+1)$, $N_{\bf P}=18$. Note this gives
a count of nodal plane cubics passing through $6$ fixed points and
for which $D_{\out}$ is a tri-tangent. On the other hand, for 
${\bf P}=(1+2+0,1+1+1)$, $N_{\bf P}=3$. Note that there are a total
of $12$ partitions of this shape. This latter count represents nodal cubics
with the node at one of the chosen points, passing also through four
other chosen points, with $D_{\out}$ being tritangent. One concludes that 
\[
\log f_{\fod}=9 t^2x^{-1}y^{-1}+2(-9/4+18)t^4x^{-2}y^{-2}+3(9/9+54)t^6x^{-3}y^{-3}+
\cdots.
\]
A direct comparision with the value given in Example \ref{scatdiagexample}
gives agreement.
\end{example}

We end this section with brief additional 
motivation for Theorem \ref{mainGPStheorem} and a word about the proof. 

Suppose we have a piece of an integral affine
manifold as depicted in Figure \ref{affinemanpiece}. Here we imagine a
situation with two singular points in a surface, with local monodromy
around the singularities contained in the horizontal and vertical line
segments being $\begin{pmatrix}1&\ell_1\\ 0&1\end{pmatrix}$
and $\begin{pmatrix}1&\ell_2\\ 0&1\end{pmatrix}$ in suitably chosen
bases (different for each segment). This is a slightly more general 
situation than was considered in \S\ref{Bmodelsect}, where we only discussed
singularities with monodromy of the form 
$\begin{pmatrix}1&1\\ 0&1\end{pmatrix}$. Nevertheless, the techniques
of that section still apply, but the functions attached to the initial
rays emanating from the singularities towards the central vertex $v$
can be taken to be of the form $(1+x^{-1})^{\ell_1}$ and $(1+y^{-1})^{\ell_2}$.
This is roughly the shape of the examples discussed above. Applying the
scattering procedure would then produce a smoothing of $\check X_0(B,\P,s)$.
However, on the mirror side, we interpret $B$ as a dual intersection complex,
which means it should arise from a degeneration $\shX\rightarrow D$
where the central fibre $\shX_0$ has an irreducible component $Y_v$ isomorphic
to $\PP^2$ (corresponding to the vertex $v$). Furthermore, the total
space $\shX$ should have $\ell_1+\ell_2$ ordinary double points lying
on the toric boundary of $Y_v$. If one blows up the 
Weil divisor $Y_v$ inside of $\shX$, one obtains a small resolution
$\widetilde\shX\rightarrow \shX$
of these ordinary double points, and in particular, the proper transform
$\widetilde Y_v$ of $Y_v$ is the blow-up of $Y_v$ at the points
$Y_v\cap \Sing(\shX)$. This operation 
blows up $\ell_1$ points on one coordinate axis
of $Y_v$ and $\ell_2$ on the other. This is exactly the same surface
considered in Theorem \ref{mainGPStheorem}.

Now consider the kind of curves on $\tilde Y_v$ counted by Theorem 
\ref{mainGPStheorem}. 
These are curves in $\widetilde Y_v$ which only intersect the third coordinate
axis at one point. These can be viewed as curves in $\widetilde\shX_0$,
but not ones which deform to holomorphic curves in a general fibre
of the family $\widetilde\shX\rightarrow D$. Rather, roughly, we expect
such curves to deform to holomorphic disks, with the point of intersection
with the singular locus of $\widetilde\shX_0$ (i.e., the point of intersection
with the third axis of $\widetilde Y_v$) expanding into an $S^1$, giving
the boundary of the holomorphic disk. Approximately, we expect this boundary
to lie in a fibre of an SYZ fibration on a general fibre of the family
$\widetilde\shX\rightarrow D$. The homology class of this boundary inside
the fibre is determined by the order of tangency of the curve with the
third axis. 

This correspondence between the relative curves considered in Theorem
\ref{mainGPStheorem} is only a moral one; there is no proof yet that we
are really counting such holomorphic disks. However, this argument served
as the primary motivation for Theorem \ref{mainGPStheorem}.

Finally, as far as the proof is concerned, there are several steps. First,
we show that scattering diagrams can be deformed to look like a union
of tropical curves, and use a variant of Mikhalkin's fundamental curve-counting
results \cite{Mik} as developed by Nishinou and Siebert
\cite{NS} to show that scattering diagrams perform certain curve counts
on toric surfaces. This is then related to the Gromov-Witten counts
of the blown-up surfaces using Jun Li's gluing formula \cite{Li2}.

\begin{figure}
\input{affinemanpiece.pstex_t}
\caption{}
\label{affinemanpiece}
\end{figure}

\section{Other recent results and the future}

I will close with a brief discussion of applications and future
developments of the methods discussed here.

Recently a variant of the smoothing mechanism described here was
used by myself, Hacking and Keel \cite{GHK} to give a very general construction
of mirrors of pairs $(Y,D)$ where $Y$ is a rational surface and $D$
is an effective anti-canonical divisor forming a cycle of rational curves.
We make use of \cite{GPS} to write down what we call the \emph{canonical
scattering diagram}, which can be described entirely in terms of the
Gromov-Witten theory of the pair $(Y,D)$ (and more specifically, counts
of curves intersecting $D$ at only one point). This scattering diagram
determines the mirror family. However, there is an additional crucial
tool used to partially compactify the family constructed. This is necessary
because unlike the affine manifolds considered in this paper, the 
natural one to associate to the pair $(Y,D)$ has a singularity at a vertex
of the polyhedral decomposition. There is no local model for a smoothing
at this vertex, and as a consequence, one constructs families which are
``missing'' a point. To add this point back, one needs to be sure there
are enough functions on the family constructed, and it turns out
homological mirror symmetry suggests a natural way to construct such
functions. This can be done tropically, creating what we call \emph{theta
functions}. The same construction applied to the case of degenerating
abelian varieties indeed produces ordinary theta functions, and we anticipate
the functions we construct in these other contexts will be similarly
useful. See \cite{GStheta} for a survey of these ideas.

The construction of \cite{GHK} then also solves a problem which pre-dates
mirror symmetry. In particular, we prove a conjecture of
Looijenga concerning smoothability of cusp singularities.

Theta functions can be viewed as canonical bases for rings of functions
on an affine variety or spaces of sections of line bundles on projective
varieties. As such, they make contact with canonical bases in cluster
algebra theory, providing a framework for constructing canonical bases
of cluster algebras. 

We also expect that the techniques for surface pairs $(Y,D)$ will generalize.
Indeed, a mirror partner to any maximally unipotent normal crossings
degeneration of K3 surfaces can be constructed along similar lines,
in work in progress with Hacking, Keel and Siebert. The expectation is
that with an additional helping of log Gromov-Witten theory, one should
be able write down a general construction in all dimensions for
mirror partners to maximally unipotent degenerations of Calabi-Yau
manifolds.

There still remains the question of extracting enumerative information
from periods which provided the original excitement in mirror symmetry. Here
we showed how enumerative geometry can be reflected in the mirror, but
in a rather local way. We expect that it should be possible to carry out the
computation of period integrals to extract genus zero Gromov-Witten invariants
of the mirror, but some technical issues remain in this direction. Nevertheless,
the program of understanding mirror symmetry via degenerations, inspired
by the SYZ conjecture, seems to provide a powerful framework of thinking
about mirror symmetry inside the realm of algebraic and tropical geometry.


\begin{thebibliography}{cccccc}

\bibitem{AC} D.~Abromovich, Q.~Chen: \emph{Stable logarithmic maps to
Deligne--Faltings pairs II,} preprint, 2011.

\bibitem{Amari} S.~Amari: \emph{Differential-geometric methods in statistics},
Lecture Notes in Statistics, {\bf 28}, Springer-Verlag, 1985.

\bibitem{Clay} P.~Aspinwall, T.~Bridgeland, A.~Craw, M.~Douglas, 
M.~Gross, A.~Kapustin, G.~Moore, G.~Segal, B.~Szendroi, P.~Wilson, 
Dirichlet branes and mirror symmetry. 
Clay Mathematics Monographs, 4. American Mathematical Society, Providence, RI;
Clay Mathematics Institute, Cambridge, MA, 2009. x+681 pp.

\bibitem{Aur} D.~Auroux: \emph{Mirror symmetry and $T$-duality in the
complement of an anticanonical divisor}, J.\ G\"okova Geom.\ Topol.
GGT {\bf 1} (2007), 51--91.

\bibitem{Bat} V.~Batyrev: \emph{Dual polyhedra and mirror symmetry 
for Calabi-Yau hypersurfaces in toric varieties.}
J. Algebraic Geom.  {\bf 3}  (1994), 493--535.

\bibitem{BB} V.~Batyrev, and L.~Borisov: \emph{On Calabi-Yau complete
intersections in toric varieties,}
in {\it Higher-dimensional complex varieties (Trento, 1994)}, 39--65,
de Gruyter, Berlin, 1996.

\bibitem{Oren} O.~Ben-Bassat, \emph{Mirror symmetry and generalized
complex manifolds}, preprint, 2004, {\tt math.AG/0405303}.

\bibitem{CastMat} R.~Casta\~no-Bernard and D.~Matessi, \emph{Lagrangian
3-torus fibration}, J.\ Differential Geom., {\bf 81} (2009), 483--573.

\bibitem{CLS} P.~Candelas, M.~Lynker, R.~Schimmrigk,
\emph{Calabi-Yau manifolds in weighted ${\bf P}\sb 4$,}
Nuclear Phys.\ B {\bf 341} (1990), 383--402.

\bibitem{COGP} P.~Candelas, X.~de la Ossa, P.~Green, and L.~Parkes,
\emph{A pair of Calabi-Yau manifolds as an exactly soluble superconformal 
theory,}  Nuclear Phys. B  {\bf 359}  (1991), 21--74.

\bibitem{Chen} Q.~Chen: \emph{Stable logarithmic maps to Deligne--Faltings
pairs I,} preprint, 2010.

\bibitem{ChengYau} S.-Y.~Cheng and S.-T.~Yau, \emph{The real Monge-Amp\`ere 
equation and affine flat structures},
in \emph{Proceedings of the 1980 Beijing Symposium on Differential 
Geometry and Differential Equations}, Vol. 1, 2, 3 (Beijing, 1980),
339--370, Science Press, Beijing, 1982. 

\bibitem{Fukaya} K.~Fukaya, \emph{Multivalued Morse theory, 
asymptotic analysis and mirror symmetry,} in \emph{Graphs and patterns 
in mathematics and theoretical physics}, 205--278, 
Proc. Sympos. Pure Math., {\bf 73}, Amer. Math. Soc., Providence, RI, 2005.

\bibitem{GMN} D.~Gaiotto, G.~Moore, A.~Neitzke:
\emph{Four-dimensional wall-crossing via three-dimensional field theory,}
Comm.\ Math.\ Phys.\ {\bf 299} (2010), 163--224.
 
\bibitem{Givental} A.~Givental, \emph{Equivariant Gromov-Witten invariants,}
Internat. Math. Res. Notices {\bf 13}, (1996), 613--663. 

\bibitem{Gold} E.~Goldstein: \emph{A construction of new families of 
minimal Lagrangian submanifolds via torus actions,}  J. Differential Geom.  
{\bf 58} (2001),  233--261.

\bibitem{GrPl} B.~Greene and M.~Plesser, \emph{Duality in Calabi-Yau moduli 
space,} Nuclear Phys.\ B  {\bf 338}  (1990), 15--37.

\bibitem{SlagI} M.~Gross:
\emph{Special Lagrangian Fibrations I: Topology,}
in: {\sl Integrable Systems and Algebraic Geometry},
(M.-H.\ Saito, Y.\ Shimizu and K.\ Ueno eds.),
World Scientific 1998, 156--193.

\bibitem{SlagII} M.~Gross:
\emph{Special Lagrangian Fibrations II: Geometry,}
in: {\sl Surveys in Differential Geometry}, Somerville: MA,
International Press 1999, 341--403.

\bibitem{TMS} M.~Gross:
\emph{Topological Mirror Symmetry},
Invent.\ Math.\ {\bf 144} (2001), 75--137.

\bibitem{SLAGex} M.~Gross: \emph{Examples of special Lagrangian fibrations,}
in \emph{Symplectic geometry and mirror symmetry (Seoul, 2000)}, 81--109, 
World Sci. Publishing, River Edge, NJ, 2001.

\bibitem{GBB} M.~Gross:
\emph{Toric Degenerations and Batyrev-Borisov Duality},
Math. Ann. {\bf 333}, (2005) 645-688.

\bibitem{GAMS} M.~Gross: \emph{The Strominger-Yau-Zaslow conjecture: 
from torus fibrations to degenerations,}  
Algebraic geometry—Seattle 2005. Part 1,  149--192, Proc. Sympos. Pure Math., 
{\bf 80}, Part 1, Amer. Math. Soc., Providence, RI, 2009. 

\bibitem{GP2} M.~Gross: \emph{Mirror symmetry for $\PP^2$ and tropical
geometry}, Adv.\ Math., {\bf 224} (2010), 169--245.

\bibitem{Gbook} M.~Gross: \emph{Tropical geometry and mirror symmetry},
CBMS Regional Conference Series in Mathematics, {\bf 114}. American Mathematical
Society, Providencem RI, 2011. xvi+317 pp.

\bibitem{GHK} M.~Gross, P.~Hacking, S.~Keel: \emph{Mirror symmetry for log
Calabi-Yau surfaces I}, preprint, 2011.

\bibitem{GHKK} M.~Gross, P.~Hacking, S.~Keel, M.~Kontsevich: \emph{Mirror
symmetry and cluster algebras}, in preparation.

\bibitem{GPS} M.~Gross, R.~Pandharipande, B.~Siebert:
\emph{The tropical vertex}, Duke Math. J.  {\bf 153} (2010), 297--362. 

\bibitem{Announce} M.~Gross, and B.~Siebert:
\emph{Affine manifolds,  log structures, and mirror symmetry},
Turkish J.\ Math.\ {\bf 27} (2003), 33-60.

\bibitem{tori} M.~Gross, and B.~Siebert:
\emph{Torus fibrations and toric degenerations,}
in preparation.

\bibitem{PartI} M.~Gross, and B.~Siebert:
\emph{Mirror symmetry via logarithmic degeneration data I},
J. Diff. Geom., {\bf 72}, (2006) 169--338.

\bibitem{PartII} M.~Gross, and B.~Siebert:
\emph{Mirror symmetry via logarithmic degeneration data II},
J. Algebraic Geom., {\bf 19}, (2010) 679--780

\bibitem{Annals} M.~Gross and B.~Siebert: \emph{From affine
geoemtry to complex geometry}, \emph{Annals of Mathematics}, 
{\bf 174}, (2011), 1301-1428.

\bibitem{JAMS} M.~Gross and B.~Siebert: \emph{Logarithmic
Gromov-Witten invariants}, preprint, 2011, to appear in
J.\ of the AMS.

\bibitem{GStheta} M.~Gross and B.~Siebert: \emph{Theta functions and
mirror symmetry}, preprint, 2011.

\bibitem{GTZ} M.~Gross, V.~Tosatti, Y.~Zhang: \emph{Collapsing
of abelian fibred Calabi-Yau manifolds}, preprint, 2011. To appear in
Duke Math. J.

\bibitem{GrWiBV} M.~Gross, and P.M.H.~Wilson: \emph{Mirror symmetry via $3$-tori
for a class of Calabi-Yau threefolds,}  Math. Ann. {\bf 309} (1997),  505--531.

\bibitem{GrWi} M.~Gross, and P.M.H.~Wilson: \emph{Large complex structure 
limits of $K3$ surfaces,}  J. Differential Geom. {\bf 55} (2000), 475--546.

\bibitem{Gual} M.~Gualtieri, \emph{Generalized complex geometry},
Oxford University DPhil thesis, {\tt math.DG/0401221}.

\bibitem{HZ} C.~Haase, and I.~Zharkov:
\emph{Integral affine structures on spheres and torus fibrations of
Calabi-Yau toric hypersurfaces I},
preprint 2002, {\tt math.AG/0205321}.

\bibitem{HZ3} C.~Haase, and I.~Zharkov:
\emph{Integral affine structures on spheres III: complete intersections},
preprint, {\tt math.AG/0504181}.

\bibitem{Hit} N.~Hitchin: 
\emph{The Moduli Space of Special Lagrangian Submanifolds},
Ann.\ Scuola Norm.\ Sup.\ Pisa Cl.\ Sci.\ (4) {\bf 25} (1997), 503--515.

\bibitem{HitGen} N.~Hitchin: \emph{Generalized Calabi-Yau manifolds},
Q. J. Math. {\bf 54}  (2003),  281--308. 

\bibitem{Huy} D.~Huybrechts: \emph{Generalized Calabi-Yau structures, 
$K3$ surfaces, and $B$-fields,} Internat. J. Math. {\bf 16}  (2005),  13--36.

\bibitem{Illu}
L.\ Illusie: \emph{Logarithmic spaces (according to K. Kato)},
in {\sl Barsotti Symposium in Algebraic Geometry} (Abano Terme 1991),
183--203, Perspect.\ Math.\ 15, Academic Press 1994.

\bibitem{IP1} E.~Ionel, T.~Parker, \emph{Relative Gromov-Witten invariants,}
Ann.\ of Math.\ (2) {\bf 157} (2003),  45--96.

\bibitem{IP2} E.~Ionel, T.~Parker, \emph{The symplectic sum formula for 
Gromov-Witten invariants},  Ann.\ of Math.\ (2)  {\bf 159}  (2004), 
935--1025.

\bibitem{Joyce} D.~Joyce, \emph{Singularities of special Lagrangian 
fibrations and the SYZ conjecture,}  Comm. Anal. Geom. {\bf 11}  (2003),
859--907.

\bibitem{F.Kato} F.~Kato:
\emph{Log smooth deformation theory},
Tohoku Math.\ J.\ {\bf 48} (1996), 317--354.

\bibitem{K.Kato} K.~Kato:
\emph{Logarithmic structures of Fontaine--Illusie},
in: {\sl Algebraic analysis, geometry, and number theory}
(J.-I.~Igusa et.~al.\ eds.), 191--224,
Johns Hopkins Univ.~Press, Baltimore, 1989.

\bibitem{KN} Y.~Kawamata, Y.~Namikawa:
\emph{Logarithmic deformations of normal crossing varieties
and smooothing of degenerate Calabi-Yau varieties},
Invent.\ Math.\ {\bf 118} (1994), 395--409.

\bibitem{KHMS} M.~Kontsevich: \emph{Homological algebra of mirror symmetry},
Proceedings of the International Congress of Mathematicians,
Vol.\ 1, 2 (Z\"urich, 1994), 120--139,
Birkh\"auser, Basel, 1995.

\bibitem{KS} M.~Kontsevich, and Y.~Soibelman: 
\emph{Homological mirror symmetry and torus fibrations},
in: {\sl Symplectic geometry and mirror symmetry} (Seoul, 2000), 203--263, 
World Sci.\ Publishing, River Edge, NJ, 2001. 

\bibitem{KS2} M.~Kontsevich, and Y.~Soibelman: 
\emph{Affine structures and non-archimedean analytic spaces},
The unity of mathematics,  321--385, 
Progr.\ Math., 244, Birkh\"auser Boston, Boston, MA, 2006. 

\bibitem{Leung} N.C.~Leung:
\emph{Mirror symmetry without corrections},
Comm.\ Anal.\ Geom.\ {\bf 13} (2005),  287--331.

\bibitem{LV} N.C.~Leung, C.~Vafa: \emph{Branes and toric geometry},
Adv.\ Theor.\ Math.\ Phys.\ {\bf 2} (1998), 91--118.

\bibitem{Li2} J.~Li: \emph{A degeneration formula of GW-invariants,}
J.\ Differential Geom.\ {\bf 60}  (2002), 199--293.

\bibitem{LR} A.-M.~Li, Y.~Ruan: \emph{Symplectic surgery and Gromov-Witten
invariants of Calabi-Yau 3-folds}, Invent.\ Math., {\bf 145} (2001), 151--218.

\bibitem{LLY} B.~Lian, K.~Liu, S-T.~Yau,
\emph{Mirror principle. I,} Asian J. Math. {\bf 1} (1997), 729--763. 

\bibitem{Matessi} D.~Matessi, \emph{Some families of special Lagrangian tori,}
Math. Ann. {\bf 325} (2003), 211--228.

\bibitem{McLean} R.~McLean, \emph{Deformations of calibrated submanifolds,}
Comm. Anal. Geom. {\bf 6} (1998), 705--747.

\bibitem{Mik} G.~Mikhalkin, \emph{Enumerative tropical algebraic
 geometry in $\RR^2$,}
J. Amer. Math. Soc. {\bf 18} (2005), 313--377.

\bibitem{Morr} D.~Morrison, \emph{Compactifications of moduli spaces
inspired by mirror symmetry,}
in \emph{Journ\'ees de G\'eom\'etrie Alg\'ebrique d'Orsay (Orsay, 1992)},
Astérisque {\bf 218} (1993), 243--271.

\bibitem{NS} T.~Nishinou, B.~Siebert: \emph{Toric degenerations
of toric varieties and tropical curves}, Duke Math.\ J., {\bf 135} (2006),
1--51.

\bibitem{Ols} M.~Olsson: \emph{Log algebraic stacks and moduli of log schemes,}
Ph.D.\ thesis, Berkeley.

\bibitem{Park1} B.~Parker: \emph{Exploded manifolds}, Adv.\ in Math., 
{\bf 229} (2012), 3256--3319.

\bibitem{Park2} B.~Parker: \emph{Gromov-Witten invariants of exploded
manifolds}, preprint, 2011.

\bibitem{Petersen} P.~Petersen: \emph{Riemannian geometry,}
Graduate Texts in Mathematics, 171. Springer-Verlag, New York, 1998.

\bibitem{Rei} M.~Reineke: \emph{Poisson automorphisms and quiver
moduli}, J.\ Inst.\ Math.\ Jussieu {\bf 9}, (2010), 653--667.

\bibitem{Ruan} W.-D.~Ruan:
\emph{Lagrangian torus fibration and mirror symmetry of Calabi-Yau
hypersurface in toric variety},
preprint 2000, math.DG/0007028.

\bibitem{RuanJSG} W.-D.~Ruan: \emph{Lagrangian torus fibration of quintic 
Calabi-Yau hypersurfaces. II. Technical results on gradient flow construction,}
J. Symplectic Geom.  {\bf 1}  (2002),  no. 3, 435--521.

\bibitem{ssKod} S.\ Schr\"oer, B.\ Siebert: \emph{Irreducible 
degenerations of primary Kodaira surfaces,} in Complex geometry 
(G\"ottingen, 2000),  193--222, Springer, Berlin, 2002.

\bibitem{ss} S.\ Schr\"oer, B.\ Siebert:
\emph{Toroidal crossings and logarithmic structures},
Adv.\ Math.\ {\bf 202} (2006), 189--231. 


\bibitem{STalk} B.~Siebert: \emph{Gromov-Witten invariants in relative and
singular cases.} Lecture given at the workshop ``Algebraic aspects of
mirror symmetry,'' Univ.\ Kaiserslautern, Germany, June 2001.

\bibitem{SYZ} A.\ Strominger, S.-T.\ Yau, and E.~Zaslow, \emph{Mirror Symmetry
is $T$-duality,} Nucl.\ Phys.\ {\bf B479}, (1996) 243--259.

\bibitem{Tos} V.~Tosatti, \emph{Adiabatic limits of Ricci-flat
K\"ahler metrics,} J.\ Differential Geom., {\bf 84} (2010), 427--453.

\bibitem{Yau} S.-T.~Yau: \emph{On the Ricci curvature of a compact
K\"ahler manifold and the complex Monge-Amp\`ere equation. I,}
Comm.\ Pure Appl.\ Math., {\bf 31} (1978), 339--411.

\bibitem{Zhang} Y.~Zhang: \emph{Collapsing of Calabi-Yau manifolds and
special Lagrangian submanifolds,} preprint, 2009.
\end{thebibliography}
\end{document}